\documentclass[reqno, final,a4paper]{amsart}
\usepackage{color}
\usepackage[colorlinks=true,allcolors=blue
]{hyperref}
\usepackage{amsmath, amssymb, amsthm}
\usepackage{mathrsfs}
\usepackage{mathtools}
\usepackage[noabbrev,capitalize,nameinlink]{cleveref}
\crefname{equation}{}{}
\usepackage{fullpage}
\usepackage[noadjust,nocompress]{cite}
\usepackage{graphics}
\usepackage{pifont}
\usepackage{tikz}
\usepackage{bbm}
\usepackage[T1]{fontenc}
\usetikzlibrary{arrows.meta}

\usepackage{environ}
\usepackage{framed}
\usepackage{url}
\usepackage[linesnumbered,ruled,vlined]{algorithm2e}
\usepackage[noend]{algpseudocode}
\usepackage[labelfont=bf]{caption}
\usepackage{framed}
\usepackage[framemethod=tikz]{mdframed}
\usepackage{pgfplots}
\usepackage{appendix}
\usepackage{graphicx}
\usepackage[textsize=tiny]{todonotes}
\usepackage{tcolorbox}
\usepackage{xcolor}
\usepackage{flafter}
\usepackage[shortlabels]{enumitem}
\usepackage{subcaption}
\crefformat{enumi}{#2#1#3}
\crefrangeformat{enumi}{#3#1#4 to~#5#2#6}
\crefmultiformat{enumi}{#2#1#3}
{ and~#2#1#3}{, #2#1#3}{ and~#2#1#3}
\allowdisplaybreaks

\usepackage{nicematrix}   
\usetikzlibrary{calc} 
\usepackage{tikz}
\usetikzlibrary{fit}

\let\originalleft\left
\let\originalright\right
\renewcommand{\left}{\mathopen{}\mathclose\bgroup\originalleft}
\renewcommand{\right}{\aftergroup\egroup\originalright}

\newcommand*{\claimproofname}{Proof of claim}
\newenvironment{claimproof}[1][\claimproofname]{\begin{proof}[#1]}{\end{proof}}

\crefname{algocf}{Algorithm}{Algorithms}

\crefname{equation}{}{}
 
\AtBeginEnvironment{appendices}{\crefalias{section}{appendix}} 

\usepackage[color]{showkeys}

\colorlet{refkey}{orange!20}
\colorlet{labelkey}{blue!30}

\crefname{algocf}{Algorithm}{Algorithms}

\numberwithin{equation}{section}

\newtheorem{theorem}{Theorem}[section]

\newtheorem{lemma}[theorem]{Lemma}
\newtheorem{claim}[theorem]{Claim}

\newtheorem*{claim*}{Claim}

\theoremstyle{definition}
\newtheorem{fact}[theorem]{Fact}
\newtheorem{definition}[theorem]{Definition}

\newtheorem*{definition*}{Definition}

\theoremstyle{remark}
\newtheorem{remark}[theorem]{Remark}

\crefname{theorem}{Theorem}{Theorems}

\crefname{remark}{Remark}{Remarks}
\creflabelformat{remark}{#2~#1#3}

\crefname{fact}{Fact}{Facts}
\creflabelformat{fact}{#2~#1#3}

\crefname{claim}{Claim}{Claims}
\creflabelformat{claim}{#2~#1#3}

\crefname{lemma}{Lemma}{Lemmas}
\creflabelformat{lemma}{#2~#1#3}

\crefname{definition}{Definition}{Definitions}
\creflabelformat{definition}{#2~#1#3}

\crefname{conjecture}{Conjecture}{Conjectures}
\creflabelformat{conjecture}{#2~#1#3}

\crefname{corollary}{Corollary}{Corrolaries}
\creflabelformat{corollary}{#2~#1#3}

\newcommand{\mb}{\mathbb}
\newcommand{\mbf}{\mathbf}

\newcommand{\mc}{\mathcal}

\newcommand{\mr}{\mathrm}

\newcommand{\on}{\operatorname}

\newcommand{\pvec}[1]{\vec{#1}\mkern2mu\vphantom{#1}'}

\newcommand{\cA}{\mathcal{A}}

\newcommand{\cC}{\mathcal{C}}

\newcommand{\cF}{\mathcal{F}}

\newcommand{\cH}{\mathcal{H}}
\newcommand{\cI}{\mathcal{I}}
\newcommand{\cJ}{\mathcal{J}}

\newcommand{\cL}{\mathcal{L}}

\newcommand{\cP}{\mathcal{P}}
\newcommand{\cQ}{\mathcal{Q}}
\newcommand{\cR}{\mathcal{R}}
\newcommand{\cS}{\mathcal{S}}
\newcommand{\cT}{\mathcal{T}}
\newcommand{\cU}{\mathcal{U}}
\newcommand{\cV}{\mathcal{V}}

\newcommand{\cX}{\mathcal{X}}
\newcommand{\cY}{\mathcal{Y}}

\newcommand{\sC}{\mathscr{C}}

\newcommand{\sR}{\mathscr{R}}
\newcommand{\sS}{\mathscr{S}}

\newcommand{\bfB}{\mathbf{B}}

\newcommand{\bfG}{\mathbf{G}}
\newcommand{\bfH}{\mathbf{H}}

\newcommand{\bfL}{\mathbf{L}}

\newcommand{\bfP}{\mathbf{P}}

\newcommand{\bfR}{\mathbf{R}}
\newcommand{\bfS}{\mathbf{S}}
\newcommand{\bfT}{\mathbf{T}}

\newcommand{\bfX}{\mathbf{X}}

\newcommand{\bfZ}{\mathbf{Z}}

\newcommand{\bfz}{\mathbf{z}}
\newcommand{\bflambda}{\boldsymbol{\lambda}}
\newcommand{\bfmu}{\boldsymbol{\mu}}
\newcommand{\bfnu}{\boldsymbol{\nu}}

\DeclareMathOperator{\Li}{Li}
\DeclareMathOperator{\supp}{supp}

\newcommand{\unif}{\operatorname{Unif}}
\newcommand{\trp}{\operatorname{TRP}}
\newcommand{\dens}{\operatorname{dens}}
\newcommand{\Tbad}{\operatorname{\cT^{\textnormal{bad}}}}
\newcommand{\Ubad}{\cU^{\textnormal{bad}}}

\renewcommand{\Pr}{\mb P}
\newcommand{\E}{\mb E}

\allowdisplaybreaks
\title{On the probability that a random Latin square is all-even}

\author[Greenhill]{Catherine Greenhill}
\address{School of Mathematics and Statistics, UNSW Sydney, NSW 2052, Australia}
\email{c.greenhill@unsw.edu.au}

\author[Kopfova]{Lenka Kopfov\'a}
\address{Institute of Science and Technology Austria (ISTA). Am Campus 1, 3400 Klosterneuburg, Austria}
\email{lenka.kopfova@ist.ac.at}

\author[Kwan]{Matthew Kwan}
\address{Institute of Science and Technology Austria (ISTA). Am Campus 1, 3400 Klosterneuburg, Austria}
\email{matthew.kwan@ist.ac.at}

\author[Petrova]{Kalina Petrova}
\address{Institute of Science and Technology Austria (ISTA). Am Campus 1, 3400 Klosterneuburg, Austria}
\email{kalina.petrova@ist.ac.at}

\thanks{CG was supported by ARC Discovery Project DP250101611.
LK, MK and KP were supported in part by ERC Starting Grant ``RANDSTRUCT'' No.~101076777. KP was supported by the European Union’s Horizon 2020 research and innovation programme
under the Marie Skłodowska-Curie grant agreement No.~101034413 \includegraphics[width=4mm]{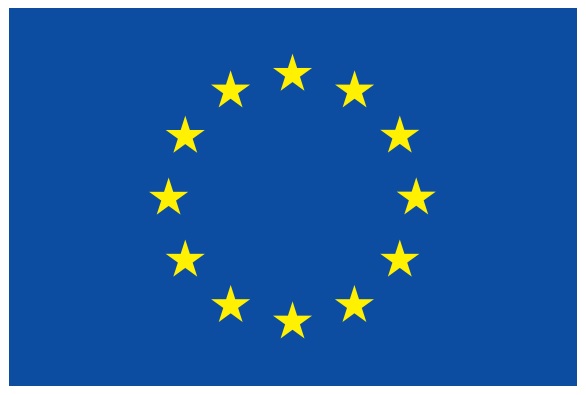}.
}
\begin{document}

\begin{abstract}
In a Latin square, every row can be interpreted as a permutation, and therefore has a parity (even or odd). We prove that in a uniformly random $n\times n$ Latin square, the probability that all rows are even is of the form $2^{-n+o(n)}$. This fills a gap left by recent work of Kwan, Petrova and Sawhney: they proved that the row parities behave like a sequence of independent coin flips in a variety of different senses, but they were not able to lower-bound the all-even probability, as their approach had a small number of ``exceptional rows'' whose parities they could not control. 

In fact, we prove a more general theorem on the probability of any possible prescription of row parities, column parities and symbol parities. Our approach extends the ideas of Kwan, Petrova and Sawhney, introducing a new way to ``bootstrap'' intercalate switchings to control more complex switching operations.
\end{abstract}

\maketitle

\section{Introduction}

A \emph{Latin square} of order $n$ is an $n\times n$ array filled with $n$ distinct symbols, such that each symbol appears exactly once in each row and each column. The study of Latin squares is one of the oldest areas of combinatorics; see \cite{KD15} for a broad introduction to the subject.

A natural source of examples of Latin squares is multiplication tables of groups. One may therefore imagine that Latin squares are necessarily very ``symmetrical'' objects with highly restricted structure, but an increasing body of evidence has shown that Latin squares can have almost arbitrary structure. In fact, it seems that one can quite effectively model a \emph{uniformly random} order-$n$ Latin square in one of several quite na\"ive ways: for example, we could imagine that the $n$ rows are each independent uniformly random order-$n$ permutations, or even more crudely we could imagine that each row/column/symbol triple is present with probability $1/n$ independently. However, it turns out to be exceedingly difficult to rigorously justify these types of heuristics.

Perhaps the first investigation along these lines concerned \emph{parities} in random Latin squares: every row of a Latin square is a permutation, which can be even or odd. These parities play quite an important role in the study of Latin squares, most notably in connection with the so-called \emph{Alon--Tarsi conjecture} (see e.g.\ \cite{FM19}). A random permutation is odd with probability $1/2$, so heuristically we might imagine that for a random Latin square, the parities of the rows behave like a sequence of independent coin flips. In particular, we might imagine that the probability all rows are even is about $2^{-n}$. A conjecture to this effect was made by Peter Cameron in 1993 (see \cite{Cam01}), and the first result in this direction was proved shortly after by H\"aggkvist and Janssen~\cite{HJ96}: namely, they showed that the all-even probability is at most $(\sqrt{3/4} + o(1))^n \approx 0.87^n$. Recently, Kwan, Petrova and Sawhney~\cite{kwan2025parities} showed that the parities of the rows do behave like a sequence of independent coin flips, in a variety of different senses; in particular, their work implies that the all-even probability is at most $(1/2+o(1))^n$. See the introduction of \cite{kwan2025parities} for much more discussion of the history of Cameron's conjecture and its surrounding context.

However, until now there has been no nontrivial \emph{lower bound} on the all-even probability (for good reasons, as we will discuss shortly). In this paper we fill this gap, proving the following theorem.

\begin{theorem}\label{thm:all-even}
Let $\bfL$ be a random Latin square, sampled uniformly from all order-$n$ Latin squares with symbols $1,\dots,n$. Then the probability that all rows of $\bfL$ are even permutations is $(1/2+o(1))^n$.
\end{theorem}

Actually, \cref{thm:all-even} is a special case of a more general theorem (\cref{thm:main}), which says that for any specification of the parities of the row permutations, the parities of the column permutations, and the parities of the ``symbol permutations'' (satisfying a certain necessary identity, see \cref{eq:subspace}), the probability that these parities arise is $(1/2+o(1))^{3n}$. Before this work, it was not clear that for all such parity specifications a Latin square actually exists!

We will outline our proof approach in some detail in \cref{subsec:outline}, but it is worth saying a few words here about why the lower bound is harder than the upper bound (and more generally, why it is so hard to prove almost anything about random Latin squares). The most obvious way to study random Latin squares is by direct enumeration (i.e., we might hope to estimate the number of all-even order-$n$ Latin squares and divide by the total number of order-$n$ Latin squares). Unfortunately, enumerating Latin squares is a notoriously difficult task, and state-of-the-art bounds (via so-called \emph{permanent estimates}) are nowhere near good enough for this. Instead, all work on parities in random Latin squares has proceeded via \emph{switching} methods, which study certain parity-changing operations to compare the numbers of Latin squares with each parity (without actually needing to know the total number of Latin squares). 

In \cite{kwan2025parities}, the approach was to first use crude enumeration estimates to establish that random Latin squares typically have a rich constellation of \emph{intercalates} ($2\times 2$ Latin subsquares), which permit very useful parity-changing operations. Unfortunately, this type of enumeration-based approach has fundamental quantitative limitations that prevent us from ruling out the possibility that there are a small number of rows which do not interact with any intercalates at all, and we do not have any control over the parities of these rows. This is fine for an upper bound, if we tolerate a $2^{o(n)}$ error, but it is unacceptable for a lower bound. To overcome this, we introduce a way to use intercalate switchings to ``bootstrap'' to more complex switching operations that can be used to control the parities of \emph{all} the rows.

\subsection{Main theorem}
For a Latin square $L$, let $\vec \xi_{\mr{row}}( L)\in \mb F_2^n$ be the sequence of parities of the row permutations 
(where $0$ stands for even and $1$ stands for odd)\footnote{Throughout this paper we identify $\mb F_2$ with $\mb Z/2\mb Z$ in the obvious way.}, and let $\vec \xi_{\mr{col}}(L)\in \mb F_2^n$
be the sequence of parities of the column permutations. We also define a third sequence $\vec \xi_{\mr{sym}}( L)\in \mb F_2^n$ describing the sequence of parities of the \emph{symbol permutations}. Perhaps
the most natural way to define a symbol permutation is to reinterpret
a Latin square in a way that emphasises the natural symmetry between
rows, columns, and symbols: indeed, observe that an $n\times n$ Latin
square (with symbols $1,\dots,n$) can be interpreted as an $n\times n\times n$
array in which every entry is 0 or 1, such that every
two-dimensional ``slice'' of this array is a permutation matrix\footnote{Specifically, if the Latin square has symbol $s$ in
the $(i,j)$-entry, then the $(i,j,s)$-entry
of the corresponding $n\times n\times n$ array is set to 1.}. So, the row and column permutations correspond to slices in two
of the three possible directions, and the symbol permutations correspond
to slices in the third direction.\footnote{Note that each such permutation matrix could be interpreted as two different permutations which are inverses of each other. Since a permutation and its inverse have the same parity, this choice does not affect our result and proofs. Still, for concreteness, we think of row permutations as mapping columns to symbols; column permutations as mapping rows to symbols; and symbol permutations as mapping rows to columns.}

The sequences $\vec \xi_{\mr{row}}( L),\vec \xi_{\mr{col}}( L),\vec \xi_{\mr{sym}}( L)$ cannot vary completely independently: it turns out that they must lie in a certain affine subspace. Indeed, writing $|\vec \xi|_1$ for the number of 1-entries in a sequence $\vec \xi\in \mb F_2^n$, it was proved by Janssen~\cite{Jan95} and Zappa~\cite{Zap96} that for any Latin square $L$ of order $n$ we have
\begin{equation}
    |\vec \xi_{\textnormal{row}}(L)|_1+|\vec \xi_{\textnormal{col}}(L)|_1+|\vec \xi_{\textnormal{sym}}(L)|_1 \equiv f(n) \pmod 2,\label{eq:subspace}
\end{equation}
where 
\[
f(n)=
\begin{cases}
0 & \text{if }n\equiv0 \text{ or }1\pmod 4\\
1 & \text{if }n\equiv2 \text{ or }3\pmod 4.
\end{cases}
\]
Our main result is the following theorem, which shows that all possibilities for $(\vec\xi_{\textnormal{row}}(L),\vec\xi_{\textnormal{col}}(L),\vec\xi_{\textnormal{sym}}(L))$, satisfying \cref{eq:subspace}, are roughly equally likely to occur.
\begin{theorem}
\label{thm:main}
Let $\bfL$ be a uniformly random $n\times n$ Latin square. Then for any $\vec r,\vec c,\vec s\in \mb F_2^n$ satisfying $|\vec r|_1+|\vec c|_1+|\vec s|_1\equiv f(n) \pmod 2$ we have
\[\mathbb{P}\big[(\vec\xi_{\textnormal{row}}(\mbf L),\vec\xi_{\textnormal{col}}(\mbf L),\vec\xi_{\textnormal{sym}}(\mbf L))=(\vec r,\vec c, \vec s)\big]= 2^{-3n+o(n)}.\]
\end{theorem}

\subsection{Further directions}The most obvious open question related to this work is whether the relative error  $2^{o(n)}$ can be improved to $1+o(1)$: is it true that in the setting of \cref{thm:main} we have $\mathbb{P}\big[(\vec\xi_{\textnormal{row}}(\mbf L),\vec\xi_{\textnormal{col}}(\mbf L),\vec\xi_{\textnormal{sym}}(\mbf L))=(\vec r,\vec c, \vec s)\big]= (1+o(1))2^{-3n}$?

We note that some kind of error term is (presumably) necessary: if there were some even $n$ for which $(\vec\xi_{\textnormal{row}}(\mbf L),\vec\xi_{\textnormal{col}}(\mbf L),\vec\xi_{\textnormal{sym}}(\mbf L))$ was exactly uniform on the set of all $(\vec r,\vec c,\vec s)$ satisfying \cref{eq:subspace}, this would falsify the so-called \emph{Alon--Tarsi} conjecture, which is widely believed to be true (see the survey in \cite{FM19}).

More generally, there are a lot of open problems about random Latin squares that until now have seemed quite intractable, as there have been very few general-purpose methods in the area. The paper \cite{kwan2025parities} (which first showed how to use intercalates to provide robust switching operations) was an important step towards a more complete theory of random Latin squares, and the ideas in this paper (showing how to use intercalate switchings to control more complex switching operations) constitute an additional step in this direction. We believe these ideas can be taken much further. For example, in upcoming work of the second, third and fourth authors with Richard Montgomery, we build on this approach to estimate the probability that a given set of entries is present in a random Latin square, which for example allows us to determine the asymptotic distribution of the number of $3\times 3$ subsquares. Interested readers may contact us if they wish to see a draft of our upcoming work.

\subsection{Proof ideas}\label{subsec:outline}
Before proceeding to the details of the proof, we take a moment to sketch some of the proof ideas. For simplicity, we restrict our attention to the problem of lower-bounding the probability that all rows are even. That is, we sketch a proof of \cref{thm:all-even}; the proof of \cref{thm:main} uses these same ideas, but is a bit more complicated, as one has to be careful that certain different types of switching operations don't interfere with each other.

\subsubsection{Previous work: intercalate switchings}
An \emph{intercalate} in a Latin square is a $2\times 2$ Latin subsquare. Intercalates permit a certain type of ``switching'' operation: we can exchange the two symbols of an intercalate to obtain a slightly different Latin square. Intercalates are very useful for controlling parity, as switching an intercalate changes the parities of the two rows it intersects (for each of these two rows, the switch acts like a transposition). In \cite{kwan2025parities}, it is shown that random Latin squares typically have a collection of special types of disjoint intercalates that are well-spread
throughout the Latin square and behave kind of `independently'. Such intercalates are called \emph{stable}. Roughly speaking, given a prescription of desired parities for the
rows that are controlled by these stable intercalates, one can switch some stable intercalates to obtain the prescribed parities. Although we expect there
to be stable intercalates all over the place, the current technology for studying random Latin squares is
not strong enough to rule out the possibility that some rows participate in no intercalates at all (this is related to the difficulty of enumerating Latin squares, as we will discuss further in \cref{subsubsec:approximation}). So, we are forced to deal with the possibility that there are some \emph{stubborn} rows that do not participate richly enough in our stable intercalates, and thus their parity cannot be controlled by stable intercalate switchings.
The main difficulty is thus controlling the parities of the stubborn rows.

\subsubsection{Short even-length cycle switchings, via intercalate switchings}
Intercalate switchings are a special case of a more general type of switching operation called \emph{(row) cycle switchings}. Specifically, given two rows $r_1,r_2$ in a Latin square $L$, we define $\sigma_L(r_1,r_2)$ to be the permutation that rearranges the entries of $r_1$ into the entries of $r_2$, and consider its cycle structure. The basis of cycle switching is the following simple observation: if we choose one cycle in this permutation, then swapping the entries of $r_1$ and $r_2$ in columns corresponding to this cycle maintains the Latin property. We refer to this operation as \emph{switching a row cycle}. If the cycle has even length, then this switching changes the parities of $r_1$ and $r_2$. (An intercalate switching is the special case where the cycle has length 2).

So, although we may have a few stubborn rows with no intercalates, we may hope that we can arrange for these stubborn rows to participate in short even-length cycles (they need to be short not to interfere too much with our collection of stable intercalates, as we will still need to use intercalate switchings to handle the non-stubborn rows). The key new idea is to use intercalate switchings in the non-stubborn rows in an \emph{indirect} way, to control the presence of short even-length cycles intersecting the stubborn rows. These short even-length cycles can then be switched to control the parities of the stubborn rows, after using stable intercalates to control the parities of all the non-stubborn rows.

To say a bit more about the details of the intercalate switching argument: our goal is to show that in a random Latin square (which may have some stubborn rows), we can typically switch a small number of intercalates to obtain a Latin square in which all odd stubborn rows intersect disjoint short even-length row cycles. A double-counting argument (together with a stable intercalate switching argument, for the $n-o(n)$ non-stubborn rows) will then allow us to deduce that a random Latin square is all-even with probability at least $(1/2)^{n-o(n)}$.

Let $s$ be some stubborn row (and suppose it is odd) in a Latin square $L$. For a row $r$, recall that $\sigma_L(r, s)$ is the permutation that rearranges the entries of row $r$ into the entries of $s$. We would like there to be a short even-length cycle in $\sigma_L(r, s)$. However, it may be the case that all cycles are long, or all cycles are short and of odd length. If there is a long cycle, the idea is to switch an intercalate that intersects the row $r$ in the entries belonging to this long cycle. This will have the effect of splitting the long cycle into two parts. On the other hand, if all cycles are short and of odd length, we can use an intercalate switch to merge two cycles. We will show that in a random Latin square, there is typically some choice of $r$ for which intercalates with the desired properties exist. (In fact, we need many such choices for $r$, so that we can pick a distinct such $r$ for each stubborn row.)

\subsubsection{Approximation with binomial random hypergraphs}\label{subsubsec:approximation}
In order to actually show that the desired intercalates are likely to exist, an important tool in our proof is an improved approximation lemma comparing random Latin squares with binomial random hypergraphs.

To motivate this, note that there is an alternative, more symmetrical view of Latin squares in terms of hypergraphs. We can represent a Latin square $L$ as a 3-partite 3-uniform hypergraph with parts $\cR,\cC,\cS$ of size $n$  corresponding to the rows, columns, and symbols, respectively. A hyperedge $(r,c,s)\in \mc R\times \mc C\times \mc S$ is present if and only if symbol $s$ appears in the cell determined by row $r$ and column $c.$ Thus, a Latin square $L$ can always be viewed as a subgraph of the complete 3-partite 3-uniform hypergraph $K^{(3)}_{n,n,n}$. As each vertex is contained in exactly $n$ hyperedges, this hypergraph contains a $(1/n)$-fraction of the hyperedges of $K^{(3)}_{n,n,n}$. Hence, as a more tractable heuristic model of a random Latin square, we can consider a random subset of the hyperedges of $K^{(3)}_{n,n,n}$ where each hyperedge is taken independently at random with probability $1/n$.

It turns out that one can use estimates on the number of $n\times n$ Latin squares, together with an analysis of the so-called \emph{triangle removal process}, to relate the likelihoods of certain events in the above binomial hypergraph model, to the likelihoods of corresponding events for random Latin squares. The first ``approximation lemma'' of this type appeared in \cite{kwan2020almost} (see also \cite{KSS21}), and a more refined version of this lemma recently appeared in \cite{kwan2025parities}.
Since there is a rather enormous gap between the state-of-the-art lower bounds and upper bounds on the number of $n\times n$ Latin squares, this method is only suitable for studying events that hold with overwhelmingly high probability. Fortunately, this includes various events about intercalates~\cite{KSSS22,kwan2025parities} (as long as we consider intercalates ``in aggregate'' over a large part of the Latin square; e.g.\ we cannot say anything about the intercalates intersecting a particular row).

In this paper, to prove our desired properties about intercalates, we prove a generalization of the approximation lemma in \cite{kwan2025parities} that allows us to study a uniformly random completion of a fixed Latin rectangle. This is necessary because the intercalates we are interested in depend on the cycle structure of certain pairs of rows, so we need a way to study a random Latin square after revealing some information about some of its rows.

\subsection{Statement on AI use} ChatGPT 6 Pro was used in the final stages of proofreading.

\section{A concentration inequality}
Throughout this paper, we will frequently use a general-purpose concentration inequality for functions of independent Bernoulli random variables. The following statement appears as \cite[Corollary~6]{warnke2016method}\footnote{In~\cite{warnke2016method}, \cref{freedman_inequality} is stated only for the upper tail of $\phi(\vec {\mathbf g})$; applying it to $-\phi(\vec {\mathbf g})$ yields the lower tail. Note that we have omitted the $(1-p_i)$ factor in the denominator since we do not need it.}; it is proved with the martingale approach of Freedman~\cite{Fre75}.
\begin{theorem}
\label{freedman_inequality}
Let $\vec{\mathbf{g}} = (\mathbf{g}_1, \dots, \mathbf{g}_N)$ be a sequence of independent Bernoulli random variables with $\mathbb{P}[\mathbf{g}_i=1]= p_i$. Let $\phi:\{0,1\}^N \rightarrow \mathbb{R}$ satisfy $| \phi(\vec {g}) - \phi(\pvec{g})| \leq K_i$ for all pairs $\vec{g}, \pvec{g} \in \{0,1\}^N$ differing only in the $i$-th coordinate, and suppose $K_i \leq K$ for all $i\in[N]$. Then for all $t \geq 0$, we have
\[ \mathbb{P}\Big[\big|\phi(\vec {\mathbf g}) - \mathbb{E}[\phi(\vec{ \mathbf{g}})]\big| \geq t\Big] \leq 2\exp\Bigg( - \frac{t^2}{ 2 \sum_{i=1}^N p_i K_i^2 + 2Kt/3 }\Bigg).\]
\end{theorem}

Note that, in order to obtain a strong bound from \cref{freedman_inequality}, we need to know that $\phi(\vec g)$ is not very sensitive to changes in any individual $g_i$. In this respect \cref{freedman_inequality} is similar to the much more well-known \emph{bounded differences inequality} of McDiarmid~\cite{McD89} (which is a consequence of the Azuma--Hoeffding martingale concentration inequality). However, we need the stronger statement of \cref{freedman_inequality} for its essentially-optimal dependence on $p_1,\dots,p_N$. 

\section{Approximation lemma}
In this section, we state and prove \cref{lem:random_LS_to_triangle_removal}, which allows us to approximate uniformly random completions of Latin rectangles to Latin squares in terms of the triangle removal process. In \cref{subsec:cor}, we state a convenient corollary, which approximates random completions of Latin rectangles via binomial random hypergraphs. The latter are typically easier to work with than the triangle removal process itself. Note that a similar approximation lemma and its corresponding corollary were proved in~\cite{kwan2025parities}; the advantage of our new result is that it allows us to condition on the presence of a given Latin rectangle in the random Latin square. The proof approach is similar to the proof in~\cite{kwan2025parities}.

To be able to state our approximation lemma, we first introduce some definitions.

\begin{definition}[Hypergraph view of Latin squares]\label{def:hypergraph_view}
We can view Latin squares as sets of row, column, and symbol triples, which can be thought of as subsets of $[n]^3$. The set $[n]^3$, in turn, can be identified with the complete $3$-uniform $3$-partite hypergraph $K^{(3)}_{n,n,n}$. Given a subgraph $G\subseteq K^{(3)}_{n,n,n}$, say a ``conflict'' is a pair of hyperedges which intersect in two vertices. Thus, a partial Latin square is nothing more than a subgraph of $K^{(3)}_{n,n,n}$ with no conflicts. Similarly, a partial $w\times n$ Latin rectangle is a subgraph of $K^{(3)}_{w,n,n}$ with no conflicts.
\end{definition}

In what follows, we will interchangeably take the hypergraph and the array view of (partial) Latin squares.

\begin{definition}[Latin rectangles]
Fix $w\in [n]$ and let $\cL_n^w$ be the set of $w\times n$ Latin rectangles with rows indexed by $[w]$, and $\cL_n$ be the set of $n \times n$ Latin squares. Let $\cL_n^{\overline{w}}$ be the set of $(n-w)\times n$ Latin rectangles with rows indexed by $[n]\setminus [w]$.
\end{definition}

\begin{definition}[Partial Latin rectangles]
\label{def:partial-latin-rectangles}
Let $\cP_{n,m}^{\,\overline{w}}$ be the set of all partial Latin rectangles with $m$ entries, rows indexed by $[n]\setminus [w]$ and columns indexed by $[n].$ Let $\cP_{n}^{\,\overline{w}}=\bigcup_{m\leq (n-w)n}\cP_{n,m}^{\,\overline{w}}$ be the set of all partial $(n-w)\times n$ Latin rectangles. We let $\cP_{n,m}:=\cP^{\mkern 1mu \overline 0}_{n,m}$ and $\cP_{n}:=\cP^{\mkern 1mu \overline 0}_{n}$ be the corresponding sets of partial Latin squares.

Fix a $w\times n$ Latin rectangle $W\in \cL_n^w$. Let $\cP_m^W$ be the set of all $(n-w)\times n$ partial Latin rectangles $P$ with rows indexed by $[n]\setminus [w]$ and with $m$ hyperedges that don't conflict with $W$ (i.e., such that $W\cup P$ is a partial Latin square), and let $\cP^W = \bigcup_{m \leq (n-w)n} \cP^W_m$. 
\end{definition}

\begin{definition}[Ordered partial Latin rectangles]
An \emph{ordered} partial Latin rectangle is a partial Latin rectangle together with an ordering on its hyperedges.
We define $\vec \cP_{n,m}^{\,\overline{w}}, \vec \cP^{\,\overline{w}}_n, \vec \cP^W_m,$ and $ \vec \cP^W$ as the ordered versions of $ \cP_{n,m}^{\,\overline{w}}, \cP^{\,\overline{w}}_n, \cP^W_m$, and $ \cP^W$ respectively. That is,  e.g.\ $\vec \cP_{n,m}^{\,\overline{w}}$ is the set of all ordered partial Latin rectangles with $m$ entries, rows indexed by $[n]\setminus [w]$ and columns indexed by $[n].$
\end{definition}

To formally define the triangle removal process, we need the notion of a \emph{residual graph}, which keeps track of which pairs of vertices are not yet covered by a partial Latin square or rectangle.
\begin{definition}[Residual graph]
\label{def:residue-graph}
    For a $w \times n$ Latin rectangle $W\in \cL_n^w$, let $G^W\subseteq K_{n-w,n,n}$ be the graph which contains an edge $uv$ (of $K_{n-w,n,n}$) if and only if there is no hyperedge of $W$ containing $u$ and $v$. That is to say, $G^W$ is the graph of pairs that would need to be covered to complete $W$ to a full Latin square. Note that one of the parts of this tripartite graph has only $n-w$ vertices, corresponding to the $n-w$ rows not in $W$. 
    
    Let $P\in \cP_m^W$ be a partial Latin rectangle not conflicting with $W$. Then let $G^W(P)$ be the subgraph of $G^W$ which contains an edge $uv$ if and only if there is no hyperedge of $P$ containing $u$ and $v.$
\end{definition}

\newcommand{\abort}{\perp}
\begin{definition}[Triangle removal process starting from a Latin rectangle]
    For a $w \times n$ Latin rectangle $W\in \cL_n^w$, let $\trp(W, m)$ be the partial Latin square distribution obtained from $m$ steps of the triangle removal process starting from $G^W$. That is, we start from the residual graph $G^W$ and during the triangle removal process we sequentially add edges to an initially empty outcome, which is and will remain a partial Latin square throughout. At each step, we select one triangle (copy of $K_3$) in the residual graph uniformly at random, add its corresponding row/column/symbol triple to the outcome and remove its three edges from the residual graph. Thus, an outcome of $\trp(W,m)$ always has $m$ hyperedges, and none of these hyperedges conflict with $W$ (unless we run out of triangles and get the outcome ``$\abort$''). We sometimes treat outcomes of $\trp(W,m)$ as ordered, and other times as unordered partial Latin squares.
\end{definition}

Our approximation lemma will compare Latin rectangles to ordered partial Latin rectangles. To relate these two types of objects, we need a notion of inheritance, defined as follows. 
Note that we identify a property of Latin squares (rectangles) with the set of all Latin squares (rectangles) which have that property, and similarly for ordered and/or partial Latin squares (rectangles).

\begin{definition}[Property inheritance]\label{def:inherited}
    Let $\vec\cU\subseteq \vec{\cP}_{n}^{\,\overline{w}}$ be a property of ordered partial Latin rectangles  and $\cT\subseteq\cL_{n}^{\overline{w}}$ be a property of Latin rectangles. For $m \leq n(n-w)$, we say $\vec\cU$ is \emph{$(\rho,m)$-inherited} from $\cT$ if for any $L \in \cT$, taking $\vec{\bfP}_m(L)$ to be a uniformly random subset of $m$ hyperedges of $L$, equipped with a uniformly random order, we have $\vec{\bfP}_m(L) \in \vec\cU$ with probability at least $\rho$. We similarly define $(\rho,m)$-inheritance for properties $\cU \subseteq \cP^{\,\overline{w}}_n$.
\end{definition}

Our approximation lemma will concern the uniform distribution over all completions of a Latin rectangle to a Latin square, defined as follows.
\begin{definition}[Completions]
\label{def:completions}
For any $W \in \cL_n^w$, we denote by $\cL_n(W)$ the set of completions of $W$, that is, the set of $n\times n$ Latin squares whose first $w$ rows induce $W$. We further denote by $\cL_n^{\overline{w}}(W)$ the set of Latin rectangles obtained from $\cL_n(W)$ by removing the first $w$ rows of each Latin square. Similarly, for $P \in \cP_n$, we can define $\cL_n(P)$ analogously.
\end{definition}

We can now state our approximation lemma.

\begin{lemma}
    \label{lem:random_LS_to_triangle_removal}
    Fix any $\alpha, \zeta \in (0,1)$ with $1-\alpha = \Omega(1) $ and $\zeta = \Omega(1)$. Let $\rho \in (0,1]$ with $\rho \geq \exp(-o(\log^2{n}))$. Let $w\leq (1 - \zeta)n$ and $m=\alpha n(n-w)$. Suppose $W\in\cL_n^w$ is a $w \times n$ Latin rectangle. Consider properties $\vec{\cU}\subseteq \vec{\cP}_{n}^{\,\overline{w}}$ and $\cT\subseteq\cL_{n}^{\overline{w}}$ such that $\vec{\cU}$ is $(\rho,m)$-inherited from $\cT$. For $\vec{\bfR} \sim \trp(W,m)$ and $\mathbf{L} \sim \on{Unif}(\cL_n^{\overline{w}}(W))$,
    \[ \mathbb{P}[\mathbf{L} \in \cT ] \leq \exp\Bigg(\frac{n \log^2 {n}}{1-\alpha-o(1)}\Bigg) \mathbb{P}[ \vec{\mathbf{R}} \in \vec{\cU}].\]
\end{lemma}

\cref{lem:random_LS_to_triangle_removal} is a generalisation of Lemma 4.7 in~\cite{kwan2025parities}, with the main difference being that in~\cite{kwan2025parities}, we do not condition on a Latin rectangle $W$ (that is, we have $w = 0$). The proof of \cref{lem:random_LS_to_triangle_removal} follows closely that of Lemma 4.7 in~\cite{kwan2025parities}. Before we present it, we introduce some more notation and state some helper lemmas. To reason about residual graphs of Latin rectangles, we need concepts of density, triangle-typicality and quasirandomness for them that are normalized in terms of the size of the Latin rectangle.

\begin{definition}[Triangle count and co-degree]
    For a graph $G$, we denote by $X_\triangle(G)$ the number of triangles in $G$. For $W \in \cL_n^w$, $P \in \cP^W$, and vertices $u,v$, let $N_W(u,v)$ and $N_{W\cup P}(u,v)$ be the sets of common neighbours of $u$ and $v$ in $G^W$ and in $G^W(P)$, respectively. 
\end{definition}

\begin{remark}\label{rem:triangles-and-common-neighbors-in-latin-rectangle}
Note that for $W \in \cL_n^w$, we have $X_\triangle(G^W)=n(n-w)^2$, since the bipartite graph between the column and symbol parts of $G^W$ contains $n(n-w)$ edges, and each such edge forms a triangle with each of the $n-w$ vertices in the row part. Additionally, for each row $u$ and each column/symbol $v$, the number of common neighbours of $u$ and $v$ in $G^W$ is precisely $n-w$, since $v$ has $n-w$ neighbours that are not rows, and all of them are also neighbours of $u$. Each column $u$ and symbol $v$ also have precisely $n-w$ common neighbours in $G^W$, namely all rows. Thus, $|N_W(u,v)| = n-w$ for all $u$ and $v$ in different parts of $G^W$. Finally, we have $e(G^W)=3n(n-w)$.
\end{remark}

\begin{definition}[$W$-relative density, triangle-typicality and quasirandomness]
\label{def:density-typicality-quasirandomness}
Let $\varepsilon, \gamma\geq 0$, $W\in\cL_n^w$ and $P\in\cP^W$. 

\begin{itemize}

\item We define the $W$-\emph{relative density} of $P$ as $\dens(G^W(P))=e(G^W(P))/e(G^W)$.

\item We say that $G^W(P)$ is $W$-\emph{relative $\varepsilon$-triangle-typical} if its number of triangles satisfies \newline $X_\triangle(G^W(P)) \in(1\pm \varepsilon)X_\triangle(G^W)\dens(G^W(P))^3.$ 
Let $\mc P^{W,\triangle:\varepsilon}_m$ be the set of partial Latin rectangles $P\in \cP_m^W$ such that $G^W(P)$ is $W$-relative $\varepsilon$-triangle-typical. Let $\vec{\mc P}^{W,\triangle:\varepsilon}_m$ be the set of ordered partial Latin rectangles $\vec{P}\in \vec{\cP}_m^W$ such that for each $i\leq m$, if we consider the partial Latin rectangle $P_i$ consisting of the first $i$ hyperedges of $\vec{P}$, then $P_i\in \mc P^{W,\triangle:\varepsilon}_i$.

\item We say that $G^W(P)$ is $W$-\emph{relative $\gamma$-quasirandom} if for each $u,v$ in different parts, the number of common neighbours of $u$ and $v$ in $G^W(P)$ satisfies $ |N_{W\cup P}(u,v)| \in (1\pm\gamma)|N_{W}(u,v)|\dens(G^W(P))^2.$ Let $\mc P^{W,\gamma}_m$ be the set of partial Latin rectangles $P\in \cP_m^W$ such that $G^W(P)$ is $W$-relative $\gamma$-quasirandom, and let $\vec \cP_m^{W,\gamma}$ be the set of ordered partial Latin rectangles whose unordered analogues are in $\cP_m^{W,\gamma}$.
\end{itemize}
\end{definition}

The following two lemmas show that most random subsets of Latin rectangles are triangle-typical and quasirandom.
\begin{lemma}
\label{lem:epsilon-triangular} Let $\alpha, \zeta \in (0, 1)$ with $1-\alpha = \Omega(1) $ and $\zeta = \Omega(1)$. Let $w\leq (1 - \zeta)n$ and $m=\alpha n(n-w)$. Consider any $w\times n$ Latin rectangle $W\in\cL_n^w$, any $L\in \cL_n^{\overline{w}}(W)$, and let $\vec{\bfP}_{m}(L)$ be a random ordering of a random set of $\alpha n(n-w)$ hyperedges of $L$. Then, with $\varepsilon = n^{-1} \log n$, we have
\[
\Pr[\vec{\bfP}_{m}(L) \notin \vec{\mc P}^{W,\triangle:\varepsilon}_m] \leq \exp(-\Omega(\log^2 n)).
\]
\end{lemma}
\begin{proof}
We will use the union bound over all $m'\leq \alpha n(n-w)$, so fix some $m' \in [\alpha n(n-w)]$. We define $\bfB_p(L)$ to be a random subset of hyperedges of $L$ where each hyperedge is taken independently with probability $p=m'/(n(n-w)).$ Moreover, let $\bfX$ be the number of triangles in $ G^W(\bfB_p(L))$. By ``Pittel’s inequality'' (see e.g.\ \cite[page~17]{janson2011random}), for any property $\cQ$ of unordered partial
Latin rectangles,\footnote{With a mild abuse of notation, we denote by $\vec{\bfP}_{m'}(L)\notin \cQ$ the event ``the unordered version of $\vec{\bfP}_{m'}(L)$ is not in $\cQ$''.}
\[\mathbb{P}[\vec{\bfP}_{m'}(L)\notin \cQ]\leq 3\sqrt{m'}\,\, \Pr[\bfB_p(L)\notin \cQ].\]
Thus, it is sufficient to show that $\Pr[\bfX \notin (1 \pm \varepsilon)n(n-w)^2(1-p)^3] \leq \exp(-\Omega(\log^2 n))$.
Consider any vertices $v_1,v_2, v_3$ in distinct parts of $G^W(\bfB_p(L))$. If $v_1v_2v_3 \in L$, the probability that $v_1v_2v_3$ is a triangle in $G^W(\bfB_p(L))$ is precisely $1 - p$ and there are $n(n-w)$ such triples. On the other hand, if $v_1v_2v_3 \notin L$ and $v_1v_2v_3$ is a triangle in $G^W$, the probability $v_1v_2v_3$ is a triangle in $G^W(\bfB_p(L))$ is $(1 - p)^3$, since each pair of vertices among $v_1, v_2, v_3$ is in a distinct hyperedge of $L$, and all these hyperedges need to be absent from $G^W(\bfB_p(L))$ in order for $v_1v_2v_3$ to be in $G^W(\bfB_p(L))$. Note that there are $n(n-w)^2-n(n-w)$ triples $v_1,v_2,v_3$ of this type (as there are $n(n-w)^2$ triangles in $G^W$ by \cref{rem:triangles-and-common-neighbors-in-latin-rectangle}). Thus, we get
\[
\mu := \mb{E}[\bfX] = n(n-w)(1 - p) + (n(n-w)^2 - n(n-w))(1 - p)^3 = \Theta(n^3).
\]
There are $n(n-w)$ hyperedges in $L$, and adding or removing a hyperedge from $\bfB_p(L)$ can affect $\bfX$ by at most $1+3(n-w-1)\leq 3n$. The conclusion of the lemma follows by \Cref{freedman_inequality}.
\end{proof}

\begin{lemma}
\label{lem:quasirandomness}
Let $\alpha, \zeta \in (0, 1)$ with $1-\alpha = \Omega(1) $ and $\zeta = \Omega(1)$. Let $ w\leq (1 - \zeta)n$ and let $m=\alpha n(n-w)$. Consider any $w\times n$ Latin rectangle $W\in\cL_n^w$, any $L\in \cL_n^{\overline{w}}(W)$ and let $\bfP_{m}(L)$ be a random subset of $\alpha n(n-w)$ hyperedges of $L$. Let $\gamma \in [0,1]$ with \( \gamma = \omega(n^{-1/2} \log^{1/2} n) \). Then
\[
\Pr[\bfP_{m}(L) \notin \mc P^{W,\gamma}_{m}] \leq \exp(-\Omega(\gamma^2 n)).
\]

\end{lemma}
\begin{proof}
We need to show that $G^W(\bfP_m(L))$ is $W$-relative $\gamma$-quasirandom with probability at least $1 - \exp(-\Omega(\gamma^2 n))$.
As in the proof of the previous lemma, let $\bfB_p(L)$ be a random subset of the hyperedges of $L$ where each hyperedge is taken independently at random with probability $p = m/(n(n-w)) = \alpha$. Let $\bfS_{u,v}$ be the set of common neighbours of $u$ and $v$ in the graph $G^W(\bfB_p(L))$. It suffices to show that for every
$u, v$ that are in different parts of $G^W(\bfB_p(L))$, it holds that $\Pr[|\bfS_{u,v}| \notin (1 \pm \gamma)(n-w)(1 - \alpha)^2] \leq \exp(-\Omega(\gamma^2 n))$.
A union bound over all pairs of vertices \( u, v \) and Pittel’s inequality would then complete the proof.

Let $\bfX = |\bfS_{u,v}|$ and recall that $|N_W(u,v)|=n-w$ by \cref{rem:triangles-and-common-neighbors-in-latin-rectangle}. We distinguish two cases based on whether $u,v$ are connected by an edge in $G^W$. 

If $u,v$ are neighbours in $G^W$, let $x$ be the unique vertex such that $uvx$ is a hyperedge of $L$. Then $\Pr[x\in\bfS_{u,v}]=1-p$. On the other hand, for all $x'\in N_W(u,v)$ different from $x$ we have $\Pr[x'\in \bfS_{u,v}]=(1-p)^2$. Thus, we get
\[
\mu := \mb{E}[\bfX] = (1 - p) + (n -w- 1)(1 - p)^2 = \Theta(n).
\]
There are $2(n-w)-1$ hyperedges in $L$ that affect $\bfX$, and adding or removing such a hyperedge from $\bfB_p(L)$ can affect $\bfX$ by at most 1. 

Otherwise, $u,v$ are not neighbours in $G^W$, so we get
\[
\mu := \mb{E}[\bfX] = (n -w)(1 - p)^2 = \Theta(n).
\]
In this case, we have $2(n-w)$ hyperedges in $L$ that each affect $\bfX$ by at most 1.
We finish both cases by \Cref{freedman_inequality}.
\end{proof}

We will also need a lower bound on the number of completions of a Latin rectangle to a Latin square, as follows.

\begin{theorem}
\label{thm:completions-of-rectangle}
Fix any $\zeta\in (0,1)$ and let $w\leq (1-\zeta)n$. Consider any $w\times n$ Latin rectangle $W\in \cL^w_n$.
The number of completions of $W$ to a Latin square is at least
\[|\cL_n(W)|\geq \left(\frac{n-w}{e^2}\right)^{n(n-w)}.\]
\end{theorem}
\begin{proof}
It follows from the Egorychev--Falikman lower bounds on the permanent~\cite{Ego81,Fal81} that the number of extensions by one row of a $k\times n$ Latin rectangle is at least $(n-k)^n\, n!/n^n$ (see for example~\cite[Theorem 4.A6]{cameron1976parallelisms}). Hence, using Stirling's inequality $a!\geq (a/e)^{\textnormal{a}}$ which holds for all $a\in\mathbb{Z}^+$, the number of ways to complete a $w\times n$ Latin rectangle is at least 
\[ \prod_{k=w}^{n-1} \left( (n-k)^n\, \frac{n!}{n^n}\right) 
= \big( (n-w)!\big)^n\, \left(\frac{n!}{n^n}\right)^{n-w}
 \geq  \left(\frac{n-w}{e^2}\right)^{n(n-w)}.  \qedhere
 \]
\end{proof}

Given a Latin rectangle $W$, let $P\in \cP^W$ be a partial Latin rectangle such that $G^W(P)$ is $W$-relative quasirandom and $W$-relative triangle-typical.
The next lemma gives an upper bound on the number of completions of the partial Latin square $W \cup P$  to a Latin square. Its proof relies on the entropy method (see, e.g.\ \cite{samotij2026entropy}).
\begin{lemma}
\label{lem:completions_of_partial_latins}
Fix any $\alpha,\zeta \in (0, 1)$ with $1- \alpha = \Omega(1)$ and $\zeta = \Omega(1)$. Let $w\leq (1 - \zeta)n$, $m=\alpha(n-w)n,$ $\gamma = o(1)$, and $\varepsilon = n^{-1} \log n $. Consider any $w\times n$ Latin rectangle $W\in\cL_n^w$. Then for any partial Latin square $P \in \mc P^{W,\gamma}_m \cap \mc P^{W,\triangle:\varepsilon}_m$,
the number  $|\cL_n(W\cup P)|$ of completions of $W\cup P$ to an $n \times n$ Latin square satisfies
\[
|\cL_n(W\cup P)| \leq \left(\frac{(1 - \alpha)^{2}(n-w)}{e^2}\right)^{(1 - \alpha)n(n-w)} \cdot \exp\left(\frac{n \log^2 n}{1 - \alpha - o(1)} \right).\]
\end{lemma}
\begin{proof}
For a row $x$ (respectively, column $y$), we write $r_x$ (respectively, $c_y$) for the number of empty cells in that row (respectively, column), in $W\cup P$. Then for an empty cell $e = (x,y)$, we write $Q_{x,y}$ for the set of symbols that do not appear in the row or column of $e$ in $W\cup P$. Note that $Q_{x,y}$ is precisely the set of common neighbours of $x$ and $y$ in $G^W(P)$. Then, let $\bfL(W\cup P)$ be a uniformly random completion of $W\cup P$. We will estimate the entropy $H[\bfL(W\cup P)] = \log |\mc L_n(W\cup P)|$ of $\bfL(W\cup P)$.

    Let $E \subseteq [n]^2$ be the set of empty cells in $W\cup P$. For each $e = (x,y) \in E$, let $\bfz_e$ be the symbol in cell $e$ in $\bfL(W\cup P)$. So the sequence $(\bfz_e)_{e \in E}$ determines $\bfL(W\cup P)$. For any total ordering $\prec$ on $E$, we have that by the chain rule for entropy~\cite{MR2239987},
    \begin{align}
    \label{equation:entropy-cell-by-cell}
    H[\bfL(W\cup P)] = \sum_{e \in E} H\big[\bfz_e \,\big|\,(\bfz_{e'} : e' \prec e)\big].
    \end{align}
    Now, consider a pair of sequences $\mu = (\mu_x)_x, \nu = (\nu_y)_y\in [0,1]^n$ (of real numbers in the range $[0,1]$, which we assume to be all distinct from each other). Let $\lambda = (\lambda_{x,y})_{x,y}$ be the array defined by $\lambda_{x,y} = (\mu_x, \nu_y)$. We can use $\lambda$ to define a total order $\prec_\lambda$ on $E$, via reverse lexicographic order on the pairs $\lambda_{x,y}$. To be precise, write $(x',y')\prec_\lambda(x,y)$ when 
    $\mu_{x'}>\mu_{x}$ or when $x'=x$ and $\nu_{y'} > \nu_y$. 
    Let $\bfR_e(\lambda)$ be an upper bound on the conditional support size
    \[\Big|\supp\big(\bfz_e\,\big|\,(\bfz_{e'} : e' \prec_\lambda e)\big)\Big|\]
    defined as follows: $\bfR_{e}(\lambda)$ for $e=(x,y)$ is $1$ plus the number of symbols $z \in Q_e \setminus \bfz_e$ for which $\mu_{x_z} < \mu_x$ and $\nu_{y_z} < \nu_y$, where $x_z$ and $y_z$ are the row and column such that cells $(x_z,y)$ and $(x,y_z)$ contain symbol $z$ in $\bfL(W \cup P)$.

    Since $\bfR_e(\lambda)$ is an upper bound on $|\supp(\bfz_e\,|\,(\bfz_{e'}:e' \prec_\lambda e))|$, we have
    \begin{align}
    \label{inequality:entropy-upper-bound}
    H\big[\bfz_e\,\big|\,(\bfz_{e'}:e' \prec_\lambda e)\big] \leq \mathbb{E}[\log{\bfR_e(\lambda)}].
    \end{align}
    It follows from~(\ref{equation:entropy-cell-by-cell}) and~(\ref{inequality:entropy-upper-bound}) that
    \[
    H[\bfL(W\cup P)] \leq \sum_{e \in E} \mathbb{E}[\log{\bfR_e(\lambda)}].
    \]
    This is true for any fixed $\lambda$, so it is also true if $\lambda$ is chosen randomly: let $\bflambda \in \Big( [0,1]^2\Big)^{n^2}$ be obtained via $2n$ independent random variables $\bfmu_x, \bfnu_y$ each uniformly distributed in $[0,1]$ (noting that with probability $1$ they are all distinct). Then
    \begin{align}
    \label{inequality:entropy-random-lambda-upper-bound}
     H[\bfL(W\cup P)] \leq \sum_{e \in E} \mathbb{E}[\log{\bfR}_e(\bflambda)]. 
    \end{align}
    Next, for any completion $L$ of $W\cup P$, any $e = (x,y)$ and any $\lambda_e = (\mu_x, \nu_y) \in [0,1]^2$, let
    \[ R_e^{L,\lambda_e} = \mathbb{E}[\bfR_e(\bflambda)\, |\, \bfL(W\cup P) = L, \bflambda_e=\lambda_e].\]
    (Note that $\bflambda_e = \lambda_e$ occurs with probability zero, so formally we should condition on $\bflambda_e=\lambda_e \pm \mathrm{d} \lambda_e$ and take limits as $\mathrm{d} \lambda_e \rightarrow 0$ in what follows, but there are no continuity issues so we will ignore this detail.) By the definition of $\bfR_e(\bflambda)$ and linearity of expectation, we have
    \[ R_e^{L, \lambda_e} = 1 + (|Q_e| - 1) \mu_x \nu_y.\]
    By Jensen's inequality,
    \[ \mathbb{E}[\log{\bfR_e(\bflambda)} \,|\, \bfL(W\cup P) = L, \bflambda_e=\lambda_e] \leq \log{R_e^{L,\lambda_e}},\]
    so
    \begin{align*}
    \mathbb{E}[\log{\bfR_e(\bflambda)} \,|\, \bfL(W\cup P) = L] &\leq \mathbb{E}\Big[ \log{R_e^{L, \bflambda_e}} \Big]\\
    &= \int_{[0,1]^2} \log\big(1 + (|Q_e| - 1) \mu_x \nu_y\big) \,\mathrm{d} \lambda_e.\end{align*}
    For $C > 0$ we compute
    \[ \int_0^1 \int_0^1 \log(1 + Cts) \,\mathrm{d}t\, \mathrm{d}s = \log(1+C) - 2 + \frac{\log(1+C) - \Li_2(-C)}{C}, \]
    where
    \[ \Li_2(C)= - \int_0^C \frac{\log(1-t)}{t}\,\mathrm{d}t \]
    is an evaluation of the polylogarithm function, so, as $C \rightarrow \infty$,
    \begin{align*} -\Li_2(-C) =  \int_0^{-C} \frac{\log(1-t)}{t} \,\mathrm{d}t &= \int_0^{C} \frac{\log(1+t)}{t}\,\mathrm{d}t\\
    &\leq \int_{0}^1 \frac{\log(1+t)}{t}\, \mathrm{d}t + \sum_{t=1}^{C-1} \frac{\log(2+t)}{t} \leq \log^2(C+1)  + O(1).\end{align*}
    By definition $|Q_e| = |Q_{x,y}|$ is the number of common neighbours of $x$ and $y$ in $G^W(P)$. As $G^W(P)$ is $W$-relative $\gamma$-quasirandom, we have 
    \begin{align}
    \label{eq:bounds_Q_e}
    |Q_e| = (1 \pm \gamma)(n-w)(1 - \alpha)^2 \geq (n-w)(1-\alpha-o(1))^2.
    \end{align}
    Therefore, $|Q_e| \rightarrow \infty$ as $n \rightarrow \infty$, so we can substitute $C = |Q_e|-1$ and get
    \begin{align*} 
    \mathbb{E}[\log{\bfR_e(\bflambda)} \,|\, \bfL(W\cup P) = L]  &\leq \log{|Q_e|} - 2 + \frac{\log{|Q_e|} + \log^2{|Q_e|} + O(1) }{|Q_e|}\\ 
    &= \log{|Q_e|} - 2 +\frac{(1 + o(1)) \log^2{|Q_e|} }{|Q_e|}. 
    \end{align*}

Using (\ref{eq:bounds_Q_e}) and that $|Q_e| \leq n$, and noting that the expression above does not depend on the choice of $L$, we get
\begin{align}
\label{inequality:upper-bound-common-neighbors}
\mb{E}[\log \bfR_e(\bflambda)] \leq \log |Q_e| - 2 + \frac{\log^2 n}{(1-\alpha-o(1))^2(n-w)}.
\end{align}
Next, because $G^W(P)$ is $W$-relative $\varepsilon$-triangle-typical,  the number of triangles in $G^W(P)$ is:
\[
\sum_{e \in E} |Q_e| = (1 \pm \varepsilon)(1 - \alpha)^3n(n-w)^2.
\]
By Jensen's inequality and the fact that $|E|=(1-\alpha)n(n-w)$ we get:
\[
\sum_{e \in E} \log |Q_e| \leq |E| \log \left( \frac{\sum_{e \in E} |Q_e|}{|E|} \right)
\leq |E| \log \left( (1 + \varepsilon)(1 - \alpha)^2(n-w) \right).
\]
Combining this with \cref{inequality:entropy-random-lambda-upper-bound} and \cref{inequality:upper-bound-common-neighbors} we get
\begin{align*}
\log |\cL_n(W\cup P)| = H[\bfL(W \cup P)] &\leq \sum_{e \in E} \mb{E}[\log \bfR_e(\bflambda)] \leq \sum_{e \in E} \log |Q_e| + |E| \left(\frac{\log^2 n}{(1-\alpha-o(1))^2(n-w)} - 2\right)\\
&\leq (1 - \alpha)n(n-w) \bigg(\log \big((1 + \varepsilon)(1 - \alpha)^2 (n-w)\big) - 2\bigg)
+ \frac{n \log^2 n}{1 - \alpha - o(1)}.
\end{align*}
Thus,
\begin{align*}
|\cL_n(W\cup P)| &\leq \left(\frac{(1 + \varepsilon)(1 - \alpha)^{2}(n-w)}{e^2}\right)^{(1 - \alpha)n(n-w)} \cdot \exp\left( \frac{n \log^2 n}{1 - \alpha - o(1)} \right)\\
&\leq \left(\frac{(1 - \alpha)^{2}(n-w)}{e^2}\right)^{(1 - \alpha)n(n-w)} \cdot \exp\left( \varepsilon (1-\alpha) n(n-w)+\frac{n \log^2 n}{1 - \alpha - o(1)} \right)\\
&= \left(\frac{(1 - \alpha)^{2}(n-w)}{e^2}\right)^{(1 - \alpha)n(n-w)} \cdot \exp\left(\frac{n \log^2 n}{1 - \alpha - o(1)} \right),
\end{align*}
as $\varepsilon=n^{-1}\log n$ and so $\exp(\varepsilon (1-\alpha) n(n-w))=\exp(o(n\log^2 n))$.
\end{proof}

Finally, we need an estimate of the probability of a particular $W$-relative triangle-typical outcome of the triangle removal process.

\begin{lemma}
\label{lem:uniform_trp_results}
 Let $ \alpha,\zeta \in (0, 1)$. Let $w\leq (1 - \zeta)n$ and $m=\alpha (n-w)n$. Consider any $w\times n$ Latin rectangle $W\in\cL_n^w$ and any \( \vec{P} \in \vec{\mc P}^{W,\triangle:\varepsilon}_{m} \) for some \( \varepsilon \in (0,1) \). Then for \( \vec{\bfR} \sim \trp(W, m) \) and \(N = n(n-w)\), we have
\[
\Pr[\vec{\bfR} = \vec{P}] = (1 \pm 2\varepsilon)^{\alpha N} \exp (O(\log n)) \left( \frac{e^3}{n(n-w)^2} \right)^{\alpha N} (1 - \alpha)^{3N (1-\alpha)}.
\]
\end{lemma}
\begin{proof}
Since $\vec{P}\in \vec \cP_m^{W, \Delta:\varepsilon}$,
\begin{align*}
&\Pr[\vec{\bfR} =\vec{P}] \\
&= \prod_{i=0}^{\alpha N - 1} \frac{1}{(1 \pm \varepsilon)(1 - i/N)^3 n(n-w)^2} = (1 \pm 2\varepsilon)^{\alpha N} n^{-\alpha N}(n-w)^{-2\alpha N} \exp \left( -3 \sum_{i=0}^{\alpha N - 1} \log \left( 1 - \frac{i}{N} \right) \right).
\end{align*}
Note that 
\[ \sum_{i=0}^{\alpha N -1} \frac{1}{N} \log\Big(1 - \frac{i+1}{N}\Big) \leq \int_{0}^\alpha \log(1-t) \mathrm{d}t \leq  \sum_{i=0}^{\alpha N -1} \frac{1}{N} \log\Big(1 - \frac{i}{N}\Big). \]
We have
\[ \sum_{i=0}^{\alpha N - 1} \Bigg( \log\Big(1 - \frac{i}{N}\Big) -\log\Big(1 - \frac{i+1}{N}\Big)   \Bigg) = \sum_{i=0}^{\alpha N - 1}  \log\Big( 1 + \frac{1}{N- (i+1)} \Big) \leq \sum_{i=0}^{\alpha N - 1}  \frac{1}{N- (i+1)} = O(\log{n}).\]
Thus, using the fact that $ \log{x}$ has antiderivative $ x(\log{x} - 1)$, we get
\begin{align*}
3 \sum_{i=0}^{\alpha N -1} \log\Big(1 - \frac{i}{N}\Big) &=   3N \int_{0}^\alpha \log(1-t) \,\mathrm{d}t \pm O(\log{n})\\
&=  3 N \int_{1-\alpha}^1 \log{s}\ \mathrm{d}s \pm O(\log{n})\\
&= 3N \big( - \alpha - (1-\alpha)\log(1-\alpha)\big) \pm O (\log{n}),\end{align*}
so
\[
\exp \left( -3 \sum_{i=0}^{\alpha N - 1} \log \left( 1 - \frac{i}{N} \right) \right) = e^{3N \alpha} (1 - \alpha)^{3N (1-\alpha)}\, \exp (O(\log n)).
\]
Therefore,
\[
\Pr[\vec{\bfR} = \vec{P}] = (1 \pm 2\varepsilon)^{\alpha N} \exp (O(\log n)) \left( \frac{e^3}{n(n-w)^2} \right)^{\alpha N} (1 - \alpha)^{3N (1-\alpha)}.
\qedhere\]
\end{proof}

We are now ready to prove our approximation lemma by combining the helper lemmas stated so far.

\begin{proof}[Proof of \cref{lem:random_LS_to_triangle_removal}]
Since $\vec \cU$ is $\rho$-inherited from $\cT$, for each $L\in \cT$, we have $\Pr[\vec\bfP_m(L)\in \vec\cU]\geq \rho.$ Let $\cS\coloneq\vec\cU\cap \vec{\mc P}^{W,\triangle:\varepsilon}_m\cap \vec{\mc P}^{W,\gamma}_m$ for $\varepsilon := n^{-1}\log{n}$ and $\gamma:=n^{-1/4}$. Then by \cref{lem:epsilon-triangular} and \cref{lem:quasirandomness}, we have
\begin{align}
\label{eqn:lower_bound_S}
\Pr[\vec\bfP_m(L)\in \cS]\geq \rho - \exp(-\Omega(\log^2{n})) = \rho(1-o(1)).
\end{align}
Next, let $\mathbf{L} \sim \on{Unif}(\cL_n^{\overline{w}}(W))$ and consider the probability $\Pr[\vec\bfP_m(\bfL)\in \cS]$. We can bound this probability from below by $\Pr[\bfL\in\cT]\cdot \rho(1-o(1))$ due to \cref{eqn:lower_bound_S}. As $\Pr[\vec\bfP_m(\bfL)\in \cS]=\sum_{\vec P\in\cS}\Pr[\vec\bfP_m(\bfL)=\vec P]$, we get:
\[\Pr[\bfL\in\cT]\leq \frac{1}{\rho(1-o(1))}\sum_{\vec P\in\cS}\Pr[\vec\bfP_m(\bfL)=\vec P].\]
Letting $\cL_n(W\cup\vec P)$ be the set of completions of the union of $W$ and the unordered analogue of $\vec P$, and setting $N=(n-w)n$, we get
\[\Pr[\bfL\in\cT]\leq \frac{1}{\rho(1-o(1))}\sum_{\vec P\in\cS} \frac{|\cL_n(W\cup\vec P)|}{|\cL_n(W)|}\cdot\frac{(N-m)!}{N!}\]
since for each $(n-w) \times n$ Latin rectangle, there are $N(N-1)\dots (N-m+1)$ ways to select an (ordered) sequence of $m$ hyperedges.

Using \cref{thm:completions-of-rectangle} and \cref{lem:completions_of_partial_latins}, we have
\begin{align*}
\Pr[\bfL\in\cT] &\leq \frac{1+o(1)}{\rho}\sum_{\vec P\in\cS} \left(\frac{(1 - \alpha)^{2}(n-w)}{e^2}\right)^{(1 - \alpha)N} \exp\left(\frac{n \log^2 n}{1 - \alpha - o(1)} \right) \cdot \left(\frac{e^2}{n-w}\right)^{N}\cdot\frac{(N-m)!}{N!}\\
&\leq \sum_{\vec P\in\cS} \rho^{-1}\exp\left(\frac{n \log^2 n}{1 - \alpha - o(1)}\right)\cdot \left(\frac{e^2}{n-w}\right)^{\alpha N}\cdot (1-\alpha)^{2N(1-\alpha)}\cdot\frac{(N-m)!}{N!}.
\end{align*}
By Stirling's approximation and since $\rho \geq e^{-o(\log^2{n})}$, we have
\begin{align*}
\Pr[\bfL\in\cT]&\leq \sum_{\vec P\in\cS}\exp\left(\frac{n \log^2 n}{1 - \alpha-o(1)}\right)\cdot \left(\frac{e^2}{n-w}\right)^{\alpha N}\cdot (1-\alpha)^{2N(1-\alpha)}\cdot\left(\frac{N(1-\alpha)}{e}\right)^{N(1-\alpha)}\cdot \left(\frac{e}{N}\right)^N\\
&= \sum_{\vec P\in\cS}\exp\left(\frac{n \log^2 n}{1 - \alpha - o(1)}\right)\cdot \left(\frac{e^3}{n(n-w)^2}\right)^{\alpha N}\cdot (1-\alpha)^{3N(1-\alpha)}.
\end{align*}
On the other hand, for each $\vec P \in\cS$ by \cref{lem:uniform_trp_results} we have:
\begin{align*}
\Pr[\vec{\bfR} = \vec{P}] &\geq (1 - 2\varepsilon)^{\alpha N} \exp (O(\log n)) \left( \frac{e^3}{n(n-w)^2} \right)^{\alpha N} (1 - \alpha)^{3N (1-\alpha)}\\
&= \exp \big(O(\log{n})-(2+o(1))\alpha (n-w)\log n\big) \left( \frac{e^3}{n(n-w)^2} \right)^{\alpha N} (1 - \alpha)^{3N (1-\alpha)}.
\end{align*}
Therefore,
\[\Pr[\bfL\in\cT]\leq \sum_{\vec P\in\cS}\exp\left(\frac{n \log^2 n}{1 - \alpha - o(1)}\right)\Pr[\vec{\bfR} = \vec{P}]\leq \exp\left(\frac{n \log^2 n}{1 - \alpha - o(1)}\right)\Pr[\vec{\bfR} \in \vec{\cU}],\]
since $\cS\subseteq \vec{\cU}$.
\end{proof}
\subsection{A corollary in terms of binomial random hypergraphs}

\label{subsec:cor}

In this subsection, we state a corollary of our approximation lemma that allows us to transfer results about binomial random hypergraphs to results about random Latin squares. To do this, we introduce some minor notational definitions.

\begin{definition}[Latin rectangles with arbitrarily indexed rows]
\label{def:completions-arbitrary}
For a set $\cV \subseteq [n]$, we denote by $\cL_n^\cV$ the set of Latin rectangles with rows indexed by $\cV$ and columns indexed by $[n]$, and we let $\cL_n^{\overline{\cV}}:= \cL_n^{[n]\setminus \cV}$. For any $W \in \cL_n^\cV$, we denote by $\cL_n(W)$ the set of completions of $W$, that is, the set of $n\times n$ Latin squares whose rows indexed by $\cV$ induce $W$. We further denote by $\cL_n^{\overline{\cV}}(W)$ the set of Latin rectangles obtained from $\cL_n(W)$ by removing the rows indexed by $\cV$ from each Latin square. We similarly define $\cP^{\overline{\cV}}_n$ as the analogue of $\cP^{\overline{w}}_n$ (Definition~\ref{def:partial-latin-rectangles}).
\end{definition}
Note that we stated \cref{lem:random_LS_to_triangle_removal} only for rectangles $W \in \cL^w_n$, but the same proof goes through for any $\cV \subseteq [n]$ and any $W \in \cL^\cV_n$ by permuting the rows.

Recall that the residual graph $G^W$ of a (partial) Latin rectangle $W$ (\cref{def:residue-graph}) records the pairs of vertices that are not covered by $W$.

\begin{definition}
\label{def:partial-random-hypergraph}
Let $W$ be a $w\times n$ Latin rectangle. We denote by $G^{W,(3)}$ the 3-partite 3-uniform hypergraph on the same vertex set as $G^W$, where we include an edge $ijk$ if and only if $ijk$ is a triangle in $G^W.$
Let ${\mb G}^{(3)}(W, p)$ be the random subgraph of $G^{W,(3)}$ obtained by including each hyperedge independently with probability $p.$ 
\end{definition}

We will also need the notion of $(\rho,m)$-inheritedness of properties (\cref{def:inherited}), which naturally generalizes to (partial) Latin rectangles indexed by $\cV$.

A property $\cU$ of partial Latin rectangles is monotone decreasing if $P \in \cU$ and $P' \subseteq P$ imply $P' \in \cU$ (i.e., ``removing hyperedges doesn't hurt''), and monotone increasing if $P \in \cU$ and $P' \supseteq P$ imply $P' \in \cU$ (i.e., ``adding hyperedges doesn't hurt'').
\begin{lemma}[Random Latin squares to binomial random hypergraphs]
\label{lem:LS-to-hypergraphs}
Let $\alpha, \zeta \in (0,1)$ with  $\zeta = \Omega(1)$ and $\alpha=  1- \Omega(1)$. Consider some subset $\cV\subseteq [n]$ of rows with $w=|\cV|$ and $w \leq (1-\zeta)n$, and consider a $\cV\times n$ Latin rectangle $W$. Let $\mc U \subseteq \cP^{\overline{\cV}}_n$ be a property of subgraphs of $K^{(3)}_{n,n,n}$ and let $\cT \subseteq \cL_n^{\overline{\cV}}(W)$, and suppose that $\mc U$ is \emph{$(\rho,m)$-inherited} from $\cT$, for $m=\alpha n (n-w)$ and  $\rho\ge \exp(-o(\log^2 n))$. Let $\bfL \sim \unif (\cL^{\overline{\cV}}_n(W))$ be a random Latin rectangle sampled uniformly from all Latin rectangles with rows indexed by $[n] \setminus \cV$ whose union with $W$ is a Latin square.
\begin{enumerate}
    \item \label{LS-to-hyper-decreasing} If $\cU$ is a monotone decreasing property, we have \[ \mathbb{P}[\mathbf{L} \in \cT] \leq \exp\big(O(n \log^2 {n})\big)  \mathbb{P}[\mathbf{G^*} \in \cU],\]
    where $\bfG \sim \mb G^{(3)}(W,\alpha/(n-w))$, and $\bfG^*$ is obtained from $\bfG$ by deleting all hyperedges involved in conflicts.
    \item \label{LS-to-hyper-increasing} If $\cU$ is a monotone increasing property, then for any $\gamma \in (0,1)$ such that $\frac{\gamma}{(1+\gamma)^2} \geq \frac{12 \alpha}{\zeta}$ and $\gamma^2 \alpha n^2 = \omega( n\log^2{n})$, we have \[ \mathbb{P}[\mathbf{L} \in \cT] \leq \exp\big(O(n \log^2 {n})\big)  \mathbb{P}[\mathbf{G} \in \cU] + \exp(-\Omega(\gamma^2 \alpha n^2)),\]
    where $\bfG \sim \mb G^{(3)}(W,(1+\gamma)\alpha /(n-w))$.
\end{enumerate}
\end{lemma}
\begin{proof}
By the symmetry of rows in Latin squares, we can assume without loss of generality that $\cV = [|\cV|]$. Let $ \bfR \sim \trp(W, \alpha n(n-w))$. It suffices to show that $\Pr[ \bfR \in \cU] = O(\Pr[\bfG^{*} \in \cU])$ for part~(\cref{LS-to-hyper-decreasing}) and $\Pr[\bfR \in \cU] \leq \Pr [\bfG \in \cU] + \exp(-\Omega(\gamma^2 \alpha n^2))$ for part~(\cref{LS-to-hyper-increasing}), as then the conclusion of the lemma would follow by \cref{lem:random_LS_to_triangle_removal}, where we consider the `ordered' version $\vec \cU$ of property $\cU$.

We first show $\Pr[ \bfR \in \cU] = O(\Pr[\bfG^{*} \in \cU])$ for part~(\cref{LS-to-hyper-decreasing}). We can couple  $\bfR$ and $\bfG^*$ such that, if $e(\bfG) \leq \alpha n(n-w)$, then either $\bfG^* \subseteq \bfR$ or $\bfR = \bot$ (to see this, randomly order the hyperedges of $\bfG$ and run the triangle removal process on that ordering). Since $\cU$ is monotone decreasing, we have
\[
\Pr[\bfR \in \cU] \leq \Pr[\bfG^* \in \cU \mid e(\bfG) \leq \alpha n (n-w)].
\]
Since $e(\bfG)$ is binomial with mean $\alpha n(n-w)$, we have that $\Pr[e(\bfG) \leq \alpha n(n-w)] = \Omega(1)$. Thus,
\[
\Pr[\bfR \in \cU] \leq \frac{\Pr[\bfG^* \in \cU]}{\Pr[e(\bfG) \leq \alpha n(n-w)]} = O(\Pr[\bfG^* \in \cU]).
\]

We now turn to showing $\Pr[\bfR \in \cU] \leq \Pr [\bfG \in \cU] + \exp(-\Omega(\gamma^2 \alpha n^2))$ for part~(\cref{LS-to-hyper-increasing}). Let $p=(1+\gamma)\alpha/(n-w)$ and let $\bfG^*$ be obtained from $\bfG$ by deleting all hyperedges involved in conflicts. Note that we can couple $(\bfG^*, \bfG)$ and $\bfR$ such that $\bfR \subseteq \bfG$ as long as $\bfG^*$ has at least $m$ hyperedges (to see this, randomly order the hyperedges in $\bfG$ and run the triangle removal process on that ordering).

Let $\bfX$ be the number of hyperedges in $\bfG^*$ (i.e., the number of isolated hyperedges in $\bfG$). Note that there are precisely $n(n-w)^2 $ potential hyperedges for $\bfG$ (these are the hyperedges that do not conflict with $W$). Each of these is present and isolated in $\bfG$ with probability at least $p (1-p)^{3n}  \geq (1 + \gamma)\alpha \exp\big(-6(1+\gamma)\alpha / \zeta\big) /(n-w)$. Thus, we have 
\[ \mathbb{E}[ \bfX] \geq n (n-w) \cdot (1+\gamma)\alpha \exp\big(-6(1+\gamma)\alpha/\zeta\big)   \geq (1+\gamma/2)\alpha n(n-w) , \]
where we used that $\frac{\gamma}{(1+\gamma)^2} \geq \frac{12\alpha}{\zeta}$. Adding a hyperedge to $\bfG$ can increase $\bfX$ by at most $1$, and removing a hyperedge from $\bfG$ can increase $\bfX$ by at most $3$, by making three hyperedges isolated. Thus, by \cref{freedman_inequality},
\[ \mathbb{P}[\bfX < \alpha n(n-w)] \leq \mathbb{P}[|\bfX - \mathbb{E}[\bfX]| \geq \gamma \alpha n(n-w)/2 ] \leq \exp(-\Omega(\gamma^2 \alpha n^2)).\]
Therefore,
\[ \mathbb{P}[\bfR \in \cU] \leq \mathbb{P}[\bfG \in \cU] + \mathbb{P}[\bfX < \alpha n (n-w)] = \mathbb{P}[\bfG \in \cU] + \exp(-\Omega(\gamma^2 \alpha n^2)). \qedhere \]
\end{proof}

\section{Intercalates and the effects of switching them}
\label{sec:intercalates_and_effects}

Intercalates constitute a central tool in the proof of our main theorem. We make use of them in two different ways: to flip the parities of the rows, columns, and symbols they participate in, and to change the cycle structure of the permutations induced by pairs of rows, columns, or symbols.

\begin{definition}
In a Latin square $L$, an \emph{intercalate} $I\subseteq L$ is a $2\times 2$ Latin subsquare. To \emph{switch} $I$ means to exchange its two symbols: this results in another Latin square $\mr{sw}_I(L)$. For a set $\cI$ of pairwise disjoint intercalates in $L$, we define $\mr{sw}_\cI (L)$ as the Latin square obtained by switching each intercalate in $\cI$ (note that these operations commute).
\end{definition}
We also use the following notation to describe the cycle structure between two rows. For a row index $r$, we denote by $L[r]$ the $1 \times n$ Latin rectangle containing row $r$ in $L$. Similarly, for a column index $c$, we denote by $L[c]$ the $n \times 1$ Latin rectangle containing column $c$ in $L$, and for a symbol $s$, we have that $L[s]$ is the partial Latin square consisting of all occurrences of symbol $s$ in $L$.\footnote{Here, we slightly abuse notation. Whenever we refer to rows, columns, and symbols, we regard them as indices in $[n]$ with an implicit ``type'' (the ``type'' is always either ``row'', ``column'', or ``symbol''). Hence $L[z]$ for $z\in [n]$ depends on the ``type'' of $z$ and should always be clear from the context.}
\begin{definition}\label{def:permutation-between-rows}
For two $1\times n$ Latin rectangles $D,D'$, let $\sigma(D,D')$ be the permutation that rearranges the entries of the unique row of $D$ into the entries of the unique row of $D'$. That is, if $\sigma_D$ and $\sigma_{D'}$ are the row permutations given by the only row in $D$ and $D'$ respectively, then $\sigma(D,D') = \sigma_{D'} \circ \sigma_D^{-1}$, where $\pi_1 \circ \pi_2$ is given by $(\pi_1 \circ \pi_2)(i) = \pi_1 (\pi_2(i))$. For two rows $x,r\in [n]$ in a Latin square $L$, we write $\sigma_L(x,r):=\sigma(L[x],L[r])$.  We define analogous notions for columns and symbols.
\end{definition}

To understand how intercalate switchings affect the cycle structure of pairs of rows, columns, or symbols, we need the following fact.
\begin{fact}\label{fact:intercalate-switching}
    Let $x,r\in [n]$ be two rows of a Latin square $L\in \mc L_n$. Consider an intercalate $I$ involving $r$ but not $x$, and let $c_1,c_2\in [n]$ be the two columns of $I$.
    \begin{itemize}
        \item If $c_1,c_2$ lie in different cycles\footnote{$\sigma_L(x,r)$ is defined as a permutation on symbols, but we can equivalently think of it as a permutation on columns (using the correspondence between symbols and columns given by row $x$).}
        of $\sigma_L(x,r)$, then switching $I$ causes these two cycles to merge into a single cycle.
        \item If $c_1,c_2$ lie on the same cycle of $\sigma_L(x,r)$, then switching $I$ causes this cycle to split into two cycles (one of these two cycles corresponds to the part of the original cycle after $c_1$ and before and including $c_2$, and the other cycle corresponds to the part of the original cycle after $c_2$ and before and including $c_1$). We then say that the intercalate switch \emph{excises} each of these two cycles.
    \end{itemize}
\end{fact}
Note that the same fact applies to $\sigma(D,L[r])$ for any $1 \times n$ Latin rectangle $D$. Similar facts can be stated about the effect of intercalate switches on the cycle structure of two columns or two symbols, as can be seen by the symmetry between rows, columns, and symbols.

\subsection{Stable intercalates}
For a Latin square $L$, we denote by $N_{\mathrm{row}}(L)$, $N_{\mathrm{col}}(L)$, $N_{\mathrm{sym}}(L)$ the number of odd rows, columns, and symbols in $L$, respectively. The approach in~\cite{kwan2025parities} for understanding the distribution of $(N_{\mathrm{row}}(\bfL)$, $N_{\mathrm{col}}\bf(L)$, $N_{\mathrm{sym}}(\bfL))$ involves controlling the parities of most rows, columns, and symbols by switching a ``canonical'' set of disjoint intercalates which we refer to as \emph{stable} intercalates. In this subsection, we introduce some definitions and results from~\cite{kwan2025parities} related to stable intercalates, which we need for our proof.

The important property of stable intercalates is that if a Latin square $L'$ can be obtained from $L$ by switching a stable intercalate, then $L$ and $L'$ have the same set of stable intercalates (except that of course one of them is switched in $L'$). This motivates the following definition.

\begin{definition}[Stable intercalates]\label{def:stable_intercalates}
Let $P$ be a partial Latin square. 
For an intercalate $A$, we denote by $\bar{A}$ the outcome of switching $A$.
\begin{itemize}
    \item We say that an intercalate in $P$ is \emph{isolated} 
    if it does not share an entry with any other intercalate.
    \item Consider a set of isolated intercalates $\{A_1, \dots, A_t\}$ in $P$. We say that this is a \emph{critical} set of intercalates if it is possible to switch a subset of $A_1, \dots, A_t$ (resulting in intercalates $B_1, \dots, B_t$ where $B_i = A_i$ or $B_i = \bar{A_i}$ for each $i\in [t]$) to create a new intercalate that shares an entry with each of $B_1, \dots, B_t$. (Note that this implies $t \leq 4$.)
    \item We say that an isolated intercalate in $P$ is \emph{stable} if it is not contained in a critical set of intercalates.
\end{itemize}
\end{definition}

As stable intercalates ``behave independently from each other'', sometimes it is useful to record their presence without remembering whether they are switched or not. This is formalized in the next definition.

\begin{definition}\label{def:intercalate-abstraction}
    For an intercalate $A$ with rows $r_1,r_2$, columns $c_1,c_2$, and symbols $s_1,s_2$, define the tuple $\tau(A) = (\{r_1,r_2\}, \{c_1,c_2\}, \{s_1,s_2\})$ (i.e., this records the rows, columns and symbols of the intercalate, without recording which of the two possible ``states'' it is in). For a (partial) Latin square $P$, let $\cS(P)$ be the set of all tuples $\tau(A)$ for all stable intercalates $A$ in $P$ (i.e., this records the intercalate switches we can perform, without recording which state each of those intercalates is in).
\end{definition}

Now, we can state the previously-mentioned property that the set of stable intercalates (in the sense of $\cS(\cdot)$) does not change upon switching some subset of them.

\begin{lemma}{\emph{\cite[Lemma~7.3]{kwan2025parities}}}
\label{lem:canonicity_of_stable_intercalates}
Let $P_1$ be a (partial) Latin square, and suppose that we can switch a single stable intercalate in $P_1$ to obtain a second (partial) Latin square $P_2$. Then $\cS(P_1) = \cS(P_2)$.
\end{lemma}

The next definition formalizes the intuition that we would like the stable intercalates to be well distributed over the rows, columns, and symbols.

\begin{definition}[Intercalate expanders]
\label{def:intercalate-expander}Let $P$ be an $n \times n$ (partial) Latin square, and consider some $\ell\in [n]$ and $\beta>0$. We say that $P$ is an \emph{$(\ell,\beta)$-intercalate-expander} if the following property holds.
For any sets of rows $\sR, \sR^*$, any sets of columns $\sC, \sC^*$, and any sets of symbols $\sS, \sS^*$, such that five out of these six sets have size at least $\beta n$ and the remaining set has size at least $\ell$, there is a stable intercalate in $P$ with one row in $\sR$, the other row in $\sR^*$, one column in $\sC$, the other column in $\sC^*$, one symbol in $\sS$, and the other symbol in $\sS^*$.
\end{definition}
Note that this definition does not tell us anything about the intercalates that intersect a fixed row (unless $\ell=1$, but we will have $\ell = \log^{11}{n}$). We expect that the number of intercalates intersecting a fixed row should behave nicely, but the currently known techniques for studying random Latin squares are insufficient to show this, since it is a property that is not likely enough to go through permanent estimates (see \cref{lem:random_LS_to_triangle_removal}). Thus, we are instead forced to reason about the number of intercalates intersecting large sets of rows, columns, and symbols, as in the definition above. 

In a uniformly random Latin square, most intercalates heavily intersect with each other, which means there are very few or no stable intercalates. To remedy this, we randomly sparsify our Latin square by intersecting it with a `template' in order to ensure that many of the intercalates are stable.

\begin{definition}\label{def:template}
An $n\times n$ \emph{template} is a subset $T\subseteq [n]^2$, which we interpret as a set of row/column pairs. For a Latin square $L\in \mc L_n$, we write $T \cap L\in \mc P_n$ for the partial Latin square containing just the entries of $L$ in positions specified by $T$.
\end{definition}

The next lemma shows that the intersections of most templates and most Latin squares are good intercalate expanders, which will allow us to control the parities of most rows, columns, and symbols.

\begin{lemma}[Most templates and Latin squares give intercalate expanders]
\label{lemma:existence_split_stable_intercalate} Let $\beta >0$ be a constant. Let $\bfT \subseteq [n]^2$ be a random subset of row/column pairs, where each is included with probability $\varepsilon = o(\log^{-1}{n})$ independently. Let $\bfL \sim \mr{Unif}(\mc L_n)$ be a uniformly random Latin square (interpreted as a subgraph of $K^{(3)}_{n,n,n}$) and let $\ell = \log^{11}n$. Then 
\[ \mathbb{P}[ \bfT\cap \bfL \text{ is an }(\ell,\beta)\text{-intercalate-expander}\,] \geq 1- \exp(-\omega(n \log^2{n})).\]
\end{lemma}

To prove ~\cref{lemma:existence_split_stable_intercalate}, we state a slightly different version of it, shown in~\cite{kwan2025parities}, in which the template is a random hypergraph.

\begin{lemma}{\emph{\cite[Lemma~9.2]{kwan2025parities}}}
\label{lemma:intercalate_expander_random_hypergraph_template}
Fix a constant $\beta >0$ and let $\ell=\log^{11}n$. Consider a random hypergraph $\bfH \sim \mb G^{(3)}(n,\varepsilon)$ with $\varepsilon = o(\log^{-1}n)$ and an independent random Latin square $\bfL \sim \on{Unif}(\cL_n)$ (both interpreted as subgraphs of $K^{(3)}_{n,n,n}$). Then
\[ \mathbb{P}[ \bfH\cap \bfL \text{ is an }(\ell,\beta)\text{-intercalate-expander}\,] \geq 1- \exp(-\omega(n \log^2{n})).\]
\end{lemma}

We now show the (simple) deduction of \cref{lemma:existence_split_stable_intercalate} from \cref{lemma:intercalate_expander_random_hypergraph_template}, which is almost identical to the proof of \cite[Lemma 8.3]{kwan2025parities} but we include the details for completeness. We first record the following simple fact.
\begin{fact}\label{fact:joint_prob_vs_most_outcomes}
Suppose $\bfH_1 \in \cH_1$ and $\bfH_2 \in \cH_2$ are independent random objects, let $\cP^{(2)} \subseteq \cH_1 \times \cH_2$, and for any $p\in [0,1]$ let $\cP(p)$ be the set of all $H_1\in \mc H_1$ such that $\Pr[(H_1,\mbf H_2)\in \mc P^{(2)}]> p$. Then
\begin{enumerate}
    \item \label{fact:joint_to_separate} if $\mathbb{P}[(\bfH_1, \bfH_2) \in \cP^{(2)}] \leq p$, then $\Pr[\mbf H_1\in \mc P(\sqrt p)]\le \sqrt p$.
    \item \label{fact:separate_to_joint} If $\Pr[\mbf H_1\in \mc P(p_2)]\le p_1$, then $\mathbb{P}[(\bfH_1, \bfH_2) \in \cP^{(2)}] \leq p_1+p_2$.
\end{enumerate}
\end{fact}

\begin{proof}[Proof of \cref{lemma:existence_split_stable_intercalate}]
Let $\bfH \sim \mathbb{G}^{(3)}(n,\varepsilon)$ with $\varepsilon = o(\log^{-1}{n})$. Say that $L\in \cL_n$ is \emph{good} if \[\Pr[\bfH \cap L\text{ is an }(\ell,\beta)\text{-intercalate-expander}]\ge 1 - \exp(-\omega(n\log^2{n})).\]
By \cref{lemma:intercalate_expander_random_hypergraph_template} and \cref{fact:joint_prob_vs_most_outcomes}(\ref{fact:joint_to_separate}), $\mbf L$ is good with probability $1-\exp(-\omega(n\log^2{n} ))$.

Note that for each fixed outcome $L$ of $\bfL$, the two random partial Latin squares $\bfH \cap L$ and $\bfT \cap L$ have exactly the same distribution. Therefore, for each good $L$,
\[\Pr[\bfT \cap L\text{ is an }(\ell,\beta)\text{-intercalate-expander}]\ge 1 - \exp(-\omega(n\log^2{n})),\]
and \cref{fact:joint_prob_vs_most_outcomes}(\ref{fact:separate_to_joint}) yields
\[\Pr[\bfT \cap \bfL\text{ is an }(\ell,\beta)\text{-intercalate-expander}]\ge 1 - \exp(-\omega(n\log^2{n})). \qedhere \]
\end{proof}

\subsection{Robust parity generators}

In this subsection, we show that stable intercalates can be used to control the parities of most rows, columns, and symbols. We start with a linear-algebraic definition that reflects how intercalates act on the parities of the rows, columns, and symbols of a Latin square.
\begin{definition}
Consider an $n\times n$ (partial) Latin square $L$. For a set of rows $\cR$
and a set of intercalates $\mathcal{I}$ in $L$, let $M_{\cR,\mathcal{I}}\in\{0,1\}^{|\cR|\times|\cI|}$
be the ``incidence matrix'' defined as follows:
\[
M_{\cR,\mathcal{I}}(r,I)=\begin{cases}
1 & \text{if row }r\text{ participates in intercalate }I\\
0 & \text{otherwise}.
\end{cases}
\]
We can similarly define $M_{\cC,\mc I}$ or $M_{\cS,\mc I}$ for a set of columns $\cC$ or a set
of symbols $\cS$, and
\[M(\cR,\cC,\cS,\cI) = 
\begin{pmatrix}M_{\cR,\mathcal{I}}\\M_{\cC,\mathcal{I}}\\M_{\cS,\mathcal{I}}\end{pmatrix}\in\{0,1\}^{(|\cR|+|\cC|+|\cS|)\times|\mathcal{I}|}
\]
\end{definition}

\begin{remark}
\label{remark:matrix_interpretation}
For a set $\cR$ of rows, a set $\cC$ of columns, a set $\cS$ of symbols, and a set $\cI$ of intercalates in a Latin square $L$, we can interpret $M = M(\cR, \cC, \cS,\cI)$ in the following way. If we switch a subset $\cI'$ of the intercalates in $\cI$ described by $\vec i \in \mathbb{F}_2^{|\cI|}$ (where $\vec i_I = 1$ if and only if $I \in \cI'$), then $M \vec i \in \mathbb{F}_2^{|\cR| + |\cC| + |\cS|}$ describes the changes in the parities of the rows, columns, and symbols in $\cR,\cC,\cS$. That is, if $\vec \xi \in \mathbb{F}_2^{|\cR| + |\cC| + |\cS|}$ denotes the parities of the rows, columns, and symbols $\cR,\cC,\cS$ in $L$, then $\vec \xi + M \vec i$ describes their parities in $\mr{sw}_{\cI'}(L)$. Thus, the image of $M$ corresponds to all the different row, column, and symbol parity sequences that are achievable by switching some intercalates in $\cI$. The $2^3$ vectors of the form $(\vec{x}_{\cR}\; \vec{x}_{\cC}\; \vec{x}_{\cS})\in \mb F_2^{|\cR|+|\cC|+|\cS|}$ where $\vec{x}_Z\in \mb F_2^{|Z|}$ is the all-zero vector $(0,\dots,0)$ or the all-one vector $(1,\dots,1)$ for each $Z \in \{\cR,\cC,\cS\}$ are always left kernel vectors of $M$, since each intercalate intersects precisely two rows, two columns, and two symbols. This reflects the fact that intercalate switches cannot alter the parity of the number of odd rows, the number of odd columns, or the number of odd symbols and implies that $M$ cannot have rank more than $|\cR|+|\cC|+|\cS|-3$ (since $\textrm{im}(M) = \textrm{ker}(M^T)^\perp$).
\end{remark}

\begin{definition}
For a partial Latin square $P$ and sets $\cR$ of rows, $\cC$ of columns, and $\cS$ of symbols, we denote by $P[\cR,\cC,\cS] \subseteq P$ the partial Latin subsquare of $P$ consisting only of entries with rows in $\cR$, columns in $\cC$, and symbols in $\cS$.
\end{definition}

A parity generator, defined below, is a set of stable intercalates that can be used to control the parities of some rows, columns, and symbols in a Latin square.

\begin{definition}[Parity generators]
For a partial Latin square $P$, rows $\cR$, columns $\cC$, and symbols $\cS$, a set $\cI$ of stable intercalates in $P[\cR ,\cC,\cS]$ is a \emph{parity generator} (with respect to $P, \cR, \cC, \cS$) if the following holds. The matrix $M(\cR,\cC,\cS,\cI)
$ has rank exactly $|\cR|+|\cC|+|\cS|-3$ over $\mb F_2$, and the left kernel vectors of this matrix are precisely the $2^3$ vectors of the form $(\vec{x}_{\cR}\; \vec{x}_{\cC}\; \vec{x}_{\cS})\in \mb F_2^{|\cR|+|\cC|+|\cS|}$ where $\vec{x}_Z\in \mb F_2^{|Z|}$ is the all-zero vector $(0,\dots,0)$ or the all-one vector $(1,\dots,1)$ for each $Z \in \{\cR,\cC,\cS\}$. For a parity generator $\cI$ for $P,\cR,\cC,\cS$, we call the rows $\cR^{\textnormal{s}} = [n] \setminus \cR$, the columns $\cC^{\textnormal{s}} = [n]\setminus \cC$, and the symbols $\cS^{\textnormal{s}} = [n] \setminus \cS$ \emph{stubborn}.
\end{definition}

Let $\vec \xi_{\cR}( L)\in \mathbb{F}^{|\cR|}_2,$ $\vec \xi_{\cC}( L)\in \mathbb{F}^{|\cC|}_2$ and $\vec \xi_{\cS}( L)\in \mathbb{F}^{|\cS|}_2$ respectively denote the sequence of parities of the rows $\cR$, columns $\cC$, and symbols $\cS$ in a Latin square $L$. Observe that if $\cI$ is a parity generator for $P,\cR,\cC$, and $\cS$, then the image of $M=M (\cR,\cC,\cS,\cI)$ consists of precisely the vectors $(\vec x_\cR \; \vec x_\cC \; \vec x_\cS) \in \mathbb{F}_2^{|\cR|+|\cC|+|\cS|}$ with $|\vec x_\cR|_1, |\vec x_\cC|_1$, and $|\vec x_\cS|_1$ even (that is, the vectors orthogonal to the left kernel space of $M$). By \cref{remark:matrix_interpretation}, this implies that given any Latin square $L\supseteq P$, we can achieve any desired parities $\vec{r}\in \mb F_2^{|\cR|}$ for the rows $\cR$, $\vec{c}\in \mb F_2^{|\cC|}$ for the columns $\cC$, and $\vec{s}\in \mb F_2^{|\cS|}$ for the symbols $\cS$, by switching some subset of the intercalates in $\mathcal{I}$, assuming that
\[
|\vec{r}|_1=|\vec \xi_{\cR}( L)|_1\!\!\pmod 2,\quad \quad
|\vec{c}|_1=|\vec \xi_{\cC}( L)|_1\!\!\pmod 2,\quad \quad 
|\vec{s}|_1=|\vec \xi_{\cS}( L)|_1\!\!\pmod 2.
\]

We also need a notion of robustness for parity generators, since our proof requires us to forbid using some entries in the stable intercalate switches.
\begin{definition}[Robust parity generators]
For a partial Latin square $P$, rows $\cR$, columns $\cC$, and symbols $\cS$, a set $\cI$ of stable intercalates in $P[\cR ,\cC,\cS]$ is a $k$-\emph{robust parity generator} (with respect to $P, \cR, \cC, \cS$) if for every choice $E$ of at most $k$ entries, the set of intercalates $\cI_E\subseteq \cI$ that do not intersect $E$ is a parity generator with respect to $P, \cR,\cC,\cS$. For a $k$-robust parity generator $\cI$ for $P,\cR,\cC,\cS$, we call the rows $\cR^{\textnormal{s}} = [n] \setminus \cR$, the columns $\cC^{\textnormal{s}} = [n]\setminus \cC$, and the symbols $\cS^{\textnormal{s}} = [n] \setminus \cS$ \emph{stubborn}.
\end{definition}

The following lemma states that a random Latin square intersected with a random template contains a robust parity generator for almost all rows, columns, and symbols. It is a variation of {\cite[Lemma~8.5]{kwan2025parities}} with an added robustness condition that allows us to forbid a set of entries.

\begin{lemma}[Few stubborn rows, columns, symbols: robust version]\label{lem:few_stubborn_rows_cols_symbols}
Let $\bfT \subseteq [n]^2$ be a random subset of row/column pairs, where each is included with probability $\varepsilon = o(\log^{-1}{n})$ independently. Let $\bf{L} \sim \unif(\cL_n)$. With probability at least $1 - \exp(-\omega(n\log^2{n}))$, the following holds.\ There are sets $\cR^{\textnormal{s}}, \cC^{\textnormal{s}},\cS^{\textnormal{s}} \subset [n]$ of stubborn rows, columns, and symbols respectively, of size at most $\ell = \log^{11}{n}$ each, such that the set $\cI$ of $\cR^{\textnormal{s}} \cup \cC^{\textnormal{s}} \cup\cS^{\textnormal{s}}$-avoiding\footnote{Recall that we regard rows, columns and symbols as indices in $[n]$ with an implicit ``type'' (either ``row'', ``column'' or ``symbol''). With this slight abuse of notation, the set $\cR^{\textnormal{s}}\cup \cC^{\textnormal{s}}\cup \cS^{\textnormal{s}}$ is a disjoint union of typed indices; we use this convention throughout the rest of the paper.} stable intercalates in $\bfT \cap \bfL$ is a $n/(100\log^{11}{n})$-robust parity generator with respect to $\bfT\cap \bfL, \cR^{\textnormal{f}}, \cC^{\textnormal{f}},\cS^{\textnormal{f}}$, where $\cR^{\textnormal{f}}:=[n]\setminus \cR^{\textnormal{s}}$,\,\, $\cC^{\textnormal{f}}:=[n]\setminus \cC^{\textnormal{s}}$,\,\, $\cS^{\textnormal{f}}:=[n]\setminus \cS^{\textnormal{s}}$.
\end{lemma}

\begin{proof}
Let $\beta = 1/10$ and let $L$ and $T$ be outcomes of $\bfL$ and $\bfT$ respectively such that the conclusion of~\cref{lemma:existence_split_stable_intercalate} holds, which happens with probability $1-\exp(-\omega(n\log^2{n}))$. Thus, from now on we assume that $T \cap L$ is an $(\ell,\beta)$-intercalate-expander. Let $\cI^{\textnormal{stable}}$ be the set of stable intercalates in $T \cap L$, and note that they are disjoint. After we choose $\cR^{\textnormal{s}}, \cC^{\textnormal{s}}, \cS^{\textnormal{s}}$, we will take $\cI$ to be the intercalates in $\cI^{\textnormal{stable}}$ that do not intersect $\cR^{\textnormal{s}},\cC^{\textnormal{s}},\cS^{\textnormal{s}}$.

Recall that for a set of entries $E$, we denote by $\cI_E \subseteq \cI$ the intercalates avoiding $E$. For every choice of $\cR^{\textnormal{f}},\cC^{\textnormal{f}},\cS^{\textnormal{f}}$ and every choice of $E$ of size at most $n/(100\log^{11}n)$, we have that the vectors of the form $(\vec{x}_{\cR^{\textnormal{f}}}\; \vec{x}_{\cC^{\textnormal{f}}}\; \vec{x}_{\cS^{\textnormal{f}}})\in \mb F_2^{|\cR^{\textnormal{f}}|+|\cC^{\textnormal{f}}|+|\cS^{\textnormal{f}}|}$, where $\vec{x}_Z\in \mb F_2^{|Z|}$ is either $(0,\dots,0)$ or $(1,\dots,1)$ for each $Z \in \{\cR^{\textnormal{f}},\cC^{\textnormal{f}},\cS^{\textnormal{f}}\}$, are left kernel vectors of $M(\cR^{\textnormal{f}},\cC^{\textnormal{f}},\cS^{\textnormal{f}},\cI_E)$. That is because each column of $M_{\cR^{\textnormal{f}},\cI_E}$, $M_{\cC^{\textnormal{f}},\cI_E}$, and $M_{\cS^{\textnormal{f}},\cI_E}$ has precisely two non-zero entries. Thus, $M(\cR^{\textnormal{f}},\cC^{\textnormal{f}},\cS^{\textnormal{f}},\cI_E)$ has rank at most $|\cR^{\textnormal{f}}| + |\cC^{\textnormal{f}}|  + |\cS^{\textnormal{f}}| - 3$. It remains to show we can choose $\cR^{\textnormal{f}}, \cC^{\textnormal{f}}, \cS^{\textnormal{f}}$ such that for any choice $E$ of forbidden entries, $M(\cR^{\textnormal{f}},\cC^{\textnormal{f}},\cS^{\textnormal{f}},\cI_E)$ has no other left kernel vectors. 

First, we translate the desired statement into more combinatorial language. Consider the auxiliary 6-uniform hypergraph $\mc Q(\cR^{\textnormal{f}},\cC^{\textnormal{f}},\cS^{\textnormal{f}}, E)$ on the vertex set $\cR^{\textnormal{f}}\cup \cC^{\textnormal{f}}\cup \cS^{\textnormal{f}}$, where for every intercalate in $\mc \cI_E$ we put a hyperedge consisting of the two rows, two columns, and two symbols of that intercalate. We want to choose $\cR^{\textnormal{f}},\cC^{\textnormal{f}},\cS^{\textnormal{f}}$ to have the property that for any set $E$ of at most $n/(100\log^{11}{n})$ entries and for any choices of vertex subsets $\cR'\subseteq \cR^{\textnormal{f}}$ and $\cC'\subseteq \cC^{\textnormal{f}}$ and $\cS'\subseteq \cS^{\textnormal{f}}$ which do not correspond to one of the $2^3$ left kernel vectors described above (that is, at least one of $\cR',\cC',\cS'$ is nonempty and proper), there is some hyperedge of $\mc Q(\cR^{\textnormal{f}},\cC^{\textnormal{f}},\cS^{\textnormal{f}},E)$ which intersects $\cR'\cup \cC'\cup \cS'$ in an odd number of vertices.

Now, we specify $\cR^{\textnormal{f}},\cC^{\textnormal{f}},\cS^{\textnormal{f}}$ as follows. For a row $r$, we say that a set of intercalates involving $r$ are \emph{externally disjoint} if every row other than $r$, every column and every symbol is involved in at most one intercalate in the set. Let $\cR^{\textnormal{f}}$ be the set of all rows $r$ for which there exist at least $n/(10\log^{11}{n})$ externally disjoint intercalates in $\cI^{\textnormal{stable}}$ which all involve $r$. Define $\cC^{\textnormal{f}},\cS^{\textnormal{f}}$ similarly. 

\begin{claim*}
Each of $\cR^{\textnormal{f}},\cC^{\textnormal{f}},\cS^{\textnormal{f}}$ has size greater than $n-\ell$.
\end{claim*}
\begin{claimproof}
    Suppose for contradiction (and without loss of generality) that there is a set $\cR_0$ of $\ell$ rows which are not in $\cR^{\textnormal{f}}$. Then we can use the intercalate-expansion property to greedily find $n/10$ stable intercalates in $T \cap L$ which all intersect $\cR_0$ in one row and otherwise involve disjoint rows, columns, and symbols. So, there must be some $r\in \cR_0$ involved in $(n/10)/|\cR_0|\ge n/(10\log^{11}{n})$ stable intercalates which involve disjoint rows, columns, and symbols (except that they all involve $r$). This contradicts the choice of $\cR_0$.
\end{claimproof}
With $\cR^{\textnormal{f}}, \cC^{\textnormal{f}}, \cS^{\textnormal{f}}$ specified, we set $\cR^{\textnormal{s}}:= [n] \setminus\cR^{\textnormal{f}}$, $\cC^{\textnormal{s}}:= [n] \setminus\cC^{\textnormal{f}}$, $\cS^{\textnormal{s}}:= [n] \setminus\cS^{\textnormal{f}}$, and let $\cI$ be the set of intercalates in $\cI^{\textnormal{stable}}$ that do not intersect $\cR^{\textnormal{s}}, \cC^{\textnormal{s}},\cS^{\textnormal{s}}$. Fix some set of entries $E$ of size at most $n/(100\log^{11}{n})$. Recall that the defining properties of $\cR^{\textnormal{f}},\cC^{\textnormal{f}},\cS^{\textnormal{f}}$ were that they are included in many (at least $n/(10\log^{11}{n})$) externally disjoint stable intercalates from $\cI^{\textnormal{stable}}$. Many of those intercalates lie in $\cI_E$,
and therefore correspond to hyperedges of $\mc Q(\cR^{\textnormal{f}},\cC^{\textnormal{f}},\cS^{\textnormal{f}},E)$. Specifically, each ``bad'' row, column, symbol or entry (which is in $\cR^{\textnormal{s}},\cC^{\textnormal{s}},\cS^{\textnormal{s}}$ or in $E$) kills at most one of our externally disjoint stable intercalates, so each vertex of $\mc Q(\cR^{\textnormal{f}},\cC^{\textnormal{f}},\cS^{\textnormal{f}},E)$ is contained in at least $n/(10\log^{11}{n}) - 3\ell - n/(100\log^{11}{n}) \ge n/(20\log^{11}{n})$ externally disjoint 6-uniform hyperedges in $Q(\cR^{\textnormal{f}},\cC^{\textnormal{f}},\cS^{\textnormal{f}},E)$.

Now, consider vertex subsets $\cR'\subseteq \cR^{\textnormal{f}}$ and $\cC'\subseteq \cC^{\textnormal{f}}$ and $\cS'\subseteq \cS^{\textnormal{f}}$ which are not all trivial (at least one of these subsets is nonempty and proper). We need to prove there is some hyperedge of $Q(\cR^{\textnormal{f}},\cC^{\textnormal{f}},\cS^{\textnormal{f}},E)$ which intersects $\cR'\cup \cC'\cup \cS'$ in an odd number of vertices.

\medskip
\noindent\textbf{Case 1: }First, suppose that for each $Z \in \{\cR', \cC', \cS'\}$, either $Z$ or its complement (in $\cR^{\textnormal{f}}, \cC^{\textnormal{f}},\cS^{\textnormal{f}}$ respectively) has size less than $\ell$. This gives us three ``small sets'' of size less than $\ell$. At least one of our small sets must be non-empty; let $v$ be a vertex in that set. Recall that $v$ is contained in at least $n/(20\log^{11}{n})$ externally disjoint hyperedges of $\mc Q(\cR^{\textnormal{f}},\cC^{\textnormal{f}},\cS^{\textnormal{f}},E)$, so at least one of these hyperedges (externally) avoids our three small sets. Any such hyperedge has odd intersection with $\cR'\cup \cC'\cup \cS'$ (its intersection with $\cR'$ has size $1$, whereas its intersections with $\cC'$ and $\cS'$ each have size $0$ or $2$).

\medskip
\noindent\textbf{Case 2: }If Case 1 does not occur, we may assume without loss of generality that $\cR'$ and $\cR^{\textnormal{f}}\setminus \cR'$ both have size at least $\ell$. Then we apply the intercalate-expansion property directly. Let $\cC_{E}$ be the set of the at most $n/(100\log^{11}{n})$ many columns that contain an entry in $E$. We take $\cR'$ and $\cR^{\textnormal{f}}\setminus \cR'$ as $\mathscr R^*$ and $\mathscr R$ respectively; we take the larger of $\cC' \setminus \cC_{E}$ and $\cC^{\textnormal{f}}\setminus (\cC' \cup \cC_{E}) $ as $\mathscr C = \sC^*$; and we take the larger of $\cS'$ and $\cS^{\textnormal{f}}\setminus \cS'$ as $\mathscr S = \sS^*$. This again gives us a stable intercalate $I$ in $T \cap L$ with the desired odd intersection with $\cR'\cup \cC'\cup \cS'$. Note that $I$ avoids $E$ since the columns in $\sC=\sC^*$ do not contain entries of $E$. Thus, $I$ corresponds to an edge of $\mc Q(\cR^{\textnormal{f}},\cC^{\textnormal{f}},\cS^{\textnormal{f}}, E)$, as desired.
\end{proof}

\section{Proof of the main result}

In this section, we give the proof of \Cref{thm:main} (subject to a technical lemma whose proof we defer to the next section). We use the switching method: we start with any Latin square $L$ with $N_{\mathrm{row}}(L)  = |\vec r|_1 \pmod 2$, $N_{\mathrm{col}}(L) = |\vec c|_1 \pmod 2$, $N_{\mathrm{sym}}(L) = |\vec s|_1 \pmod 2$ and show how to perform switching operations to obtain a Latin square $L'$ with $(\vec\xi_{\mathrm{row}}(L'),\vec\xi_{\mathrm{col}}(L'),\vec\xi_{\mathrm{sym}}(L')) = (\vec r, \vec c, \vec s)$.

Recalling the definitions in \cref{sec:intercalates_and_effects}: we use stable intercalates (more concretely, robust parity generators) to control the parities of all the non-stubborn rows, columns, and symbols. To fix the parities of the stubborn rows, columns, and symbols, we pair each of them up with a non-stubborn \emph{mate} of the same type (row, column, or symbol), and seek to switch a short even-length cycle in the permutation induced by each pair. If no such cycle exists, we create one via intercalate switches that intersect the non-stubborn mate (but not the stubborn one).

To implement this plan, we first need a few definitions. We first define a \emph{mate system}, which is the aforementioned pairing of stubborn rows, columns, and symbols with non-stubborn mates.

\begin{definition}[Mate systems]
    For a Latin square $L$ and sets $\cR^{\textnormal{s}}, \cC^{\textnormal{s}}, \cS^{\textnormal{s}}$ of rows, columns, and symbols respectively, a \emph{mate system} $ (\mathfrak{R}^{\textnormal{p}}, \mathfrak{C}^{\textnormal{p}}, \mathfrak{S}^{\textnormal{p}})$  consists of a set $\mathfrak{R}^{\textnormal{p}}$ of disjoint ordered pairs of rows, a set $\mathfrak{C}^{\textnormal{p}}$ of disjoint ordered pairs of columns, and a set $\mathfrak{S}^{\textnormal{p}}$ of disjoint ordered pairs of symbols such that each $r_1 \in \cR^{\textnormal{s}}$ has a unique distinct \emph{mate} $r_2 \notin \cR^{\textnormal{s}}$ such that $(r_1, r_2) \in \mathfrak{R}^{\textnormal{p}}$, and there are no other pairs in $\mathfrak{R}^{\textnormal p}$ (and similarly for columns and symbols). We let $\cR^{\textnormal{m}}$ denote the set of \emph{mates} $r_2$ such that $(r_1,r_2) \in \mathfrak{R}^{\textnormal{p}}$ for some $r_1 \in \cR^{\textnormal{s}}$, and define $\cC^{\textnormal{m}}$ and $\cS^{\textnormal{m}}$ analogously. Note that $\mathfrak{R}^{\textnormal{p}}$ gives a bijection between $\cR^{\textnormal{s}}$ and $\cR^{\textnormal{m}}$, and similarly for $\mathfrak{C}^{\textnormal{p}}$ and $\mathfrak{S}^{\textnormal{p}}$.
\end{definition}

As mentioned previously, our goal is to fix the parities of the stubborn rows, columns, and symbols by switching a short even-length cycle in the permutation induced by each stubborn row, column, or symbol and its mate. A Latin square in which this is possible (along with the corresponding mate system and short even-length cycles) is what we call a \emph{flexible Latin square}, defined as follows.

\begin{definition}[Flexible Latin squares]
    \label{def:flexible_LS}
    For sets $\cR^{\textnormal{s}}, \cC^{\textnormal{s}}, \cS^{\textnormal{s}}$ of rows, columns, and symbols respectively and a positive integer $k$, a $(\cR^{\textnormal{s}},\cC^{\textnormal{s}},\cS^{\textnormal{s}},k)$-\emph{flexible Latin square} is a Latin square $L$ together with a mate system $(\mathfrak{R}^{\textnormal{p}},\mathfrak{C}^{\textnormal{p}},\mathfrak{S}^{\textnormal{p}})$ for $\cR^{\textnormal{s}}, \cC^{\textnormal{s}}, \cS^{\textnormal{s}}$ and a set of even-length row, column, and symbol cycles \[
    \cY = \{\cY(r_1,r_2)\}_{(r_1,r_2)\in\mathfrak{R}^{\textnormal{p}}} \cup \{\cY(c_1,c_2)\}_{(c_1,c_2)\in\mathfrak{C}^{\textnormal{p}}} \cup \{\cY(s_1,s_2)\}_{(s_1,s_2)\in\mathfrak{S}^{\textnormal{p}}},
    \]
    where $\cY(r_1,r_2)$ for $(r_1,r_2)\in\mathfrak{R}^{\textnormal{p}}$ is a cycle in the permutation $\sigma_L(r_1,r_2)$, and similarly for $\mathfrak{C}^{\textnormal{p}},\mathfrak{S}^{\textnormal{p}}$, such that:
    \begin{enumerate}
        \item \label{def:even_short_cycles} each cycle in $\cY$ has length which is even and at most $k$;
        \item \label{def:disjoint_cycles} the cycles in $\cY$ are disjoint;\item \label{def:mate_repair_system_cycles_vs_rows_cols_symbs} all cycles $\cY(r_1,r_2)$ for $(r_1,r_2) \in \mathfrak{R}^{\textnormal{p}}$ avoid $\cC^{\textnormal{s}},\cS^{\textnormal{s}}$, and similarly all cycles $\cY(c_1,c_2)$ for $(c_1,c_2) \in \mathfrak{C}^{\textnormal{p}}$ avoid $\cR^{\textnormal{s}},\cS^{\textnormal{s}}$ and all cycles $\cY(s_1,s_2)$ for $(s_1,s_2) \in \mathfrak{S}^{\textnormal{p}}$ avoid $\cR^{\textnormal{s}},\cC^{\textnormal{s}}$.
    \end{enumerate}
\end{definition}

Not all Latin squares equipped with a mate system are flexible. Indeed, some pairs of rows, columns, or symbols might not induce any short even-length cycles at all. As a first step in our proof, we find a set of disjoint intercalates, such that switching these intercalates transforms a Latin square with a mate system into a flexible Latin square. We refer to such a set of intercalates as \emph{repairing}, defined below. We first define $T$-\emph{disjointness}, which is required for the definition of repairing intercalates. We would like the template $T$ we use to avoid the repairing intercalates, since it is important that switching them does not interfere with the parity generator, which will be contained in $T$.

\begin{definition}
For a Latin square $L\in \mc L_n$ and template $T\subseteq [n]^2$, we say that an intercalate $I$ in $L$ is \emph{$T$-disjoint} if its four entries are disjoint from $T\cap L.$
\end{definition}

\begin{definition}[A repairing set of intercalates]
    \label{def:repairing_intercalates}
    For a Latin square $L\in \cL_n$, a template $T\subseteq [n]^2$, a positive integer $k$, and a mate system $(\mathfrak{R}^{\textnormal{p}},\mathfrak{C}^{\textnormal{p}},\mathfrak{S}^{\textnormal{p}})$, we say that a set $\cI(\mathfrak{R}^{\textnormal{p}},\mathfrak{C}^{\textnormal{p}},\mathfrak{S}^{\textnormal{p}})$ of intercalates is $(T,k)$-\emph{repairing} if
    \begin{enumerate}
        \item \label{def:mate_repair_system_ints}
        the set $\cI(\mathfrak{R}^{\textnormal{p}},\mathfrak{C}^{\textnormal{p}},\mathfrak{S}^{\textnormal{p}})$ can be described as a disjoint union $\bigsqcup_{(z_1,z_2)\in\mathfrak{R}^{\textnormal{p}}\cup\mathfrak{C}^{\textnormal{p}}\cup\mathfrak{S}^{\textnormal{p}}}\cI(z_1,z_2)$, where for each pair $(z_1, z_2) \in \mathfrak{R}^{\textnormal{p}} \cup \mathfrak{C}^{\textnormal{p}} \cup \mathfrak{S}^{\textnormal{p}}$ the set $\cI(z_1,z_2)$ consists of at most two $T$-disjoint intercalates intersecting $z_2$;
        \item \label{def:mate_repair_system_ints_disjoint} all intercalates in $\cI(\mathfrak{R}^{\textnormal{p}},\mathfrak{C}^{\textnormal{p}},\mathfrak{S}^{\textnormal{p}})$ are disjoint;
        \item \label{def:mate_repair_system_ints_vs_rows_cols_symbs} no intercalate in $\cI (z_1,z_2)$ for $(z_1,z_2) \in \mathfrak{R}^{\textnormal{p}} \cup \mathfrak{C}^{\textnormal{p}} \cup \mathfrak{S}^{\textnormal{p}}$ intersects $\cR^{\textnormal{s}} \cup \cC^{\textnormal{s}} \cup \cS^{\textnormal{s}} \cup \cR^{\textnormal{m}} \cup \cC^{\textnormal{m}} \cup \cS^{\textnormal{m}} \setminus \{z_2\}$;
        \item \label{def:mate_repair_system_ints_create_cycles} the Latin square $\mr{sw}_{\cI(\mathfrak{R}^{\textnormal{p}},\mathfrak{C}^{\textnormal{p}},\mathfrak{S}^{\textnormal{p}})}(L)$ obtained from $L$ by
        switching all intercalates in $\cI(\mathfrak{R}^{\textnormal{p}}, \mathfrak{C}^{\textnormal{p}}, \mathfrak{S}^{\textnormal{p}})$ is $(\cR^{\textnormal{s}},\cC^{\textnormal{s}},\cS^{\textnormal{s}},k)$-flexible with mate system $(\mathfrak{R}^{\textnormal{p}},\mathfrak{C}^{\textnormal{p}},\mathfrak{S}^{\textnormal{p}})$ and cycles $\cY(z_1,z_2)$ for $(z_1,z_2) \in \mathfrak{R}^{\textnormal{p}} \cup \mathfrak{C}^{\textnormal{p}} \cup \mathfrak{S}^{\textnormal{p}}$.
    \end{enumerate}
    Note that all intercalates in $\cI(\mathfrak{R}^{\textnormal{p}},\mathfrak{C}^{\textnormal{p}},\mathfrak{S}^{\textnormal{p}})$ are disjoint due to part~(\ref{def:mate_repair_system_ints_disjoint}) of \cref{def:repairing_intercalates},
    and hence
    $\mr{sw}_{\cI(\mathfrak{R}^{\textnormal{p}},\mathfrak{C}^{\textnormal{p}},\mathfrak{S}^{\textnormal{p}})}(L)$ in part~(\ref{def:mate_repair_system_ints_create_cycles})  is well-defined.
\end{definition}

Given any Latin square $L$, the main technical challenge in our proof is to find a repairing set of intercalates that would turn $L$ into a flexible Latin square, as well as a parity generator that is disjoint from the repairing set of intercalates and the short even-length cycles in the flexible Latin square. This will allow us to fix the parities of both the stubborn \emph{and} the non-stubborn rows, columns, and symbols, without these two processes interfering with each other. The next result, which is our main technical lemma, guarantees that this is possible with the help of a template $\bfT$, which will contain the parity generator but will be disjoint from the repairing set of intercalates.

\begin{lemma}[Repairing intercalates and a parity generator]
\label{lemma:T-disjoint_chains_rows_cols_symbols}
Let $\bfT \subseteq [n]^2$ be a random subset of row/column pairs, where each is included with probability $\varepsilon = o(\log^{-1}{n})$ independently. Let $\bfL \sim \unif(\cL_n)$. With probability at least $ 1 - \exp(-\omega(n\log^2{n}))$, the following holds. 
\begin{itemize}
    \item There are sets $\cR^{\textnormal{s}},\cC^{\textnormal{s}},\cS^{\textnormal{s}} \subset [n]$ of stubborn rows, columns, and symbols respectively, of size at most $\ell = \log^{11}{n}$ each, along with a mate system $(\mathfrak{R}^{\textnormal{p}}, \mathfrak{C}^{\textnormal{p}}, \mathfrak{S}^{\textnormal{p}})$ for $\cR^{\textnormal{s}}, \cC^{\textnormal{s}}, \cS^{\textnormal{s}}$ and a set $\cI(\mathfrak{R}^{\textnormal{p}}, \mathfrak{C}^{\textnormal{p}}, \mathfrak{S}^{\textnormal{p}})$ of intercalates that is $(\bfT,n/\log^{100}{n})$-repairing with associated cycles $\cY(z_1,z_2)$ in $\mr{sw}_{\cI(\mathfrak{R}^{\textnormal{p}},\mathfrak{C}^{\textnormal{p}},\mathfrak{S}^{\textnormal{p}})}(\bfL)$ for $(z_1,z_2) \in \mathfrak{R}^{\textnormal{p}} \cup \mathfrak{C}^{\textnormal{p}} \cup \mathfrak{S}^{\textnormal{p}}$. 
    \item Furthermore, let $\cI^{\textnormal{s}}$ be the set of stable intercalates in $\bfT \cap \bfL = \bfT \cap \mr{sw}_{\cI(\mathfrak{R}^{\textnormal{p}},\mathfrak{C}^{\textnormal{p}},\mathfrak{S}^{\textnormal{p}})}(\bfL)$ disjoint from $\cR^{\textnormal{s}}, \cC^{\textnormal{s}}, \cS^{\textnormal{s}} ,$ and disjoint from all cycles $\cY(z_1,z_2)$ for $(z_1, z_2) \in \mathfrak{R}^{\textnormal{p}} \cup \mathfrak{C}^{\textnormal{p}} \cup \mathfrak{S}^{\textnormal{p}}$. Then $\cI^{\textnormal{s}}$ is a parity generator with respect to $\bfT \cap \mr{sw}_{\cI(\mathfrak{R}^{\textnormal{p}},\mathfrak{C}^{\textnormal{p}},\mathfrak{S}^{\textnormal{p}})}(\bfL), \cR^{\textnormal{f}}, \cC^{\textnormal{f}}, \cS^{\textnormal{f}}$, where $\cR^{\textnormal{f}}:=[n]\setminus \cR^{\textnormal{s}}$,\,\, $\cC^{\textnormal{f}}:=[n]\setminus \cC^{\textnormal{s}}$,\, and $\cS^{\textnormal{f}}:=[n]\setminus \cS^{\textnormal{s}}$.
\end{itemize}
\end{lemma}
We postpone the proof of \cref{lemma:T-disjoint_chains_rows_cols_symbols} to \cref{sec:technical_lemma}.

One more ingredient we need is the following theorem by Cavenagh and Wanless \cite{CW16} which shows that the four possible parities of $(N_{\textnormal{row}}(\bfL),N_{\textnormal{col}}(\bfL))$ have approximately the same probability of being attained. 
\begin{theorem}[\cite{CW16}]\label{thm:equidistribution-on-parities}
Let $\bfL$ be a uniformly random $n\times n$ Latin square. Then for any $r,c\in \{0,1\}$ we have
\[
\mathbb{P}[\, \big(N_{\textnormal{row}}(\bfL),\, N_{\textnormal{col}}(\bfL)\, \big)=(r,c)\!\!\! \pmod 2]=\frac{1}{4}+o(1).
\]
\end{theorem}

We are now ready to prove \Cref{thm:main} assuming \cref{lemma:T-disjoint_chains_rows_cols_symbols}. Starting from a Latin square $L$ with $N_{\mathrm{row}}(L)  = |\vec r|_1 \pmod 2$, $N_{\mathrm{col}}(L)  = |\vec c|_1 \pmod 2$, $N_{\mathrm{sym}}(L)  = |\vec s|_1 \pmod 2$, we perform a sequence of three steps of switchings to obtain a Latin square $L'$ with the prescribed parities of all rows, columns, and symbols. We first apply \cref{lemma:T-disjoint_chains_rows_cols_symbols} to prepare a repairing set of intercalates and a parity generator such that all switchings can be performed without interfering with each other. As a first step of the switching, we switch the intercalates in the repairing set to ensure that for each stubborn row, column, and symbol, there is a short even-length cycle in the permutation induced by it and its mate. We then set aside these short even-length cycles without switching them yet, since doing so would potentially change the set of stable intercalates and destroy our parity generator\footnote{Our short even-length cycle switches will be disjoint from the intercalates we need to switch in the parity generator, so these two types of switching operations do in principle commute. However, since the first type of switching can affect which intercalates are stable, the order of operations we describe here makes it much cleaner to reason about the number of switchings in both directions.}. Instead, as a second step of the switching, we switch some stable intercalates that are part of the parity generator in order to get the desired parities for all non-stubborn rows, columns, and symbols (taking into account the changes that we anticipate due to the short even-length cycle switchings). Finally, as a third step of the switching, we switch the short even-length cycles that are required to fix the parities of the stubborn rows, columns, and symbols.

\begin{proof}[Proof of \Cref{thm:main}]

The fact that \[\mathbb{P}\big[(\vec\xi_{\textnormal{row}}(\mbf L),\vec\xi_{\textnormal{col}}(\mbf L),\vec\xi_{\textnormal{sym}}(\mbf L))=(\vec r,\vec c, \vec s)\big] \leq 2^{-3n+o(n)}\] follows from \cite[Theorem 1.3(5)]{kwan2025parities}. It remains to show the matching lower bound.

\medskip\noindent\textit{Set-up.} Fix any $\vec r,\vec c,\vec s\in \mb F_2^n$ satisfying $|\vec r|_1+|\vec c|_1+|\vec s|_1\equiv f(n) \pmod 2$ and denote $\vec\eta=(\vec r, \vec c,\vec s)$.
Let $\bfT \subseteq [n]^2$ be a random subset of row/column pairs, where each is included with probability $\varepsilon = o(\log^{-1}{n})$ independently and let $\bfL \sim \unif(\cL_n)$. For $T\subseteq[n]^2$ and $L\in\cL_n$, we say that the pair $(T,L)$ is \emph{good} if the conclusion of \cref{lemma:T-disjoint_chains_rows_cols_symbols} holds for it. \cref{lemma:T-disjoint_chains_rows_cols_symbols} thus says that $(\bfT,\bfL)$ is good with probability at least $1-\exp(-\omega(n\log^2 n))$. By \cref{fact:joint_prob_vs_most_outcomes}(\ref{fact:joint_to_separate}) with probability at least $1-\exp(-\omega(n\log^2 n))$ the outcome $T$ of $\bfT$ is such that $(T,\bfL)$ is good with probability at least $1-\exp(-\omega(n\log^2 n))$.
For the rest of the proof let us fix such an outcome $T$ of $\bfT$. 

Next, also fix an outcome $L$ of $\bfL$ such that $(T,L)$ is good and such that $N_{\mathrm{row}}(L)=|\vec r|_1 \pmod 2$ and $N_{\mathrm{col}}(L)=|\vec c|_1 \pmod 2$; we know that $\bfL$ satisfies these two parity conditions simultaneously with probability $\frac{1}{4}+o(1)$ due to \cref{thm:equidistribution-on-parities}. Note that this also implies $N_{\mathrm{sym}}(L)=|\vec s|_1 \pmod 2$. Let $\cR^{\textnormal{s}},\cC^{\textnormal{s}},\cS^{\textnormal{s}} \subset [n]$ be the stubborn rows, columns, and symbols given by \cref{lemma:T-disjoint_chains_rows_cols_symbols}. Furthermore, let $(\mathfrak{R}^{\textnormal{p}}, \mathfrak{C}^{\textnormal{p}}, \mathfrak{S}^{\textnormal{p}})$ be the mate system for $\cR^{\textnormal{s}},\cC^{\textnormal{s}},\cS^{\textnormal{s}}$, and let $\cI(\mathfrak{R}^{\textnormal{p}}, \mathfrak{C}^{\textnormal{p}}, \mathfrak{S}^{\textnormal{p}})$ be the set of intercalates that is $(T,n/\log^{100}{n})$-repairing with associated cycles $\cY(z_1,z_2)$ for $(z_1,z_2) \in \mathfrak{R}^{\textnormal{p}} \cup \mathfrak{C}^{\textnormal{p}} \cup \mathfrak{S}^{\textnormal{p}}$ in $\mr{sw}_{\cI(\mathfrak{R}^{\textnormal{p}},\mathfrak{C}^{\textnormal{p}},\mathfrak{S}^{\textnormal{p}})}(L)$ guaranteed by the fact that $(T,L)$ is good.
Finally, let $\cI^{\textnormal{s}}$ be the set of stable intercalates in $T \cap \mr{sw}_{\cI(\mathfrak{R}^{\textnormal{p}},\mathfrak{C}^{\textnormal{p}},\mathfrak{S}^{\textnormal{p}})}(L)$ avoiding $\cR^{\textnormal{s}}, \cC^{\textnormal{s}}, \cS^{\textnormal{s}}, $ as well as avoiding all cycles $\cY(z_1,z_2)$ for $(z_1, z_2) \in \mathfrak{R}^{\textnormal{p}} \cup \mathfrak{C}^{\textnormal{p}} \cup \mathfrak{S}^{\textnormal{p}}$. Note that by \cref{lemma:T-disjoint_chains_rows_cols_symbols}, $\cI^{\textnormal{s}}$ is a parity generator with respect to $T \cap \mr{sw}_{\cI(\mathfrak{R}^{\textnormal{p}},\mathfrak{C}^{\textnormal{p}},\mathfrak{S}^{\textnormal{p}})}(L)$,\,\, $\cR^{\textnormal{f}} = [n]\setminus \cR^{\textnormal{s}}$,\,\, $\cC^{\textnormal{f}} = [n] \setminus\cC^{\textnormal{s}}$,\,\, $\cS^{\textnormal{f}} = [n] \setminus \cS^{\textnormal{s}}$. 

\medskip\noindent\textit{Switching procedure.} We would like to do some cycle switches so that all rows, columns, and symbols have their sequence of parities according to $\vec\eta=(\vec r, \vec c, \vec s)$. 
We first switch all the intercalates in $\cI(\mathfrak{R}^{\textnormal{p}},\mathfrak{C}^{\textnormal{p}},\mathfrak{S}^{\textnormal{p}})$. By the properties of $(T, n/\log^{100}{n})$-repairing sets of intercalates, for each $(z_1,z_2)\in \mathfrak{R}^{\textnormal{p}}\cup \mathfrak{C}^{\textnormal{p}}\cup \mathfrak{S}^{\textnormal{p}}$ there exists an even-length cycle $\cY(z_1,z_2)$ of length at most $n/\log^{100}n$ in $\sigma_{L'}(z_1,z_2)$ of $L':=\mr{sw}_{\cI(\mathfrak{R}^{\textnormal{p}}, \mathfrak{C}^{\textnormal{p}},\mathfrak{S}^{\textnormal{p}})}(L)$. Let $\cR^{\textnormal{s}}_{\textnormal{off}}$, $\cC^{\textnormal{s}}_{\textnormal{off}}$ and $\cS^{\textnormal{s}}_{\textnormal{off}}$ be the stubborn rows, stubborn columns, and stubborn symbols respectively whose parities in $L$ do not agree with $\vec\eta$ and let $(\mathfrak{R}^{\textnormal{p}}_{\textnormal{off}},\mathfrak{C}^{\textnormal{p}}_{\textnormal{off}},\mathfrak{S}^{\textnormal{p}}_{\textnormal{off}})$ be their corresponding mate system (which is a subset of $(\mathfrak{R}^{\textnormal{p}}, \mathfrak{C}^{\textnormal{p}}, \mathfrak{S}^{\textnormal{p}})$). Next, let $\cY_{\textnormal{off}} = \cY_{\textnormal{off}}(L,\cY,\cR^{\textnormal{s}},\cC^{\textnormal{s}},\cS^{\textnormal{s}}):=\bigcup_{(z_1,z_2)\in \mathfrak{R}^{\textnormal{p}}_{\textnormal{off}}\cup\mathfrak{C}^{\textnormal{p}}_{\textnormal{off}}\cup\mathfrak{S}^{\textnormal{p}}_{\textnormal{off}}}\cY(z_1,z_2)$ be the union of the cycles of all stubborn rows, columns, and symbols with parities not agreeing with $\vec\eta$.

Let $\vec\nu = \vec \nu (\cY_{\textnormal{off}})\in(\mb F_2^n)^3$ be the vector recording, for every individual row, column, and symbol, the parity of the number of cycles in $\cY_{\textnormal{off}}$ intersecting it. Then as $\cI^{\textnormal{s}}$ is a parity generator with respect to $T \cap \mr{sw}_{\cI(\mathfrak{R}^{\textnormal{p}},\mathfrak{C}^{\textnormal{p}},\mathfrak{S}^{\textnormal{p}})}(L),$ $ \cR^{\textnormal{f}}, \cC^{\textnormal{f}}, \cS^{\textnormal{f}}$, we can switch some intercalates from $\cI^{\textnormal{s}}$ such that all the rows in $\cR^{\textnormal{f}}$, the columns in $\cC^{\textnormal{f}}$, and the symbols in $\cS^{\textnormal{f}}$ obtain the parity prescribed by $\vec\eta + \vec\nu$. Note that since $\cI^{\textnormal{s}}$ avoids $\cR^{\textnormal{s}}, \cC^{\textnormal{s}}, \cS^{\textnormal{s}}, $ as well as all cycles $\cY(z_1,z_2)$ in $\mr{sw}_{\cI(\mathfrak{R}^{\textnormal{p}},\mathfrak{C}^{\textnormal{p}},\mathfrak{S}^{\textnormal{p}})}(L)$, this step does not change the parity of $\cR^{\textnormal{s}}, \cC^{\textnormal{s}}, \cS^{\textnormal{s}}$ and preserves the cycles $\cY_{\textnormal{off}}$. Finally, we switch all the cycles $\cY_{\textnormal{off}}$. By properties (\cref{def:disjoint_cycles}) and (\cref{def:mate_repair_system_cycles_vs_rows_cols_symbs}) of \cref{def:flexible_LS}, these cycles are pairwise disjoint so we can switch all of them in any order, and also no stubborn row, column, or symbol is ever affected by any cycle other than the one associated with it. Thus, this makes the parities of all rows, columns, and symbols agree with $\vec\eta$.

\medskip\noindent\textit{Defining the auxiliary switching graph.} To make this quantitative, we also need the following. Let $\cA_1$ be the set of all Latin squares $L_1$ such that $(T,L_1)$ is good, $N_{\mathrm{row}}(L_1)=|\vec r|_1 \pmod 2$ and $N_{\mathrm{col}}(L_1)=|\vec c|_1 \pmod 2$, and let $\cA_4$ be the set of Latin squares $L_4$ whose row, column and symbol parities are given by
$\vec r,\vec c,\vec s$
respectively. 
Consider the auxiliary bipartite graph between $\cA_1$ and $\cA_4$ where we add an edge between $L_1 \in \cA_1$ and $L_4 \in \cA_4$ if $L_4$ can be obtained from $L_1$ by the procedure described above. Since $\Pr[\bfL \in \cA_1] = \frac{1}{4} + o(1)$, if we were to show that $|\cA_4|/|\cA_1| \geq 2^{-3n + o(n)}$, the conclusion of our theorem would follow. We would like to compare the sizes of the sets $\cA_1$ and $\cA_4$ by a switching argument. However, since the degrees in the auxiliary bipartite graph can vary due to the potentially varying size of the set of stable intercalates $\cI^{\textnormal{s}}$ in the second step, we 
split the switching described above into three steps: 
\begin{enumerate} 
    \item \label{switch_step:1} firstly, we switch the $(T,n/\log^{100}n)$-repairing set $\cI(\mathfrak{R}^{\textnormal{p}},\mathfrak{C}^{\textnormal{p}},\mathfrak{S}^{\textnormal{p}})$ of intercalates to create the short even-length cycles in $\cY$;
    \item \label{switch_step:2} secondly, we switch some stable intercalates in $\cI^{\textnormal{s}}$ to obtain the correct parities of the non-stubborn rows, columns, and symbols (taking into account the planned cycle switches for the third step);
    \item \label{switch_step:3} and thirdly, we switch the even-length cycles in $\cY_{\textnormal{off}}$ that repair the parities of the stubborn rows, columns, and symbols.
\end{enumerate}
We represent this as a switching graph with layers $\cA_1,\cA_2,\cA_3,\cA_4$ and a bipartite subgraph between each pair of consecutive layers, where an edge between two vertices indicates that we can switch from one to the other in the procedure described above. In particular, the edges between $\cA_1$ and $\cA_2$ correspond to step (\cref{switch_step:1}), the edges between $\cA_2$ and $\cA_3$ correspond to step (\cref{switch_step:2}) and finally the edges between $\cA_3$ and $\cA_4$ correspond to step (\cref{switch_step:3}). We formally define the sets $\cA_2$ and $\cA_3$ as follows (note that throughout the switching, we need to keep track of the stubborn rows, mate system and even-length cycles, which are recorded in $\cA_2$ and $\cA_3$).

The set $\cA_2$ consists of $4$-tuples $(L_2, (\cR^{\textnormal{s}},\cC^{\textnormal{s}},\cS^{\textnormal{s}}),(\mathfrak{R}^{\textnormal{p}},\mathfrak{C}^{\textnormal{p}},\mathfrak{S}^{\textnormal{p}}),\cY)$ obtained in the following way. For each $L_1 \in \cA_1$, recall that $(T,L_1)$ is good, so the conclusion of \cref{lemma:T-disjoint_chains_rows_cols_symbols} holds for it. Let $\cR^{\textnormal{s}},\cC^{\textnormal{s}},\cS^{\textnormal{s}}$ be the stubborn rows, columns, and symbols given by \cref{lemma:T-disjoint_chains_rows_cols_symbols} for $L_1$, let $(\mathfrak{R}^{\textnormal{p}},\mathfrak{C}^{\textnormal{p}},\mathfrak{S}^{\textnormal{p}})$ be the mate system for $\cR^{\textnormal{s}},\cC^{\textnormal{s}},\cS^{\textnormal{s}}$ given by the lemma, and let $\cI(\mathfrak{R}^{\textnormal{p}},\mathfrak{C}^{\textnormal{p}},\mathfrak{S}^{\textnormal{p}})$ be the $(T,n/\log^{100}{n})$-repairing set of intercalates with associated cycles $\cY$ given by the lemma. Let $L_2 := \mr{sw}_{\cI(\mathfrak{R}^{\textnormal{p}},\mathfrak{C}^{\textnormal{p}},\mathfrak{S}^{\textnormal{p}})}(L_1)$ be the Latin square obtained by applying step (\ref{switch_step:1}) of the switching procedure to $L_1$. We then let $(L_2, (\cR^{\textnormal{s}},\cC^{\textnormal{s}},\cS^{\textnormal{s}}),(\mathfrak{R}^{\textnormal{p}},\mathfrak{C}^{\textnormal{p}},\mathfrak{S}^{\textnormal{p}}),\cY) \in \cA_2$ and we add an edge in the switching graph between $L_1$ and $(L_2, (\cR^{\textnormal{s}},\cC^{\textnormal{s}},\cS^{\textnormal{s}}),(\mathfrak{R}^{\textnormal{p}},\mathfrak{C}^{\textnormal{p}},\mathfrak{S}^{\textnormal{p}}),\cY)$. All members $(L_2, (\cR^{\textnormal{s}},\cC^{\textnormal{s}},\cS^{\textnormal{s}}),(\mathfrak{R}^{\textnormal{p}},\mathfrak{C}^{\textnormal{p}},\mathfrak{S}^{\textnormal{p}}),\cY)$ of $\cA_2$ are obtained this way, which is why they have the following properties:
\begin{itemize}
    \item  $(\mathfrak{R}^{\textnormal{p}},\mathfrak{C}^{\textnormal{p}},\mathfrak{S}^{\textnormal{p}})$ is a mate system for $\cR^{\textnormal{s}}, \cC^{\textnormal{s}},\cS^{\textnormal{s}}$;
    \item $L_2$ is a $(\cR^{\textnormal{s}},\cC^{\textnormal{s}},\cS^{\textnormal{s}},n/\log^{100}{n})$-flexible Latin square with mate system  $(\mathfrak{R}^{\textnormal{p}},\mathfrak{C}^{\textnormal{p}},\mathfrak{S}^{\textnormal{p}})$ and set of cycles $\cY$; 
    \item $N_{\mathrm{row}}(L_2)=|\vec r|_1 \pmod 2$ and $N_{\mathrm{col}}(L_2)=|\vec c|_1 \pmod 2$;
    \item $\cR^{\textnormal{s}},\cC^{\textnormal{s}},\cS^{\textnormal{s}}$ each have size at most $\ell = \log^{11}{n}$ and have the property that the set $\cI^{\textnormal{s}}$ of stable intercalates in $T\cap L_2$ disjoint from $\cR^{\textnormal{s}}, \cC^{\textnormal{s}},\cS^{\textnormal{s}},$ and disjoint from all cycles in $\cY$ is a parity generator with respect to $T\cap L_2, \cR^{\textnormal{f}}, \cC^{\textnormal{f}},\cS^{\textnormal{f}}$, where $\cR^{\textnormal{f}}:=[n]\setminus \cR^{\textnormal{s}}$,\,\, $\cC^{\textnormal{f}}:=[n]\setminus \cC^{\textnormal{s}}$,\,\, $\cS^{\textnormal{f}}:=[n]\setminus \cS^{\textnormal{s}}$.
\end{itemize}

Similarly, let $\cA_3$ be the set of 4-tuples $(L_3,(\cR^{\textnormal{s}},\cC^{\textnormal{s}},\cS^{\textnormal{s}}),(\mathfrak{R}^{\textnormal{p}},\mathfrak{C}^{\textnormal{p}},\mathfrak{S}^{\textnormal{p}}),\cY)$ with the following property: there is some $(L_2,(\cR^{\textnormal{s}},\cC^{\textnormal{s}},\cS^{\textnormal{s}}),(\mathfrak{R}^{\textnormal{p}},\mathfrak{C}^{\textnormal{p}},\mathfrak{S}^{\textnormal{p}}),\cY) \in \cA_2$ such that applying step (\ref{switch_step:2}) of our switching procedure to $L_2$ yields $L_3$. That is, if $\vec \nu = \vec \nu (\cY_{\textnormal{off}}) $ for $\cY_{\textnormal{off}} = \cY_{\textnormal{off}}(L_2, \cY, \cR^s,\cC^s,\cS^s)$, then switching some of the stable intercalates $\cI^s$ in $T\cap L_2$ disjoint from $\cR^s,\cC^s,\cS^s$ and disjoint from all the cycles in $\cY$ to obtain row, column, and symbol parities prescribed by $\vec \eta + \vec \nu$ results in $L_3$. Thus, the row, column, and symbol parities of any $L_3$ that is part of a $4$-tuple in $\cA_3$ agree with $\vec \eta + \vec \nu$. We represent this by an edge in the switching graph between $(L_2,(\cR^{\textnormal{s}},\cC^{\textnormal{s}},\cS^{\textnormal{s}}),(\mathfrak{R}^{\textnormal{p}},\mathfrak{C}^{\textnormal{p}},\mathfrak{S}^{\textnormal{p}}),\cY)$ and $(L_3,(\cR^{\textnormal{s}},\cC^{\textnormal{s}},\cS^{\textnormal{s}}),(\mathfrak{R}^{\textnormal{p}},\mathfrak{C}^{\textnormal{p}},\mathfrak{S}^{\textnormal{p}}),\cY)$.

Finally, given a $4$-tuple $(L_3,(\cR^{\textnormal{s}},\cC^{\textnormal{s}},\cS^{\textnormal{s}}),(\mathfrak{R}^{\textnormal{p}},\mathfrak{C}^{\textnormal{p}},\mathfrak{S}^{\textnormal{p}}),\cY) \in \cA_3$, if we apply step~(\cref{switch_step:3}) of our switching procedure by switching the cycles in $\cY_{\textnormal{off}} =  \cY_{\textnormal{off}}(L_3,\cY,\cR^{\textnormal{s}},\cC^{\textnormal{s}},\cS^{\textnormal{s}})$, we get a Latin square $L_4$ whose row, column, and symbol parities agree with $\vec r, \vec c, \vec s$, so $L_4 \in \cA_4$. We represent this by an edge in the switching graph between $(L_3,(\cR^{\textnormal{s}},\cC^{\textnormal{s}},\cS^{\textnormal{s}}),(\mathfrak{R}^{\textnormal{p}},\mathfrak{C}^{\textnormal{p}},\mathfrak{S}^{\textnormal{p}}),\cY)$ and $L_4$.

Then our switching procedure takes a Latin square $L_1\in \cA_1$ and applies the three steps to it, first switching from $\cA_1$ to $\cA_2$ (step~(\ref{switch_step:1})), then from $\cA_2$ to $\cA_3$ (step~(\ref{switch_step:2})), and finally from $\cA_3$ to a Latin square $L_4\in \cA_4$ (step~(\ref{switch_step:3})) with row, column, and symbol parities that agree with $\vec\eta.$ Instead of comparing the relative sizes of $\cA_1$ and $\cA_2$ directly, we compare the relative sizes of the two sets involved in each step of the switching, decomposing
\[\frac{|\cA_4|}{|\cA_1|}=\frac{|\cA_2|}{|\cA_1|}\cdot \frac{|\cA_3|}{|\cA_2|}\cdot \frac{|\cA_4|}{|\cA_3|}.\]

\begin{figure}[!hbt]
\centering
\includegraphics{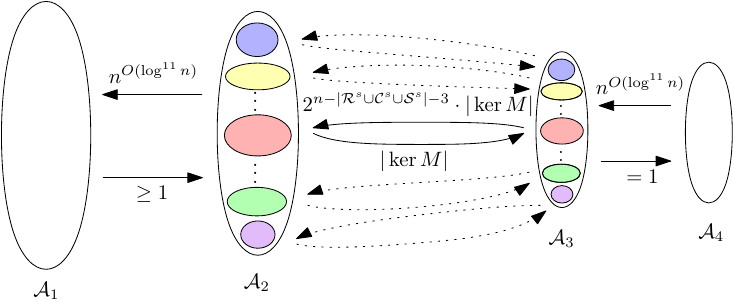}
\caption{Visualisation of the switching procedure. The arrows between consecutive sets with formulae correspond to the bounds on the degrees we prove. In the `middle' switching between $\cA_2$ and $\cA_3$, we consider each connected component of the bipartite graph independently (this is visualised by the different colours). In the picture we use $M$ to denote the matrix $M(\cR^{\textnormal{f}},\cC^{\textnormal{f}},\cS^{\textnormal{f}},\cI^{\textnormal{s}})$.}
\end{figure}

\medskip\noindent\textit{Analysis of the switching procedure.} Firstly, consider the switching between $\cA_1$ and $\cA_2$. By \cref{lemma:T-disjoint_chains_rows_cols_symbols}, each vertex in $\cA_1$ has degree at least $1$ to $\cA_2$. On the other hand, given $(L_2,(\cR^{\textnormal{s}},\cC^{\textnormal{s}},\cS^{\textnormal{s}}),(\mathfrak{R}^{\textnormal{p}},\mathfrak{C}^{\textnormal{p}},\mathfrak{S}^{\textnormal{p}}),\cY)\in \cA_2$, in order to obtain a neighbour of it $L_1 \in \cA_1$, we need to reverse the $(T,n/\log^{100}n)$-repairing intercalate switches. So for each $(z_1,z_2)\in \mathfrak{R}^{\textnormal{p}}\cup\mathfrak{C}^{\textnormal{p}}\cup\mathfrak{S}^{\textnormal{p}}$, we need to define $\cI(z_1,z_2) \in \cI(\mathfrak{R}^{\textnormal{p}},\mathfrak{C}^{\textnormal{p}},\mathfrak{S}^{\textnormal{p}})$. The set $\cI(z_1,z_2)$ consists of at most $2$ intercalates intersecting $z_2$, and an intercalate in $L_2$ intersecting $z_2$ (or lack thereof) can be chosen in at most $n\cdot\binom{n}{2} + 1\leq n^3$ many ways. Hence, the degree to $\cA_1$ of each vertex in $\cA_2$ is at most $ \left(n^{6}\right)^{3\log^{11}n}=n^{O(\log^{11} n)}$ as $|\mathfrak{R}^{\textnormal{p}}\cup\mathfrak{C}^{\textnormal{p}}\cup\mathfrak{S}^{\textnormal{p}}|\leq 3\log^{11}n$. Therefore, 
\[\frac{|\cA_2|}{|\cA_1|}\geq \frac{1}{n^{O(\log^{11} n)}}.\]

Next, consider the final switching between $\cA_3$ and $\cA_4.$ Each vertex in $\cA_3$ has degree to $\cA_4$ exactly $1$, as switching the prescribed even-length cycles in $\cY_{\textnormal{off}}$ produces a unique Latin square whose row, column, and symbol parities agree with $\vec\eta$. For a Latin square $L_4\in \cA_4$ whose parities agree with $\vec\eta$, to reverse the switching, we need to decide on the stubborn rows, columns, and symbols, assign a mate and an even-length cycle to each of them, and decide which of the stubborn rows, columns, and symbols were originally in $\cR^{\textnormal{s}}_{\textnormal{off}}\cup\cC^{\textnormal{s}}_{\textnormal{off}}\cup\cS^{\textnormal{s}}_{\textnormal{off}}$ in order to determine the cycles $\cY_{\textnormal{off}}$ to be switched. We can choose the stubborn rows, columns, and symbols in $\left(\sum_{i=0}^{\log^{11}n}\binom{n}{i}\right)^3\leq 2n^{3\log^{11}n}$ many different ways. For each such stubborn row, column, or symbol $z_1$, we need to pick its mate $z_2$, for which we have at most $n$ choices, and a cycle in $\sigma_{L_4}(z_1,z_2)$, for which we also have at most $n$ choices. Finally, we determine which of these were originally in $\cR^{\textnormal{s}}_{\textnormal{off}}\cup\cC^{\textnormal{s}}_{\textnormal{off}}\cup\cS^{\textnormal{s}}_{\textnormal{off}}$, for which we have at most $2^{3\log^{11}n}$ choices. Therefore, each vertex in $\cA_4$ has degree at most $ 2n^{3\log^{11}n}\cdot n^{6\log^{11}n}\cdot 2^{3\log^{11}n}= n^{O(\log^{11}n)}$ to $\cA_3.$ Hence, we get \[\frac{|\cA_4|}{|\cA_3|}\geq \frac{1}{n^{O(\log^{11}n)}}.\]

Finally, let us look at the switching between $\cA_2$ and $\cA_3.$ For $(L_2, (\cR^{\textnormal{s}},\cC^{\textnormal{s}},\cS^{\textnormal{s}}), (\mathfrak{R}^{\textnormal{p}},\mathfrak{C}^{\textnormal{p}},\mathfrak{S}^{\textnormal{p}}), \cY )\in \cA_2$, let $E$ consist of all entries in cycles $\cY(z_1,z_2)$ for $(z_1,z_2) \in \mathfrak{R}^{\textnormal{p}} \cup \mathfrak{C}^{\textnormal{p}} \cup \mathfrak{S}^{\textnormal{p}}$. 
Recall that the intercalates $\cI^{\textnormal{s}}$ that we consider for this switching are the stable intercalates in $L_2\cap T$ that avoid $\cR^{\textnormal{s}}\cup\cC^{\textnormal{s}}\cup\cS^{\textnormal{s}}$ and do not intersect $E$. We may notice that if two vertices are connected by an edge in the switching graph between $\cA_2$ and $\cA_3$, they have the same set of stable intercalates $\cI^{\textnormal{s}}$ (in the sense of the set $\cS(L\cap T)$, disregarding the ``switched state'' of each intercalate as in \cref{def:intercalate-abstraction}, and with any intercalates that intersect $(\cR^{\textnormal{s}}\cup\cC^{\textnormal{s}}\cup\cS^{\textnormal{s}}\cup E)$ removed). Thus, adjacent vertices in the layer of the switching graph between $\cA_2$ and $\cA_3$ have the same matrix $M(\cR^{\textnormal{f}},\cC^{\textnormal{f}},\cS^{\textnormal{f}},\cI^{\textnormal{s}})$, where $\cR^{\textnormal{f}} = [n] \setminus \cR^{\textnormal{s}}$ and similarly for $\cC^{\textnormal{f}}$ and $\cS^{\textnormal{f}}$. Hence, if we consider the connected components of the bipartite graph between $\cA_2$ and $\cA_3$, all vertices in the same connected component have the same matrix $M(\cR^{\textnormal{f}},\cC^{\textnormal{f}},\cS^{\textnormal{f}},\cI^{\textnormal{s}})$. 

Consider a connected component $C$ and denote $C_2=C\cap \cA_2$ and $C_{3}=C\cap \cA_3$. For a vertex $(L_2,(\cR^{\textnormal{s}},\cC^{\textnormal{s}},\cS^{\textnormal{s}}),(\mathfrak{R}^{\textnormal{p}},\mathfrak{C}^{\textnormal{p}},\mathfrak{S}^{\textnormal{p}}),\cY)\in C_2$, its degree to $C_3$ in the switching graph is $|\ker M(\cR^{\textnormal{f}},\cC^{\textnormal{f}},\cS^{\textnormal{f}},\cI^{\textnormal{s}})|$. For a vertex  $(L_3,(\cR^{\textnormal{s}},\cC^{\textnormal{s}},\cS^{\textnormal{s}}),(\mathfrak{R}^{\textnormal{p}},\mathfrak{C}^{\textnormal{p}},\mathfrak{S}^{\textnormal{p}}),\cY)\in C_3$ in the same connected component $C$, we know its degree to $C_2$ is at most \[2^{3n-|\cR^{\textnormal{s}}|-|\cC^{\textnormal{s}}|-|\cS^{\textnormal{s}}|-3}\cdot|\ker M(\cR^{\textnormal{f}},\cC^{\textnormal{f}},\cS^{\textnormal{f}},\cI^{\textnormal{s}})|.\]
This is because for each $(L_3,(\cR^{\textnormal{s}},\cC^{\textnormal{s}},\cS^{\textnormal{s}}),(\mathfrak{R}^{\textnormal{p}},\mathfrak{C}^{\textnormal{p}},\mathfrak{S}^{\textnormal{p}}),\cY)\in C_3$, we have $2^{3n-|\cR^{\textnormal{s}}|-|\cC^{\textnormal{s}}|-|\cS^{\textnormal{s}}|-3}$ possibilities for the parities of the non-stubborn rows, columns, and symbols in its neighbour $L_2$ such that $N_{\mathrm{row}}(L_2)=|\vec{r}|_1\pmod 2$, $N_{\mathrm{col}}(L_2)=|\vec{c}|_1\pmod 2$ and $N_{\mathrm{sym}}(L_2)=|\vec{s}|_1\pmod 2$ (recall that by \cref{remark:matrix_interpretation}, intercalate switches cannot alter the parity of the number of odd rows, the number of odd columns, or the number of odd symbols). And for each such possible choice, there are at most $|\ker M(\cR^{\textnormal{f}},\cC^{\textnormal{f}},\cS^{\textnormal{f}},\cI^{\textnormal{s}})|$ many ways to switch backwards from $L_3$ to achieve these parities. 
Hence, \[\frac{|C_{3}|}{|C_{2}|}\geq\frac{1}{2^{3n-|\cR^{\textnormal{s}}|-|\cC^{\textnormal{s}}|-|\cS^{\textnormal{s}}|-3}}\geq \frac{1}{2^{3n}}.\] As this holds for each connected component $C$, we also get 
\[\frac{|\cA_3|}{|\cA_2|}\geq\frac{1}{2^{3n}}.\]
Altogether, we get \[\frac{|\cA_4|}{|\cA_1|}=\frac{|\cA_2|}{|\cA_1|}\cdot \frac{|\cA_3|}{|\cA_2|}\cdot \frac{|\cA_4|}{|\cA_3|}\geq \frac{1}{n^{O(\log^{11} n)}} \cdot \frac{1}{2^{3n}} \cdot \frac{1}{n^{O(\log^{11}n)}}=2^{-3n+o(n)}, \] as desired.
\end{proof}

\section{Proof of the main technical lemma}
\label{sec:technical_lemma}
In this section, we give the proof of \cref{lemma:T-disjoint_chains_rows_cols_symbols}. This involves an application of \cref{lem:few_stubborn_rows_cols_symbols} to find a robust parity generator, as well as another helper lemma used to guarantee the existence of a repairing set of intercalates. We first need to define row, column, and symbol \emph{repair systems} which specify a set of candidate mates, together with a set of candidate repairing intercalates, for a given row, column, or symbol. We will need many candidate mates for each stubborn row/column/symbol, so that we can greedily build a mate system in which all the mates, and their accompanying repairing intercalates, are disjoint, and additionally the repairing intercalates are $\bfT$-disjoint.

\begin{definition}[Repair systems]
\label{def:repair_system}
    For a $1\times n$ Latin rectangle $D$, a Latin square $L$, a set of columns $\tilde{\cC}$, and a set of symbols $\tilde{\cS}$, a \emph{$(\tilde{\cC},\tilde{\cS})$-avoiding row $k$-repair-system} $(\cF,\cX,\cI)_D$ consists of 
    \begin{enumerate}
        \item \label{def:repair_system_rows} a set $\cF$ of rows;
        \item \label{def:repair_system_row_sets} for each $r \in \cF$, a set $\cX(r)$ of rows with $0 \leq |\cX(r)|\leq 2$ such that the sets $\mc X(r)$ are disjoint from each other and from $\mc F$;
        \item \label{def:repair_system_cycles} for each $r \in \cF$, a set $\cI(r)= \{I(r,r')\}_{r' \in \cX(r)}$ of at most two disjoint intercalates such that the intercalate $I(r,r')$ for each $r' \in \cX(r)$ intersects $r$ and $r'$ and avoids $\tilde{\cC} \cup \tilde{\cS}$, and such that the permutation $\sigma(D,L[r])$ in $\mr{sw}_{\cI(r)}(L)$ has an even-length cycle of length at most $k$ that avoids $\tilde{\cC}\cup\tilde{\cS}$.
    \end{enumerate}
    A $(\tilde{\cR},\tilde{\cS})$-avoiding column $k$-repair-system and a $(\tilde{\cR},\tilde{\cC})$-avoiding symbol $k$-repair-system are defined similarly.
\end{definition}

The aforementioned helper lemma states that such repair systems exist in most Latin squares for all choices of sets $\tilde{\cR},\tilde{\cC},\tilde{\cS}$, provided that these sets are not too large. 

\begin{lemma}
\label{lem:chains}
For a uniformly random Latin square $\bfL\sim \unif(\cL_n)$, the following holds with probability at least $ 1 - \exp(-\omega(n^2/\log^{2000}{n}))$. For every choice of $1\times n$ Latin rectangle $D$, set $\tilde{\cC}$ of forbidden columns, and set $\tilde{\cS}$ of forbidden symbols with $|\tilde{\cC}|,|\tilde{\cS}|\leq \log^{12}n$, there is a $(\tilde{\cC},\tilde{\cS})$-avoiding row $n/\log^{100}{n}$-repair-system $(\cF, \cX, \cI)_D$ with $|\cF| = n/\log^{1000}{n}$.

\end{lemma}

We now show how to deduce our main technical lemma (\cref{lemma:T-disjoint_chains_rows_cols_symbols}) from \cref{lem:chains}.
The proof is algorithmic: we go through the stubborn rows one-by-one, greedily selecting a mate for each, while taking care to avoid the already-selected mates (with their intercalates and cycles) and the entries of $\mbf T$. For this to be possible, we first need to use the randomness of $\mbf T$ to establish that typically each stubborn row/column/symbol has many candidate mates whose associated intercalates avoid $\mbf T$ (even when we forbid small sets of rows, columns and symbols $\cR^{\textnormal{a}},\cC^{\textnormal{a}},\cS^{\textnormal{a}}$).

\begin{proof}[Proof of \cref{lemma:T-disjoint_chains_rows_cols_symbols}]
\medskip\noindent\textit{Set-up.}
For $\bfL \sim \unif(\cL_n)$, the conclusion of  \cref{lem:chains} holds with probability at least $1 - \exp(-\omega(n\log^2{n}))$. That is, for every choice of a $1\times n$ Latin rectangle~$D$ and sets $\cC^{\textnormal{a}}$ and $\cS^{\textnormal{a}}$ of forbidden columns and symbols with $|\cC^{\textnormal{a}}|,|\cS^{\textnormal{a}}|\leq \log^{12}n$, there exists a $(\cC^{\textnormal{a}},\cS^{\textnormal{a}})$-avoiding row $n/\log^{100}{n}$-repair-system $(\cF, \cX, \cI)_D$ with $|\cF| = n/\log^{1000}{n}$. By the symmetry between rows, columns, and symbols, analogous statements with columns or symbols playing the special role of rows hold as well with the same probability. Let $\cL_n^{\textnormal{cycles}}$ be the set of Latin squares for which these three statements hold. 

\medskip\noindent\textit{Repair systems with $\bfT$-disjoint intercalates.}
We would now like to show that $\bfT$ is with high probability such that for each row, column, or symbol~$z$, many of the members of $\cI$ for $(\cF,\cX,\cI)_{\bfL[z]}$ are in fact disjoint from $\bfT$. We need this to find a mate system with a set of intercalates that is $(\bfT,n/\log^{100}{n})$-repairing. However, we would like the mates we choose to avoid the stubborn rows, columns, and symbols, as well as some additional rows, columns, and symbols, and these depend on the outcome of $\bfT$. Thus, we will need to take the union bound over the choice of avoided rows, columns, and symbols.

\begin{claim}
\label{claim:many-T-disjoint-chains}
    For each $L \in \cL_n^{\textnormal{cycles}}$, let $\bfT \subseteq [n]^2$ be a random subset of $[n] \times [n]$, where each is included with probability $\varepsilon = o(\log^{-1}{n})$ independently. Then with probability $1 - \exp(-\omega(n\log^2{n}))$, $\bfT$ is such that for every choice $\cR^{\textnormal{a}},\cC^{\textnormal{a}},\cS^{\textnormal{a}}$ of at most $\log^{12}{n}$ many rows, columns, and symbols, the following holds.
    \begin{enumerate}
        \item \label{claim:rows} For each row $r$, there is a $(\cC^{\textnormal{a}},\cS^{\textnormal{a}})$-avoiding row $n/\log^{100}{n}$-repair-system \[\cJ(r):= (\cF, \cX, \cI)_{L[r]}\] with $|\cF| = n/(2\log^{1000}{n})$ such that all intercalates in $\cI$ are $\bfT$-disjoint.
        \item \label{claim:columns} For each column $c$, there is a $(\cR^{\textnormal{a}},\cS^{\textnormal{a}})$-avoiding column $n/\log^{100}{n}$-repair-system \[\cJ(c):= (\cF, \cX, \cI)_{L[c]}\] with $|\cF| = n/(2\log^{1000}{n})$ such that all intercalates in $\cI$ are $\bfT$-disjoint.
        \item \label{claim:symbols} For each symbol $s$, there is a $(\cR^{\textnormal{a}},\cC^{\textnormal{a}})$-avoiding symbol $n/\log^{100}{n}$-repair-system \[\cJ(s):= (\cF, \cX, \cI)_{L[s]}\] with $|\cF| = n/(2\log^{1000}{n})$ such that all intercalates in $\cI$ are $\bfT$-disjoint.
    \end{enumerate}
\end{claim}
\begin{claimproof}
We will show (\cref{claim:rows}), and the other two parts will follow by symmetry between the rows, columns, and symbols.

For (\cref{claim:rows}), we will take a union bound over the $\binom{n}{\log^{12}{n}}^3$ choices of $\cR^{\textnormal{a}}, \cC^{\textnormal{a}}, \cS^{\textnormal{a}}$, so fix one such choice. Since $L \in \cL_n^{\textnormal{cycles}}$, for each row $r$, 
we have a $(\cC^{\textnormal{a}},\cS^{\textnormal{a}})$-avoiding row $n/\log^{100}{n}$-repair-system $\cJ'(r)= (\cF', \cX', \cI')_{L[r]}$ with $|\cF'| = n/\log^{1000}{n}$. For each row $r_2 \in \cF'$, there are at most two associated intercalates in $\cI'(r_2)$, and the probability that all of their entries are disjoint from $\bfT$ is at least $(1 - o(1/\log{n}))^8 =1 - o(1/\log{n})$. Note that all rows in $\cF' \cup \bigcup_{r_2\in\cF'}\cX(r_2)$ are disjoint, so by the Chernoff bound the probability that fewer than half of the rows in $\cF'$ have $\bfT$-disjoint intercalates is $\exp(-\Omega(n/\log^{1000}{n}))$. A union bound over all the rows $r$ in $L$, as well as over the choices of $\cR^{\textnormal{a}}, \cC^{\textnormal{a}}, \cS^{\textnormal{a}}$, concludes the proof.
\end{claimproof}

\medskip\noindent\textit{Most outcomes of $L$ and $T$ are suitable.} Recall that by \cref{lem:chains} we have $\Pr[\bfL\in \cL_n^{\textnormal{cycles}}]\geq 1 - \exp(-\omega(n\log^2{n}))$. Thus, by \cref{claim:many-T-disjoint-chains} and \cref{fact:joint_prob_vs_most_outcomes}(\ref{fact:separate_to_joint}), with probability at least $1 - \exp(-\omega(n\log^2{n}))$, $\bfL$ and $\bfT$ are such that for every choice of $\cR^{\textnormal{a}}, \cC^{\textnormal{a}}, \cS^{\textnormal{a}}$, parts (\ref{claim:rows}), (\ref{claim:columns}), and (\ref{claim:symbols}) of \cref{claim:many-T-disjoint-chains} hold. Furthermore, by \cref{lem:few_stubborn_rows_cols_symbols}, with probability $1 - \exp(-\omega(n\log^2{n}))$, $\bfL$ and $\bfT$ are such that there are sets of stubborn rows $\cR^{\textnormal{s}}$, stubborn columns $\cC^{\textnormal{s}}$, and stubborn symbols $\cS^{\textnormal{s}}$ of size at most $\log^{11}{n}$ each, for which the set $\cI^{\textnormal{stable}}$ of $\cR^{\textnormal{s}}\cup \cC^{\textnormal{s}} \cup \cS^{\textnormal{s}}$-avoiding stable intercalates in $\bfT \cap \bfL$ is a $n/(100\log^{11}{n})$-robust parity generator with respect to $\bfT \cap \bfL$, $\cR^{\textnormal{f}}, \cC^{\textnormal{f}}, \cS^{\textnormal{f}}$, where $\cR^{\textnormal{f}} := [n] \setminus \cR^{\textnormal{s}}, \cC^{\textnormal{f}} := [n] \setminus \cC^{\textnormal{s}}, \cS^{\textnormal{f}} := [n] \setminus \cS^{\textnormal{s}}$. From now on, we assume that $L$ and $T$ are outcomes of $\bfL$ and $\bfT$ such that both of these conclusions hold, which happens with probability at least $1 - \exp(-\omega(n\log^2{n}))$.

The rest of the proof is deterministic. We first construct the mate system $(\mathfrak{R}^{\textnormal{p}},\mathfrak{C}^{\textnormal{p}},\mathfrak{S}^{\textnormal{p}})$ for $\cR^{\textnormal{s}}, \cC^{\textnormal{s}}, \cS^{\textnormal{s}}$, along with a set $\cI(\mathfrak{R}^{\textnormal{p}},\mathfrak{C}^{\textnormal{p}},\mathfrak{S}^{\textnormal{p}})$ of $(T, n/\log^{100}{n})$-repairing intercalates. We pick a mate for each $r \in \cR^{\textnormal{s}}$ one-by-one, then a mate for each $c \in \cC^{\textnormal{s}}$, and finally a mate for each $s \in \cS^{\textnormal{s}}$.

\medskip\noindent\textit{Picking mates for the stubborn rows.}
Suppose we have already chosen a mate $r_2$ for each $r_1 \in \tilde{\cR}^{\textnormal{s}} \subset \cR^{\textnormal{s}}$ along with associated $T$-disjoint intercalates $\cI(r_1,r_2)$ that will be part of the $(T,n/\log^{100}{n})$-repairing set $\cI(\mathfrak{R}^{\textnormal{p}},\mathfrak{C}^{\textnormal{p}},\mathfrak{S}^{\textnormal{p}})$. Let $r \in \cR^{\textnormal{s}} \setminus \tilde{\cR}^{\textnormal{s}}$. By part (\ref{claim:rows}) of \cref{claim:many-T-disjoint-chains}, there is a $(\cC^{\textnormal{s}},\cS^{\textnormal{s}})$-avoiding row $n/\log^{100}{n}$-repair-system $\cJ(r)= (\cF, \cX, \cI)_{L[r]}$ with $|\cF| = n/(2\log^{1000}{n})$ such that all intercalates in $\cI$ are $T$-disjoint. Since $|\tilde{\cR^{\textnormal{s}}}| \leq |\cR^{\textnormal{s}}|\leq \log^{11}{n}$, there is a set $\tilde{\cR}^{\textnormal{a}}$ of at most $4 \log^{11}{n}$ rows that include $\cR^{\textnormal{s}}$, the mate $r_2$ of each row $r_1 \in \tilde{\cR}^{\textnormal{s}}$ and the rows intersecting the intercalates $\cI(r_1,r_2)$. Thus, we can find a mate $r' \in \cF$ for $r$ such that $r'$ and its associated rows $\cX(r')$ are disjoint from all rows in $\tilde{\cR}^{\textnormal{a}}$. We then take the associated intercalates $\cI(r')$ of $r'$ in $\cJ(r)$ as the set $\cI(r,r') \subset \cI(\mathfrak{R}^{\textnormal{p}},\mathfrak{C}^{\textnormal{p}},\mathfrak{S}^{\textnormal{p}})$.

\medskip\noindent\textit{Picking mates for the stubborn columns.}
Now suppose we have already chosen $\mathfrak{R}^{\textnormal{p}}$ and all associated intercalates $\cI(r_1,r_2) \subset \cI(\mathfrak{R}^{\textnormal{p}},\mathfrak{C}^{\textnormal{p}},\mathfrak{S}^{\textnormal{p}})$ for $(r_1,r_2) \in \mathfrak{R}^{\textnormal{p}}$. Suppose furthermore we have chosen a mate $c_2$ for each $c_1 \in \tilde{\cC}^{\textnormal{s}} \subset \cC^{\textnormal{s}}$ along with the associated intercalates $\cI(c_1, c_2) \subset \cI(\mathfrak{R}^{\textnormal{p}},\mathfrak{C}^{\textnormal{p}},\mathfrak{S}^{\textnormal{p}})$. Let $c \in \cC^{\textnormal{s}} \setminus \tilde{\cC}^{\textnormal{s}}$. Let $\cR^{\textnormal{a}}:= \cR^{\textnormal{s}} \cup \cR^{\textnormal{m}} \cup \cR^{\textnormal{i}}$, where $\cR^{\textnormal{i}}$ is the set of all rows intersecting an intercalate in $\cI(r_1,r_2)$ for $(r_1,r_2) \in \mathfrak{R}^{\textnormal{p}}$. Additionally, let $\cC^{\textnormal{a}} $ be the set of all columns intersecting an intercalate in $\cI(r_1,r_2)$ for $(r_1,r_2) \in \mathfrak{R}^{\textnormal{p}}$ and note that $|\cC^{\textnormal{a}}|\leq 4|\cR^{\textnormal{s}}| \leq 4\log^{11}{n}$. By part (\ref{claim:columns}) of \cref{claim:many-T-disjoint-chains}, there is a $(\cR^{\textnormal{a}},\cS^{\textnormal{s}})$-avoiding column $n/\log^{100}{n}$-repair-system $\cJ(c)= (\cF, \cX, \cI)_{L[c]}$ with $|\cF| = n/(2\log^{1000}{n})$ such that all intercalates in $\cI$ are $T$-disjoint. Since $|\tilde{\cC^{\textnormal{s}}}|\leq |\cC^{\textnormal{s}}| \leq \log^{11}{n}$, there is a set $\tilde{\cC}^{\textnormal{a}}$ of at most $4 \log^{11}{n}$ columns that include $\cC^{\textnormal{s}}$, the mate $c_2$ of each column $c_1\in \tilde{\cC}^{\textnormal{s}}$, and the columns intersecting the intercalates $\cI(c_1,c_2)$. Thus, we can find a mate $c' \in \cF$ for $c$ such that $c'$ and its associated columns $\cX(c')$ are disjoint from all columns in $\tilde{\cC}^{\textnormal{a}} \cup \cC^{\textnormal{a}}$. We then take the associated intercalates $\cI(c')$ of $c'$ in $\cJ(c)$ as the set $\cI(c,c') \subset \cI(\mathfrak{R}^{\textnormal{p}},\mathfrak{C}^{\textnormal{p}},\mathfrak{S}^{\textnormal{p}})$.

\medskip\noindent\textit{Picking mates for the stubborn symbols.}
Finally, suppose we have already chosen $\mathfrak{R}^{\textnormal{p}}$ and $\mathfrak{C}^{\textnormal{p}}$ and all associated intercalates $\cI(z_1,z_2) \subset \cI(\mathfrak{R}^{\textnormal{p}},\mathfrak{C}^{\textnormal{p}},\mathfrak{S}^{\textnormal{p}})$ for $(z_1,z_2)\subseteq \mathfrak{R}^{\textnormal{p}} \cup \mathfrak{C}^{\textnormal{p}}$. Suppose furthermore we have chosen a mate $s_2$ for each $s_1 \in \tilde{\cS}^{\textnormal{s}} \subset \cS^{\textnormal{s}}$ along with the associated intercalates $\cI(s_1, s_2)\subset \cI(\mathfrak{R}^{\textnormal{p}},\mathfrak{C}^{\textnormal{p}},\mathfrak{S}^{\textnormal{p}})$. Let $s \in \cS^{\textnormal{s}} \setminus \tilde{\cS}^{\textnormal{s}}$. Let $\cR^{\textnormal{a}}:= \cR^{\textnormal{s}} \cup \cR^{\textnormal{m}} \cup \cR^{\textnormal{i}}$ as before and $\cC^{\textnormal{a}} := \cC^{\textnormal{s}} \cup \cC^{\textnormal{m}} \cup \cC^{\textnormal{i}}$ where $\cC^{\textnormal{i}}$ is the set of all columns intersecting an intercalate in $\cI(c_1,c_2)$ for $(c_1,c_2) \in \mathfrak{C}^{\textnormal{p}}$. Additionally, let $\cS^{\textnormal{a}}$ be the set of all symbols intersecting an intercalate in $\cI(r_1,r_2)$ for $(r_1,r_2) \in \mathfrak{R}^{\textnormal{p}}$ or in $\cI(c_1,c_2)$ for $(c_1,c_2) \in \mathfrak{C}^{\textnormal{p}}$ and note that $|\cS^{\textnormal{a}}|\leq 4 (|\cR^{\textnormal{s}}| + |\cC^{\textnormal{s}}|) \leq 8\log^{11}{n}$. By part (\ref{claim:symbols}) of \cref{claim:many-T-disjoint-chains}, there is a $(\cR^{\textnormal{a}},\cC^{\textnormal{a}})$-avoiding symbol $n/\log^{100}{n}$-repair-system $\cJ(s)= (\cF, \cX, \cI)_{L[s]}$ with $|\cF| = n/(2\log^{1000}{n})$ such that all intercalates in $\cI$ are $T$-disjoint. Since $|\tilde{\cS^{\textnormal{s}}}|\leq |\cS^{\textnormal{s}}| \leq \log^{11}{n}$, there is a set $\tilde{\cS}^{\textnormal{a}}$ of at most $4 \log^{11}{n}$ symbols that include $\cS^{\textnormal{s}}$, the mate of each symbol $s_1 \in \tilde{\cS}^{\textnormal{s}}$, and the symbols intersecting the intercalates $\cI(s_1,s_2)$. Thus, we can find a mate $s' \in \cF$ for $s$ such that $s'$ and its associated symbols $\cX(s')$ are disjoint from all symbols in $\tilde{\cS}^{\textnormal{a}} \cup \cS^{\textnormal{a}}$. We then take the associated intercalates $\cI(s')$ of $s'$ in $\cJ(s)$ as the set $\cI(s,s') \subset \cI(\mathfrak{R}^{\textnormal{p}},\mathfrak{C}^{\textnormal{p}},\mathfrak{S}^{\textnormal{p}})$.

\medskip\noindent\textit{The chosen set of intercalates is repairing.}
We will now check that the set $\cI(\mathfrak{R}^{\textnormal{p}}, \mathfrak{C}^{\textnormal{p}}, \mathfrak{S}^{\textnormal{p}})$ is indeed $(T, n/\log^{100}{n})$-repairing for $L$. Firstly, at the end of this process, each row, column, and symbol in $\cR^{\textnormal{s}} \cup \cC^{\textnormal{s}} \cup \cS^{\textnormal{s}}$ has a distinct mate which is disjoint from $\cR^{\textnormal{s}} \cup \cC^{\textnormal{s}} \cup \cS^{\textnormal{s}}$ due to our choice of $\tilde{\cR}^{\textnormal{a}},\tilde{\cC}^{\textnormal{a}},\tilde{\cS}^{\textnormal{a}}$. Thus, $(\mathfrak{R}^{\textnormal{p}}, \mathfrak{C}^{\textnormal{p}}, \mathfrak{S}^{\textnormal{p}})$ is a mate system. Part~(\ref{def:mate_repair_system_ints}) of \cref{def:repairing_intercalates} follows from the fact that we only selected intercalates which were $T$-disjoint since we applied \cref{claim:many-T-disjoint-chains}, and from the fact that each mate $z_2$ of $z_1 \in \cR^{\textnormal{s}}\cup\cC^{\textnormal{s}}\cup\cS^{\textnormal{s}}$ was chosen to be a member of $\cF$ from $\cJ(z_1) = (\cF,\cX,\cI)_{L[z_1]}$, where the intercalates in $\cI(z_2)$ intersect $z_2$ by \cref{def:repair_system}. Part~(\ref{def:mate_repair_system_ints_disjoint}) of \cref{def:repairing_intercalates} follows since each time we chose a new mate with its associated intercalates, the intercalates avoided the rows of $\cI(r_1,r_2)$ for the already chosen $(r_1,r_2) \in \mathfrak{R}^{\textnormal{p}}$, the columns of $\cI(c_1,c_2)$ for the already chosen $(c_1,c_2) \in \mathfrak{C}^{\textnormal{p}}$, and the symbols of $\cI(s_1,s_2)$ for the already chosen $(s_1,s_2) \in \mathfrak{S}^{\textnormal{p}}$. Part~(\ref{def:mate_repair_system_ints_vs_rows_cols_symbs}) of \cref{def:repairing_intercalates} holds since each time we chose a new mate with its associated intercalates, the intercalates avoided $\cR^{\textnormal{s}} \cup \cC^{\textnormal{s}} \cup \cS^{\textnormal{s}}$, as well as the members of $\cR^{\textnormal{m}} \cup \cC^{\textnormal{m}} \cup \cS^{\textnormal{m}}$ that had already been chosen; on the other hand, new members of $ \cR^{\textnormal{m}} \cup \cC^{\textnormal{m}} \cup \cS^{\textnormal{m}}$ were chosen to avoid the rows, columns, and symbols of intercalates that have already been selected in $\cI(z_1, z_2)$ for some $(z_1,z_2) \in \mathfrak{R}^{\textnormal{p}} \cup \mathfrak{C}^{\textnormal{p}} \cup \mathfrak{S}^{\textnormal{p}}$.

For part~(\ref{def:mate_repair_system_ints_create_cycles}) of \cref{def:repairing_intercalates}, first note that due to part~(\ref{def:mate_repair_system_cycles_vs_rows_cols_symbs}), switching an intercalate in $\cI(z_1,z_2)$ can never affect any entry in $\cR^{\textnormal{s}} \cup \cC^{\textnormal{s}} \cup \cS^{\textnormal{s}} \cup \cR^{\textnormal{m}} \cup \cC^{\textnormal{m}} \cup \cS^{\textnormal{m}} \setminus \{z_2\}$. This implies that for each $(z_1,z_2) \in \mathfrak{R}^{\textnormal{p}} \cup \mathfrak{C}^{\textnormal{p}} \cup \mathfrak{S}^{\textnormal{p}}$, the entries in row/column/symbol $z_1$ and row/column/symbol $z_2$ are the same in $\textnormal{sw}_{\cI(\mathfrak{R}^{\textnormal{p}},\mathfrak{C}^{\textnormal{p}},\mathfrak{S}^{\textnormal{p}})}(L)$ and in $\textnormal{sw}_{\cI(z_1,z_2)}(L)$. Thus, by part~(\ref{def:repair_system_cycles}) of \cref{def:repair_system}, the permutation $\sigma_{L'}(z_1,z_2)$ in $L':=\textnormal{sw}_{\cI(\mathfrak{R}^{\textnormal{p}},\mathfrak{C}^{\textnormal{p}},\mathfrak{S}^{\textnormal{p}})}(L)$ contains (at least one) even-length cycle of length at most $n/\log^{100}{n}$. Choose one of these short even-length cycles for each $(z_1,z_2) \in \mathfrak{R}^{\textnormal{p}} \cup \mathfrak{C}^{\textnormal{p}} \cup \mathfrak{S}^{\textnormal{p}}$ to construct the set $\cY$. We will show that $\textnormal{sw}_{\cI(\mathfrak{R}^{\textnormal{p}},\mathfrak{C}^{\textnormal{p}},\mathfrak{S}^{\textnormal{p}})}(L)$ together with $(\mathfrak{R}^{\textnormal{p}},\mathfrak{C}^{\textnormal{p}},\mathfrak{S}^{\textnormal{p}})$ is $(\cR^{\textnormal{s}}, \cC^{\textnormal{s}}, \cS^{\textnormal{s}})$-flexible with cycles $\cY$. Part~(\ref{def:even_short_cycles}) of \cref{def:flexible_LS} follows from the preceding discussion. Part~(\ref{def:disjoint_cycles}) of \cref{def:flexible_LS} holds because all rows, columns, and symbols in $\cR^{\textnormal{s}}, \cC^{\textnormal{s}}, \cS^{\textnormal{s}}, \cR^{\textnormal{m}}, \cC^{\textnormal{m}}, \cS^{\textnormal{m}}$ are distinct, and because all column cycles avoid the rows in $\cR^{\textnormal{s}} \cup \cR^{\textnormal{m}}$, and all symbol cycles avoid the rows in $\cR^{\textnormal{s}} \cup \cR^{\textnormal{m}} $ and the columns in $ \cC^{\textnormal{s}} \cup \cC^{\textnormal{m}}$. Similarly, part~(\ref{def:mate_repair_system_cycles_vs_rows_cols_symbs}) of \cref{def:flexible_LS} holds because the row $n/\log^{100}{n}$-repair-system we chose for each row in $\cR^{\textnormal{s}}$ is $(\cC^{\textnormal{a}},\cS^{\textnormal{a}})$-avoiding with $\cC^{\textnormal{a}} \supseteq \cC^{\textnormal{s}}$ and $\cS^{\textnormal{a}} \supseteq \cS^{\textnormal{s}}$ (and similarly for the column, respectively symbol $n/\log^{100}{n}$-repair-system for each column and symbol).

\medskip\noindent\textit{The appropriate set of stable intercalates is a parity generator.}
Recall that $\cI^{\textnormal{stable}}$ is the set of $\cR^{\textnormal{s}} \cup \cC^{\textnormal{s}} \cup \cS^{\textnormal{s}}$-avoiding stable intercalates in $T \cap L$, and that it is a $n/(100\log^{11}{n})$-robust parity generator with respect to $T \cap L, \cR^{\textnormal{f}}, \cC^{\textnormal{f}}, \cS^{\textnormal{f}}$. Let $E$ consist of all entries in cycles $\cY(z_1,z_2)$ for $(z_1,z_2) \in \mathfrak{R}^{\textnormal{p}} \cup \mathfrak{C}^{\textnormal{p}} \cup \mathfrak{S}^{\textnormal{p}}$, and note that since each cycle has length at most $n/\log^{100}{n}$, we have $|E| \leq 6n/\log^{89}{n}$. Then the set $\cI^{\textnormal{s}}\subseteq \cI^{\textnormal{stable}}$ of stable intercalates in $T \cap L = T \cap \textnormal{sw}_{\cI(\mathfrak{R}^{\textnormal{p}},\mathfrak{C}^{\textnormal{p}},\mathfrak{S}^{\textnormal{p}})}(L)$ that do not intersect $E$  is a parity generator with respect to $T \cap \textnormal{sw}_{\cI(\mathfrak{R}^{\textnormal{p}},\mathfrak{C}^{\textnormal{p}}, \mathfrak{S}^{\textnormal{p}})}(L), \cR^{\textnormal{f}}, \cC^{\textnormal{f}}, \cS^{\textnormal{f}}$, as desired.
\end{proof}

It remains to prove \cref{lem:chains}, which guarantees the existence of repair systems. The crucial ingredient of its proof is the following lemma, which shows that with high probability, a uniformly random completion of a Latin rectangle contains an intercalate in a prescribed set of rows and columns. In particular, the guaranteed intercalate always has one row in the Latin rectangle, and for each such potential row, there are (possibly different) prescribed sets of columns that the intercalate should intersect. This lemma will be helpful for finding the intercalates in the repair system, as we will need to split or merge specific cycles, which calls for fine control over the location of the intercalates we look for.

\begin{lemma}
\label{lem:many-intercalates}
Fix a constant $\zeta \in (0,1)$ and let $w \leq (1 - \zeta)n$ and $\cR_0 = [w]$. Suppose $h, d \in [n]$ are such that $h = o(n)$ and $h^4 d^4 = \omega(n^7\log^2{n})$, as well as $d \geq 3w$.
Let $W\in\cL_n^w$ be a $w\times n$ Latin rectangle. Let $\cR'_0 \subseteq [w]$ and $\cR_1 \subseteq [n]\setminus[w]$ be sets of rows each of size at least $h$. Suppose for each $i \in \cR'_0$, $A_i, B_i\subset [n]$ are sets of columns with $A_i \cap B_i = \varnothing$, each of size at least $d$.  Consider a uniformly random completion $\bfL \sim \unif(\cL_n(W))$ of $W$ to a Latin square. With probability at least $1-\exp(-\Omega(h^4d^4/n^6))$, the following holds. There exist $i \in \cR'_0$ and $j \in \cR_1$ such that there is an intercalate intersecting both rows $i$ and $j$ and one column in each of $A_i$ and $B_i$.
\end{lemma}

We now show how to deduce \cref{lem:chains} from \cref{lem:many-intercalates}.

\begin{proof}[Proof of~\cref{lem:chains}]
For every choice of a $1\times n$ Latin rectangle $D$ and forbidden sets $\tilde{\cC},\tilde{\cS}$ of, respectively, columns and symbols, each of size at most $\log^{12}n$, we would like to show the existence of a sufficiently large $(\tilde{\cC},\tilde{\cS})$-avoiding row $n/\log^{100}{n}$-repair-system $(\cF, \cX, \cI)_D$. Roughly speaking, this entails finding a set of rows $\cF$ such that each row $r\in\cF$ has a set $\mc I(r)$ of at most two associated intercalates that can be used to create a short even-length cycle in $\sigma(D,L[r]).$
We will use a union bound over all possible $1\times n$ Latin rectangles $D$ and all possible choices for $\tilde{\cC},\tilde{\cS}$. For now, let us fix some $1\times n$ Latin rectangle $D$, forbidden set of columns $\tilde{\cC}$, and forbidden set of symbols $\tilde{\cS}$.

\medskip\noindent\textit{Strategy overview.} Let $w = n/\log^{200}{n}$ and consider the top $3w$ rows of $\bfL$, that is, $\cR_0 =[w]$, $\cR_1 = [2w]\setminus [w]$, and $\cR_2 = [3w] \setminus[2w]$. Set initially $\mc F=\varnothing$ and $\cX(r):= \varnothing$ for each $r\in \cR_0$. We will refer to the rows in $\cF$ as \emph{matched}. The rows in $\cR_0$ are our candidates for $r\in \mc F$: for each of them consider the cycle structure of $\sigma(D,L[r])$. Let us call a row cycle $(\tilde{\cC},\tilde{\cS})$-\emph{good} if it avoids $\tilde{\cC},\tilde{\cS}.$ We would like to show that we can use intercalate switches to create a short even-length $(\tilde{\cC},\tilde{\cS})$-good cycle in $\sigma(D,L[r])$ for most of the choices of $r\in\cR_0.$

Our strategy is to use intercalates with one row in $\cR_0$ and another row in $\cR_1 \cup \cR_2$ to change the cycle structure of $\sigma(D,L[r])$ for $r \in \cR_0$. For rows $r \in \cR_0$ such that $\sigma(D,L[r])$ does not already contain an even-length $(\tilde{\cC},\tilde{\cS})$-good cycle of length at most $n/\log^{100}{n}$, we distinguish two cases:
\begin{enumerate}
\item \label{case:long_cycle} there is at least one cycle of length at least $q:=n/\log^{101}{n}$, or 
\item \label{case:short_cycles} all cycles have length less than $q$.
\end{enumerate}
We refer to rows $r \in \cR_0$ as type~(\ref{case:long_cycle}) and type~(\ref{case:short_cycles}) depending on which case they fall into. For rows of type~(\ref{case:long_cycle}), in a first round we will \emph{excise} (recall \cref{fact:intercalate-switching}) a short $(\tilde{\cC},\tilde{\cS})$-good cycle from the cycle of length at least $q$ with the help of an intercalate intersecting $r$ and $\cR_1$. If that cycle happens to have odd length, in a second round we will further split it into two parts with the help of an intercalate intersecting $r$ and $\cR_2$ (so at least one part will have even length). For rows of type~(\ref{case:short_cycles}), one round will be sufficient---we will merge two short odd-length $(\tilde{\cC},\tilde{\cS})$-good cycles to form an even-length $(\tilde{\cC},\tilde{\cS})$-good cycle with the help of an intercalate intersecting $\cR_1$. At each step, we will add a row $r$ from $\cR_0$ to $\cF$ and create the corresponding set $\cX(r) \subseteq \cR_1 \cup \cR_2$. We will denote by $\widehat{\cR}_0\subseteq \cR_0$ the rows that require a second round. To begin with, we set $\widehat{\cR}_0:=\varnothing$. 

To show the existence of suitable intercalates as described above, we will use \cref{lem:many-intercalates} in both rounds. The Latin rectangle spanned by $\cR_0$ in the first round, respectively by $\cR_0 \cup \cR_1$ in the second round, will play the role of $W$ in \cref{lem:many-intercalates}. To apply \cref{lem:many-intercalates}, we first need to define the suitable $A^r_1, B^r_1$ for each $r\in \cR_0$ (which will play the role of $A_i, B_i$).

\medskip\noindent\textit{Set-up for round $1$.}
Let $\bfL\sim \unif(\cL_n)$ and let us reveal the entries in rows $\cR_0$, but no other entries. We go through the elements of $\cR_0$ one-by-one, so let $r \in \cR_0$. According to the previous plan, we distinguish a few cases based on how $\sigma(D,L[r])$ looks. If there already exists a short even-length $(\tilde{\cC},\tilde{\cS})$-good cycle, then we add $r$ to $\mc F$ and keep $\cX(r)=\varnothing.$ Otherwise, either $\sigma(D,L[r])$ contains at least one cycle of length at least $q= \frac{n}{\log^{101}n}$ (case (\ref{case:long_cycle})) or all cycles have length less than $q$ (case (\ref{case:short_cycles})).

In case (\ref{case:long_cycle}), even if all the forbidden columns and symbols $\tilde{\cC},\tilde{\cS}$ intersect the chosen long cycle of size at least $q$, by the pigeonhole principle there has to exist a continuous segment of this cycle of length at least $\frac{q-3\log^{12}n}{3\log^{12}n}\geq\frac{n}{\log^{114}n}$ that avoids $\tilde{\cC},\tilde{\cS}$. We select such a segment and select a subsegment of length exactly $\frac{n}{\log^{114}n}$. Let $A_1^r$ denote the first third of this subsegment and $B_1^r$ denote the last third.

In case (\ref{case:short_cycles}), there are no cycles of length at least $q$. That means there are at least $\frac{n}{q}=\log^{101}n$ cycles. As we have only at most $\log^{12}n$ forbidden columns in $\tilde{\cC}$ and at most $\log^{12}n$ forbidden symbols in $\tilde{\cS}$, there exist at least $\log^{101}n-2\log^{12}n\geq \log^{100}n$ cycles that are $(\tilde{\cC},\tilde{\cS})$-good and that together cover at least half of all the entries. As all of these cycles are short (of length at most $q=\frac{n}{\log^{101}n}$), they must all have odd length (otherwise $r$ was already matched). We split these odd-length cycles into two parts of roughly equal total number of entries and denote the union of the cycles in the first part (respectively the second part) by $A_1^r$ (respectively $B_1^r$).

After this, each $r\in \cR_0$ is either already matched (these rows are now in $\cF$) or has been assigned its $A^r_1,B^r_1$. Notice that for each $r\in \cR_0 \setminus \cF$, we always have $|A^r_1|,|B^r_1|\geq d:= \frac{n}{\log^{115}n}.$
If $|\cF|>\frac{n}{\log^{1000}n}$, then considering the so-far-constructed set $\cF$ already proves the lemma. Otherwise, we would like to use \cref{lem:many-intercalates} to find suitable intercalates that could be used to create short even-length $(\tilde{\cC},\tilde{\cS})$-good cycles in $\sigma(D,L[r])$. 

\medskip\noindent\textit{Round $1$.} We say that an intercalate is $(r,1)$-\emph{good} if it intersects each of $r,r',A^r_1,B^r_1$ for some row $r'\in \cR_1$. Let $\cR_0'\subseteq \cR_0 \setminus \cF$ and $\cR_1'\subseteq\cR_1$ be any subsets with $|\cR_0'|, |\cR_1'|\geq h:= \frac{n}{\log^{300}n}$. Then by \cref{lem:many-intercalates} with $\cR_0,\cR_0', \cR'_1$, and the previously constructed $A^r_1,B^r_1$ for rows $r\in \cR_0'$, with probability at least $1-\exp(-\Omega(h^4d^4/n^6))$, there exists an $(r,1)$-good intercalate intersecting a row in $\cR'_1$ and with $r\in \cR'_0$. As $\cR_0',\cR_1'$ were arbitrary subsets, by the union bound this holds for all pairs $\cR_0',\cR_1'$ with probability at least $1-2^{2w}\exp(-\Omega(h^4d^4/n^6))$ (as there are at most $2^w\cdot 2^w$ choices for $\cR_0',\cR_1'$). Let us condition on this event and reveal the outcome of~$\cR_1$.

Let $\cP_1$ be a maximal family of disjoint pairs of rows $(r_0,r_1)$ with $r_0 \in \cR_0 \setminus \cF$ and $r_1 \in \cR_1$ such that there is an $(r_0,1)$-good intercalate intersecting $r_1$. We claim that $|\cP_1|\geq w-|\cF|-h.$
Indeed, suppose this is not the case. Then there exist at least $h=n/\log^{300}n$ unmatched rows in $\cR_0$ and also in $\cR_1.$ So by the previous paragraph, there has to exist an $(r_0,1)$-good intercalate intersecting some $r_1\in \cR_1$ with $(r_0,r_1)$ disjoint from $\cP_1$, thus giving us a contradiction.

Next, for each pair $(r_0,r_1)\in\cP_1$, we add $r_1$ to $\cX(r_0)$. If $r_0$ was of type~(\ref{case:long_cycle}) and the found $(r_0,1)$-good intercalate would \emph{excise} an odd-length $(\tilde{\cC},\tilde{\cS})$-good cycle (see \cref{fact:intercalate-switching} for the effect of switching an intercalate on the cycle structure between two rows) we also add $r_0$ to $\widehat{\cR}_0$ and otherwise add $r_0$ to $\cF$. Now, if we already have enough matched rows (meaning $|\cF|\geq n/\log^{1000}n$), we are done with constructing the repair system---this happens if most of the rows $r\in \cR_0$ were of type~(\ref{case:short_cycles}), or of type~(\ref{case:long_cycle}) but the cycle we obtained in the first step had even length, or if they already contained a short even-length $(\tilde{\cC},\tilde{\cS})$-good  cycle in $\sigma(D,L[r])$ to begin with.

\medskip\noindent\textit{Round $2$.} Otherwise, $|\widehat{\cR}_0|\geq w-n/\log^{1000}n- h \geq w - 2n/\log^{300}{n}$, so we perform a second round where we further split the short odd-length cycles from the first round with $\widehat{\cR}_0$ in place of $\cR_0$. That is, for each $r\in \widehat{\cR}_0$, we know that there exists an intercalate $I_r$ intersecting $r$ and $\cR_1$ such that switching it would\footnote{Note that throughout this proof we actually don't perform any switches, we just keep track of how the cycle structure would change if we did.} create an odd-length $(\tilde{\cC},\tilde{\cS})$-good cycle of length at least$\frac{n}{\log^{115}n}$ and at most $\frac{n}{\log^{114}n}$. Let us divide the columns of this cycle that do not intersect $I_r$ (since we would like to find an intercalate disjoint from $I_r$) into two parts of roughly equal size and denote them by $A_2^r,B_2^r.$ We would again like to use \cref{lem:many-intercalates} to show that most of these cycles can be split with an intercalate intersecting $\cR_2.$ 
We say that an intercalate is $(r,2)$-\emph{good} if it intersects each of $r,r',A^r_2,B^r_2$ for some row $r'\in \cR_2$. So we again consider all pairs of subsets $\cR_0'\subseteq\widehat{\cR}_0$ and $\cR_2'\subseteq\cR_2$ with $|\cR_0'|,|\cR_2'|\geq h$. By the union bound and \cref{lem:many-intercalates}, with probability at least $1-2^{2w}\exp(-\Omega(h^4d^4/n^6))$, for all pairs $\cR_0',\cR_2'$, there exists an $(r,2)$-good intercalate intersecting a row in $\cR_2'$ and with $r \in \cR_0'$.

Similarly to before, let $\cP_2$ be a maximal family of disjoint pairs of rows $(r_0,r_2)$ with $r_i \in \cR_i$ for $i\in\{0,2\}$ such that there is an $(r_0,2)$-good intercalate intersecting $r_2$. By an analogous argument as in the first round, we have $|\cP_2|\geq |\widehat{\cR}_0|-h.$

Next, for each pair $(r_0,r_2)\in\cP_2$, we have its corresponding $(r_0,2)$-good intercalate. This intercalate would split the previously-created short odd-length $(\tilde{\cC},\tilde{\cS})$-good cycle into two cycles, so that one of them has even length. We add $r_2$ to $\cX(r_0)$ and also add $r_0$ to $\cF$.

\medskip\noindent\textit{Failure probability.} For a fixed choice of $D, \tilde{\cC},$ and $\tilde{\cS}$, we invoke \cref{lem:many-intercalates} $2^{2w+1}$ many times. By the union bound, all these applications are successful with probability at least $1-2^{2w+1}\exp(-\Omega(h^4d^4/n^6))$. In all such cases, we construct a set $\cF\subseteq\cR_0$ of $|\cF|\geq n/\log^{1000}n$ rows $r$ with their corresponding $\cX(r)$ for which we can perform at most two intercalate switches such that $\sigma(D,L[r])$ has a small even-length cycle. We also need to take the union bound over all $n!= 2^{\Theta(n\log n)}$ possible $1\times n$ Latin rectangles $D$ as well as over the $\binom{n}{\log^{12}n}^2= 2^{\Theta(\log^{13}n)}$ possibilities for $\tilde{\cC},\tilde{\cS}$. In total, our argument fails with probability at most 
\[2^{1+2w+\Theta(n\log n)+\Theta(\log^{13}n)}\exp(-\Omega(n^2/\log^{1660} n))\leq \exp(-\omega(n^2/\log^{2000} n)). \qedhere\]
\end{proof}

\begin{figure}[!hbt]
\centering
\includegraphics{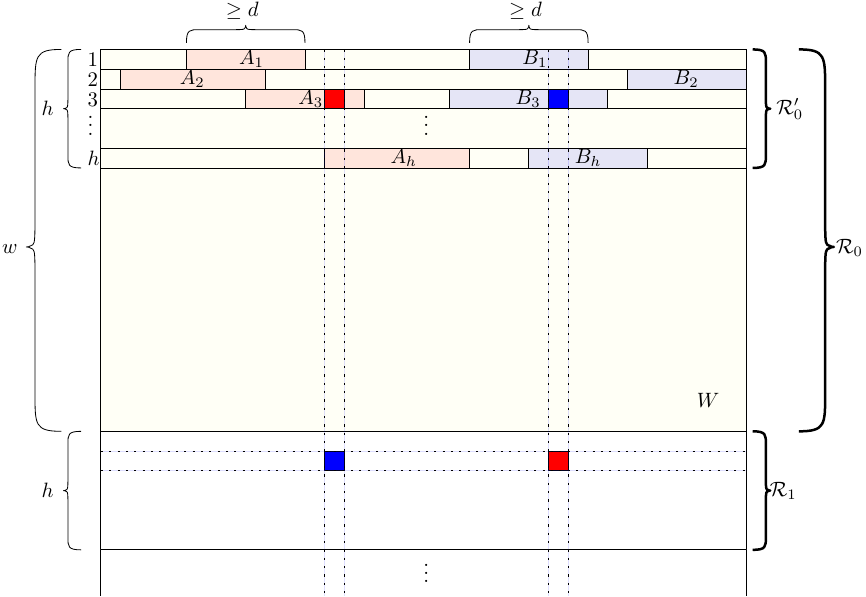}
\caption{The set-up of \cref{lem:many-intercalates}: there is always an intercalate intersecting a row $i\in \cR_0'$, a column in $A_i$, a column in $B_i$, and a row in $\cR_1$. Note that $A_i$ and $B_i$ for $i\in[h]$ are each represented as contiguous segments for illustrative purposes only.}
\label{fig:lemma-6.4}
\end{figure}

Finally, we prove \cref{lem:many-intercalates} with the help of \cref{lem:LS-to-hypergraphs} (the corollary of our approximation lemma). This allows us to work in the binomial random hypergraph model $\mathbb{G}^{(3)}(W,\alpha/(n-w))$, where we deduce the conclusion via a concentration inequality. The approach is based on the proof of \cite[Lemma 10.2]{kwan2025parities}. The statement of \cref{lem:many-intercalates} is illustrated in \cref{fig:lemma-6.4}.

\begin{proof}[Proof of \cref{lem:many-intercalates}]
We can assume without loss of generality that $|\cR_0'|=|\cR_1|=h$. We define an $(i,j,A_i,B_i)$-intercalate to be an intercalate that intersects rows $i$ and $j$ and has one column in each of $A_i$ and $B_i$. Let $\Tbad\subseteq \cL^{\overline{w}}_n(W)$ be the set of Latin rectangles $P$ which do not conflict with $W$ such that there is no $(i,j,A_i,B_i)$-intercalate with $i\in \cR'_0$ and $j\in \cR_1$ in $W \cup P$. Similarly, let $\Ubad\subseteq \cP^W$ be the set of partial Latin rectangles $P$ such that $W \cup P$ has no intercalate of this form.

For an $n\times n$ Latin square $L$ and integers $x, y$ with $1 \leq x\leq y\leq n$, let $L[x,y]$ be the Latin rectangle obtained by restricting to rows $x, x+1, \dots, y$ of $L$. We would like to bound the probability $\Pr[\bfL[w+1,n] \in \Tbad].$ Our strategy will be to consider the corresponding event in the binomial random hypergraph model $\mathbb{G}^{(3)}(W, \alpha/(n-w))$ for some constant $\alpha \in (0,1/2)$. Let $\bfG \sim \mathbb{G}^{(3)}(W, \alpha/(n-w))$ and let $\bfG^*$ be obtained from $\bfG$ by deleting all hyperedges involved in conflicts, as in \cref{lem:LS-to-hypergraphs}.
 
\begin{claim}
\label{clm:bound-ubad}
$\Pr[\bfG^*\in \Ubad]\leq \exp(-\Omega(h^4d^4/n^6)).$    
\end{claim}
Before proving the claim, let us first show how to finish the proof assuming that it holds. Firstly, $\Ubad$ is $(1,\alpha n(n-w))$-inherited from $\Tbad$ because if a partial Latin square does not contain an $(i,j,A_i,B_i)$-intercalate for $i\in\cR'_0, j \in \cR_1$, then a subset of its entries also does not. Thus, by \cref{lem:LS-to-hypergraphs}, we have
\[
\Pr[\bfL[w+1,n]\in \Tbad]\leq \exp(O(n\log^2n))\Pr[\bfG^* \in \Ubad]\leq \exp(-\Omega(h^4d^4/n^6)),
\]
where we used that $h^4d^4 = \omega(n^7\log^2{n})$.
So it is sufficient to prove \cref{clm:bound-ubad}.
\begin{claimproof}
We say that an intercalate is \emph{good} if its two rows $i$ and $j$ satisfy $i\in \cR'_0$ and $j\in \cR_1$, and if it has a column in $A_i$ and a column in $B_i$. 
Let $\bfZ$ be the size of a maximum family of disjoint good intercalates.
We first bound the number $X$ of half-intercalates, which we define as the pairs $(r,c_1,s_1),(r,c_2,s_2) \in W$ such that $r\in \cR'_0$, $c_1 \in A_r, c_2 \in B_r$ and $s_1$ does not appear in column $c_2$ in $W$ and $s_2$ does not appear in column $c_1$ in $W$.

For each $r \in \cR'_0$, there are $|A_r||B_r|$ many pairs of entries $(r,c_1,s_1),(r,c_2,s_2) \in W$ with $c_1\in A_r$ and $c_2 \in B_r$. At most $(|A_r|+|B_r|)w$ of them are not half-intercalates, since for each entry $(r,c_1,s_1)$ with $c_1 \in A_r$, each symbol in column $c_1$ of $W$ can prevent at most one pair of entries containing $(r,c_1,s_1)$ from being a half-intercalate (and similarly for entries $(r,c_2,s_2)$ with $c_2\in B_r$). Thus,
\[X \geq \sum_{r\in\cR'_0} \big(|A_r||B_r| - (|A_r|+ |B_r|)w \big)\geq \sum_{r\in\cR'_0} \big(d^2-2wd \big) \geq h d^2/3,\]
where we used that $d/3\geq w$ and that $|A_r|,|B_r| \geq d$. We then have
\begin{align*}
\E[\bfZ]&\geq X h (\alpha/(n-w))^2(1-\alpha/(n-w))^{6n-6}-Xnh^2(\alpha/(n-w))^4-Xh^2(\alpha/(n-w))^3\\
& = \Omega\left( \frac{h^2d^2}{n^2} \right).
\end{align*}
The first term in this lower bound on $\mathbb{E}[\mathbf{Z}]$ bounds from below the expected number of good intercalates in $W \cup \bfG^*$ (the multiplicative term $(1-\alpha/(n-w))^{6n-6}$ ensures the two entries in $\cR_1$ are not involved in conflicts). The second term bounds from above the number of pairs of intercalates intersecting in an entry inside $\cR'_0$ (since for each half-intercalate, we have $n$ choices for the other column of the second intercalate and $h^2$ choices for the rows in $\cR_1$, see \cref{fig:a}). The third term bounds from above the number of pairs of intercalates intersecting in an entry inside $\cR_1$ (since for each half-intercalate, we have $h$ choices for the common row in $\cR_1$ and $h$ choices for the other row of the second intercalate; the other column of the second intercalate is then determined uniquely by these choices due to the Latin property of $W$, see \cref{fig:b}). 
\begin{figure}[!hbt]
\label{fig:intersecting_intercalates}
    \centering
    \begin{subfigure}{.5\textwidth}
      \centering
      \includegraphics{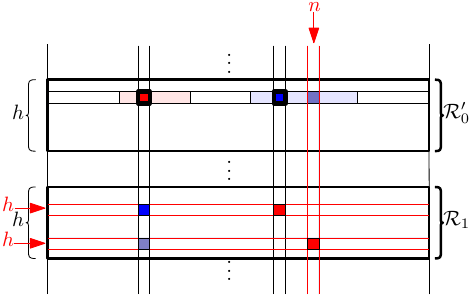}
      \caption{Two intercalates intersecting inside $\cR'_0$.}
      \label{fig:a}
    \end{subfigure}
    ~ 
    \begin{subfigure}{.5\textwidth}
      \centering
      \includegraphics{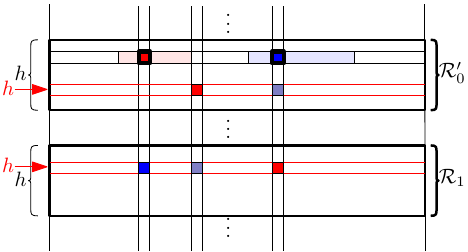}
      \caption{Two intercalates intersecting inside $\cR_1.$}
      \label{fig:b}
    \end{subfigure}\\
    \caption{The two different ways in which a pair of intercalates of interest can intersect.}
\end{figure}

Because we are considering families of disjoint intercalates, the presence of any particular hyperedge in $\bfG$ can affect $\bfZ$ by at most 3. The total number of potential hyperedges is $n(n-|\cR_0|)^2\leq n^3$. Thus, by \cref{freedman_inequality}, we have $\Pr[\bfZ\geq1]\geq 1-\exp(-\Omega(h^4d^4/n^6)).$
\end{claimproof}
Hence, $\Pr[\bfL[w+1,n]\in \Tbad]\leq \exp(-\Omega(h^4d^4/n^6))$, as desired.
\end{proof}

\bibliographystyle{amsplain_initials_nobysame_nomr}
\bibliography{bibliography}

\providecommand{\bysame}{\leavevmode\hbox to3em{\hrulefill}\thinspace}
\providecommand{\MR}{\relax\ifhmode\unskip\space\fi MR }
\providecommand{\MRhref}[2]{%
  \href{http://www.ams.org/mathscinet-getitem?mr=#1}{#2}
}
\providecommand{\href}[2]{#2}
\begin{thebibliography}{10}

\bibitem{cameron1976parallelisms}
P.~J. Cameron, \emph{Parallelisms of complete designs}, London Mathematical Society Lecture Note Series, vol. No. 23, Cambridge University Press, Cambridge-New York-Melbourne, 1976.

\bibitem{Cam01}
P.~J. Cameron, \emph{{BCC} problem list}, \url{https://webspace.maths.qmul.ac.uk/p.j.cameron/bcc/allprobs.pdf}, 2001.

\bibitem{CW16}
N.~J. Cavenagh and I.~M. Wanless, \emph{There are asymptotically the same number of {L}atin squares of each parity}, Bull. Aust. Math. Soc. \textbf{94} (2016), no.~2, 187--194.

\bibitem{MR2239987}
T.~M. Cover and J.~A. Thomas, \emph{Elements of information theory}, second ed., Wiley-Interscience [John Wiley \& Sons], Hoboken, NJ, 2006.

\bibitem{Ego81}
G.~P. Egorychev, \emph{The solution of van der {W}aerden's problem for permanents}, Adv. in Math. \textbf{42} (1981), no.~3, 299--305.

\bibitem{Fal81}
D.~I. Falikman, \emph{Proof of the van der {W}aerden conjecture on the permanent of a doubly stochastic matrix}, Mat. Zametki \textbf{29} (1981), no.~6, 931--938, 957.

\bibitem{Fre75}
D.~A. Freedman, \emph{On tail probabilities for martingales}, Ann. Probability \textbf{3} (1975), 100--118.

\bibitem{FM19}
B.~Friedman and S.~McGuinness, \emph{The {A}lon-{T}arsi conjecture: a perspective on the main results}, Discrete Math. \textbf{342} (2019), no.~8, 2234--2253.

\bibitem{HJ96}
R.~H\"aggkvist and J.~C.~M. Janssen, \emph{All-even {L}atin squares}, Proceedings of the 6th {C}onference on {F}ormal {P}ower {S}eries and {A}lgebraic {C}ombinatorics ({N}ew {B}runswick, {NJ}, 1994), vol. 157, 1996, pp.~199--206.

\bibitem{janson2011random}
S.~Janson, T.~{\L}uczak, and A.~Rucinski, \emph{Random {G}raphs}, Wiley-Interscience Series in Discrete Mathematics and Optimization, Wiley-Interscience, New York, 2000.

\bibitem{Jan95}
J.~C.~M. Janssen, \emph{On even and odd {L}atin squares}, J. Combin. Theory Ser. A \textbf{69} (1995), no.~1, 173--181.

\bibitem{KD15}
A.~D. Keedwell and J.~D\'{e}nes, \emph{Latin squares and their applications}, second ed., Elsevier/North-Holland, Amsterdam, 2015, With a foreword to the previous edition by Paul Erd\H{o}s.

\bibitem{kwan2020almost}
M.~Kwan, \emph{Almost all {S}teiner triple systems have perfect matchings}, Proc. Lond. Math. Soc. (3) \textbf{121} (2020), no.~6, 1468--1495.

\bibitem{kwan2025parities}
M.~Kwan, K.~Petrova, and M.~Sawhney, \emph{Parities in random {L}atin squares}, arXiv preprint arXiv:2509.13125 (2025).

\bibitem{KSS21}
M.~Kwan, A.~Sah, and M.~Sawhney, \emph{Note on random {L}atin squares and the triangle removal process}, arXiv preprint arXiv:2109.15201 (2021).

\bibitem{KSSS22}
M.~Kwan, A.~Sah, and M.~Sawhney, \emph{Large deviations in random {L}atin squares}, Bull. Lond. Math. Soc. \textbf{54} (2022), no.~4, 1420--1438.

\bibitem{McD89}
C.~McDiarmid, \emph{On the method of bounded differences}, Surveys in {C}ombinatorics, 1989 ({N}orwich, 1989), London Math. Soc. Lecture Note Ser., vol. 141, Cambridge Univ. Press, Cambridge, 1989, pp.~148--188.

\bibitem{samotij2026entropy}
W.~Samotij, \emph{Entropy methods in combinatorics}, International Congress of Mathematicians 2026, SIAM, 2026, pp.~253--273.

\bibitem{warnke2016method}
L.~Warnke, \emph{On the method of typical bounded differences}, Combinatorics, Probability and Computing \textbf{25} (2016), no.~2, 269--299.

\bibitem{Zap96}
P.~Zappa, \emph{Triplets of {L}atin squares}, Boll. Un. Mat. Ital. A (7) \textbf{10} (1996), no.~1, 63--69.

\end{thebibliography}

\end{document}